\documentclass[a4paper]{article}
\usepackage{latexsym}
\usepackage{amsmath}
\usepackage{amssymb}
\usepackage{graphicx}
\usepackage{hhline}
\usepackage{ascmac}
\usepackage{mathrsfs}
\usepackage{cases}
\usepackage{bm}
\usepackage{amsthm}
\usepackage{comment}
\usepackage{authblk}
\usepackage{here}
\usepackage{cite}
\usepackage{mathtools}

\MakeRobust{\eqref}

\theoremstyle{definition}
\newtheorem{definition}{\bf Definition}[section]
\newtheorem{theorem}{\bf Theorem}[section]
\newtheorem{lemma}{\bf Lemma}[section]

\newtheorem{corollary}{\bf Corollary}[section]
\newtheorem{remark}{\bf Remark}[section]
\newtheorem*{Proof}{Proof}  
\renewenvironment{proof}{\begin{Proof}}{\end{Proof}}

\def\diag{\mathop{\rm diag}}

\providecommand{\keywords}[1]
{
  \small	
  \textbf{Key words: } #1
}
\providecommand{\subclass}[1]
{
  \small	
  \textbf{Mathematics Subject Classification (2010): } #1
}

\title{A second-order accurate structure-preserving scheme for the Cahn--Hilliard equation with a dynamic boundary condition}

\author[1]{Makoto Okumura\thanks{okumura@cas.cmc.osaka-u.ac.jp}}
\author[2]{Takeshi Fukao\thanks{fukao@kyokyo-u.ac.jp}}
\author[3]{Daisuke Furihata\thanks{furihata@cmc.osaka-u.ac.jp}}
\author[4]{Shuji Yoshikawa\thanks{yoshikawa@oita-u.ac.jp}}

\affil[1]{Department of Pure and Applied Mathematics, Graduate School of Information Science and Technology, Osaka University, 1-5 Yamadaoka, Suita, Osaka, 565-0871, Japan}
\affil[2]{Department of Mathematics, Faculty of Education, Kyoto University of Education, 1 Fujinomori, Fukakusa, Fushimi-ku, Kyoto, 612-8522, Japan}
\affil[3]{Cybermedia Center, Osaka University, 1-32 Machikaneyama, Toyonaka, Osaka, 560-0043, Japan}
\affil[4]{Division of Mathematical Sciences, Faculty of Science and Technology, Oita University, 700 Dannoharu, Oita, 870-1192, Japan}

\date{}

\begin{document}
\maketitle
\begin{abstract}
We propose a structure-preserving finite difference scheme for the Cahn--Hilliard equation with a dynamic boundary condition using the discrete variational derivative method (DVDM)\cite{B}. 
In this approach, it is important and essential how to discretize the energy which characterizes the equation. 
By modifying the conventional manner and using an appropriate summation-by-parts formula, we can use a standard central difference operator as an approximation of an outward normal derivative on the discrete boundary condition of the scheme. 
We show that our proposed scheme is second-order accurate in space, although the previous structure-preserving scheme by Fukao--Yoshikawa--Wada \cite{A} is first-order accurate in space. 
Also, we show the stability, the existence, and the uniqueness of the solution for the proposed scheme. 
Computation examples demonstrate the effectiveness of the proposed scheme. 
Especially through computation examples, we confirm that numerical solutions can be stably obtained by our proposed scheme.
\end{abstract}
\\
\subclass{65M06, 65M12}

\section{Introduction}
Let $L>0$ be the length of the one-dimensional material. 
In this paper, we study the following Cahn--Hilliard equation \cite{Can}: 
\begin{numcases}
{ } 
\partial_{t}u = \partial_{x}^{2}p & in $(0,L) \times (0,\infty)$, \label{Eq_i}\\
p = -\gamma \partial_{x}^{2}u + F^{\prime}(u) & in $(0,L) \times (0,\infty)$, \label{Eq_ii}
\end{numcases}
under the dynamic boundary condition and the homogeneous Neumann boundary condition: 
\begin{numcases}
{ } 
\partial_{t}u(0,t) = \left. \partial_{x}u(x,t)\right|_{x=0} & in $(0,\infty)$, \label{BCi}\\
\partial_{t}u(L,t) = -\left. \partial_{x}u(x,t)\right|_{x=L} & in $(0,\infty)$, \label{BCi'}\\
\left. \partial_{x}p(x,t) \right|_{x=0} = \left. \partial_{x}p(x,t) \right|_{x=L} = 0 & in $(0,\infty)$. \label{BCii}
\end{numcases}
The unknown functions $u$: $[0,L] \times [0,\infty) \to \mathbb{R}$ and $p$: $[0,L] \times [0,\infty) \to \mathbb{R}$ are the order parameter and the chemical potential, respectively. 
Moreover, $\gamma$ is a positive constant. 
Furthermore, $F$: $\mathbb{R} \to \mathbb{R}$ is a given potential, and $F'$ is its derivative. 
For example, $F$ can be a double-well potential, i.e., 
\begin{equation*}
F(s) = \frac{q}{4}s^{4} - \frac{r}{2}s^{2} \quad {\rm for\ all\ } s \in \mathbb{R}, 
\end{equation*}
where $q$ and $r$ are positive constants. 
Throughout this article, we assume that the potential $F$ is bounded from below. 
Let us define the ``local energy" $G$ and the ``global energy" $J$, which characterize the equations \eqref{Eq_i}--\eqref{Eq_ii}, as follows: 
\begin{align}
G(u, \partial_{x}u) & := \frac{\gamma}{2}\left| \partial_{x}u \right|^{2} + F(u), \label{G}\\
J(u) & := \int_{0}^{L}G(u,\partial_{x}u)dx. \label{J}
\end{align}
We remark that the above words ``local energy" and ``global energy" are ones for space, not for time, and that this ``global" is different from the one of the words ``global boundedness" and ``global existence," which appear later. 
Also, let us define the ``mass" $M$ as follows: \vspace{-1mm}
\begin{equation}
M(u) := \int_{0}^{L}udx. \label{Mass} \vspace{-1mm}
\end{equation}
Then, the solution of the equations \eqref{Eq_i}--\eqref{Eq_ii} satisfies the following properties: \vspace{-1mm}
\begin{gather}
\frac{d}{dt}M(u(t)) = 0, \label{MC}\\[-1pt]
\frac{d}{dt}J(u(t)) = - \gamma |\partial_{t}u(0,t)|^{2} - \gamma |\partial_{t} u(L,t)|^{2} - \int_{0}^{L}|\partial_{x}p(x,t)|^{2}dx \leq 0, \label{D}
\end{gather}
under the boundary conditions \eqref{BCi}--\eqref{BCii}. 
In this paper, we design a structure-preserving finite difference scheme for \eqref{Eq_i}--\eqref{BCii} based on the discrete variational derivative method (DVDM) proposed by Furihata and Matsuo \cite{B}. 
Throughout this article, the term ``structure-preserving" means that the scheme inherits the conservative property such as \eqref{MC} or the dissipative property such as \eqref{D}. 
In \cite{C,Yos3}, Yoshikawa has mentioned that the merits of the structure-preserving scheme are that we often obtain the stability of numerical solutions automatically and that various strategies for the continuous case such as the energy method can be applied to the scheme similarly. 
Actually, Yoshikawa and co-authors have applied the energy method to show the existence and uniqueness of the solution and the error estimate for the numerical scheme (see \cite{A,Yan,Yos1,C,Yos3}). 

From a mathematical point of view, the problem \eqref{Eq_i}--\eqref{BCii} with initial conditions has been studied in \cite{Che3,Chi,Col_Fuk,Col_Gil_Spe1,Col_Gil_Spe2,Col_Gil_Spe3,Gal,Gil_Mir_Sch,Gil,Mir_Zel1,Mir_Zel2,Pru,Rac,Wu}. 
First, in the case of $F(s) = (q/4)s^{4} - (r/2)s^{2}$, Racke and Zheng \cite{Rac} have proved the global existence and uniqueness of the solution to the problem, and Pr\"{u}ss et al.\ \cite{Pru} have obtained the result on the maximal $L^p$-regularity and proved the existence of a global attractor. 
Also, Wu and Zheng \cite{Wu} and Chill et al.\ \cite{Chi} have proved the convergence of the solution of the problem to an equilibrium as time goes to infinity. 
Moreover, various results of the existence, uniqueness, and regularity of the solution, the existence of a global attractor, the convergence to steady states, and the optimal control have been obtained under more general assumptions on the nonlinearities in \cite{Che3,Col_Fuk,Col_Gil_Spe1,Col_Gil_Spe2,Col_Gil_Spe3,Gil_Mir_Sch,Gil,Mir_Zel1,Mir_Zel2}. 
Especially, Gal \cite{Gal} has compared the problem under other dynamic boundary conditions with that under our target dynamic boundary conditions. 
Here, we remark that in these papers, the problem is considered in the multi-dimensional case, where the boundary conditions \eqref{BCi} and \eqref{BCi'} include the Laplace--Beltrami operator, which plays the role of diffusion on the boundary. 

From a numerical point of view, there are some numerical studies of the Cahn--Hilliard equation with dynamic boundary conditions (see, for example, \cite{Che1,Che2,Che3,Isr,Nab,A}). 
In \cite{Che1,Isr}, Cherfils et al.\ and Israel et al.\ have considered the finite element space semi-discretizations of the problem in the two-dimensional or three-dimensional case and proved the error estimate. 
Moreover, the unconditional stability of fully discrete schemes based on the backward Euler scheme for the time discretization, and the convergence of the solution to a steady-state have been obtained. 
See also \cite{Che3} for other numerical results. 
Besides, Cherfils and Petcu have obtained the results of the problem with other dynamic boundary conditions by a finite element approach \cite{Che2}. 
In addition, Nabet has performed an interesting analysis for the problem in the two-dimensional case by using the finite-volume method and proved the convergence of the numerical solution \cite{Nab}. 
She also has given the error estimate in \cite{Nab2}.
More specifically, she has proved the first-order convergence in the sense of the $H^1$-norms. 
%However, the result of the error estimate has not been obtained in \cite{Nab} (she has given numerical error estimates). 
In \cite{A}, Fukao et al.\ have already proposed a structure-preserving scheme for \eqref{Eq_i}--\eqref{BCii} based on DVDM in the one-dimensional case. 
We remark that they use a forward difference operator as an approximation of an outward normal derivative on the discrete boundary condition of their scheme and that their scheme is first-order accurate in space. 
In DVDM, it is essential how to discretize the energy which characterizes the equation. 
Modifying the conventional manner and using an appropriate summation-by-parts formula, we can use a standard central difference operator as an approximation of an outward normal derivative on the boundary. 
Moreover, we show that our proposed scheme is second-order in space. 

The rest of this paper proceeds as follows. 
In Section 2, we propose a structure-preserving scheme for \eqref{Eq_i}--\eqref{BCii}, whose solution satisfies the discrete version of the conservative property \eqref{MC} and the dissipative property \eqref{D}. 
In Section 3, we prove that the solution of the proposed scheme satisfies the global boundedness.  
In Section 4, we prove that the scheme has a unique solution under a specific condition. 
In Section 5, we prove the error estimate for the scheme. 
In Section 6, we show that computation examples demonstrate the effectiveness of the scheme. 
In Appendices, we prove some lemmas we have used in this paper and mention the suggestion on comparison of long-time behaviors between the dynamic boundary condition and the Neumann boundary one. 

\section{Proposed scheme}
In this section, we propose a structure-preserving scheme for \eqref{Eq_i}--\eqref{BCii} and show that it has two properties corresponding to \eqref{MC} and \eqref{D}. 

\subsection{Preparation}
%Let $K$ and $N$ be natural numbers. 
%Also, define the space mesh size $\Delta x$ and the time mesh size $\Delta t$ by $\Delta x := L/K$ and $\Delta t := T/N$, respectively. 
%We define $U_{k}^{(n)} \ (k = -1, 0, 1, \ldots, K, K+1, n=0,1,\ldots, N)$ to be the approximation to $u(x,t)$ at location $x=k\Delta x$ and time $t = n\Delta t$. 
Let $K$ be a natural number. 
We define $U_{k}^{(n)} \ (k = -1, 0, 1, \ldots, K, K+1, n=0,1,\ldots)$ to be the approximation to $u(x,t)$ at location $x=k\Delta x$ and time $t = n\Delta t$, where $\Delta x$ is a space mesh size, i.e., $\Delta x := L/K$, and $\Delta t$ is a time mesh size. 
They are also written in vector as $\bm{U}^{(n)} := (U_{0}^{(n)},\ldots,U_{K}^{(n)})^{\top}$ or $\bm{U}^{(n)} := (U_{-1}^{(n)}, U_{0}^{(n)},\ldots,U_{K}^{(n)},U_{K+1}^{(n)})^{\top}$. 
Guess the meaning of $\bm{U}^{(n)}$ from the context. 
$U_{-1}^{(n)}$ and $U_{K+1}^{(n)}$ are artificial quantities and determined by the imposed discrete boundary condition. 
Let us define the difference operators $\delta_{k}^{+}$, $\delta_{k}^{-}$, $\delta_{k}^{\langle 1 \rangle}$, and $\delta_{k}^{\langle 2 \rangle}$ concerning subscript $k$ by  
\begin{align*}
\delta_{k}^{+}f_{k} := \frac{f_{k+1} - f_{k}}{\Delta x}, & & & \delta_{k}^{-}f_{k} := \frac{f_{k} - f_{k-1}}{\Delta x}, \\[-1pt]
\delta_{k}^{\langle 1 \rangle}f_{k} := \frac{f_{k+1} - f_{k-1}}{2\Delta x}, & & & \delta_{k}^{\langle 2 \rangle}f_{k} := \frac{f_{k+1} - 2f_{k} + f_{k-1}}{(\Delta x)^2} \quad (k = 0, 1, \ldots, K) 
\end{align*}
for all $\{f_{k}\}_{k=-1}^{K+1} \! \in \! \mathbb{R}^{K+3}$. 
Similarly, we define the difference operator $\delta_{n}^{+}$ corresponding superscript $(n)$ by 
\begin{equation*}
\delta_{n}^{+}f^{(n)} := \frac{ f^{(n+1)} - f^{(n)} }{\Delta t}. 
\end{equation*}
As a discretization of the integral, we adopt the summation operator $\sum _ {k=0}^{K}{}^{\prime \prime}$: $\mathbb{R}^{K+1+2s} \to \mathbb{R}$ defined by 
\begin{equation*}
\sum _ {k=0}^{K}{}^{\prime \prime} f_{k} := \frac{1}{2}f_{0} +\sum_{k=1}^{K-1}f_{k} +\frac{1}{2}f_{K} \quad \mbox{for\ all\ } \{f_{k}\}_{k=-s}^{K+s} \in \mathbb{R}^{K+1+2s},\ \mbox{where\ } s = 0,1. 
\end{equation*}
For later use, we define the difference quotient. 
Let $\Omega$ be a domain in $\mathbb{R}$. 
For a function $F \in C^{1}(\Omega)$ and $\xi, \eta \in \Omega$, the difference quotient $d F/d (\xi, \eta)$ of $F$ at $(\xi, \eta)$ is defined as follows: 
\begin{numcases}
    {\frac{dF}{d (\xi, \eta)} :=}
      \frac{F(\xi) - F(\eta)}{\xi - \eta} & ($\xi \neq \eta$), \nonumber\\[-1pt]
     F'(\eta) & ($\xi = \eta$). \nonumber
\end{numcases}
Here, let us define two discrete local energies $G_{\rm d}^{\pm}$: $\mathbb{R}^{K+3} \to \mathbb{R}^{K+1}$ by 
\begin{align*}
G _ {{\rm d},k}^{+}(\bm{U}) & := \frac{\gamma}{2}(\delta_{k}^{+}U_{k})^{2} + F\left( U_{k}\right) \quad (k=0,\ldots,K), \\[-1pt]
G _ {{\rm d},k}^{-}(\bm{U}) & := \frac{\gamma}{2}(\delta_{k}^{-}U_{k})^{2} + F\left( U_{k}\right) \quad (k=0,\ldots,K),  
\end{align*} 
for all $\bm{U} \in \mathbb{R}^{K+3}$. 
Note that $G _ {{\rm d},k}^{\pm}(\bm{U})$ are elements of the vectors $G _ {\rm d}^{\pm}(\bm{U})$, respectively. 
Furthermore, we define the discrete global energy $J_{\rm d}$: $\mathbb{R}^{K+3} \to \mathbb{R}$ as follows: 
\begin{equation}
J_{\rm d}(\bm{U}) := \frac{1}{2}\left\{\sum_{k=0}^{K-1}G_{{\rm d}, k}^{+}(\bm{U})\Delta x + \sum_{k=1}^{K}G_{{\rm d}, k}^{-}(\bm{U})\Delta x \right\}. \label{Jd}
\end{equation}
Also, we define a discrete mass $M_{\rm d}$: $\mathbb{R}^{K+1+2s} \to \mathbb{R}$ by 
\begin{equation*}
M_{\rm d}(\bm{U}) := \sum _ {k=0}^{K}{}^{\prime \prime}U_{k}\Delta x \quad \mbox{for\ all\ } \bm{U} \in \mathbb{R}^{K+1+2s},\ \mbox{where\ } s=0,1. 
\end{equation*}
\begin{remark}
We remark that we have constructed a structure-preserving scheme for \eqref{Eq_i}--\eqref{BCii}, which we have used a standard central difference operator as an approximation of an outward normal derivative on the boundary conditions by adopting the above discrete global energy $J_{\rm d}$ and using another summation-by-parts formula (Corollary \ref{col:2.1}).
\end{remark}
From the idea in DVDM\cite{B}, we take a discrete variation to derive a structure-preserving scheme for \eqref{Eq_i}--\eqref{BCii}. 
That is, we calculate the difference $J_{\rm d}(\bm{U}) - J_{\rm d}(\bm{V})$ for all $\bm{U}, \bm{V} \in \mathbb{R}^{K+3}$. 
For the purpose, we use the following lemmas. 
All the proofs can be obtained by direct calculation and here omitted. 

\begin{lemma}\label{lem:2.1}
The following equality holds: 
\begin{equation*}
\frac{1}{2}\left(\sum _ {k=0}^{K-1}f_{k}\Delta x + \sum _ {k=1}^{K}f_{k}\Delta x\right) = \sum _ {k=0}^{K}{}^{\prime \prime}f_{k}\Delta x \quad \mbox{for\ all\ } \{f_{k}\}_{k=0}^{K} \in \mathbb{R}^{K+1}. 
\end{equation*}
\end{lemma}

\begin{lemma}[Second-order summation by parts formula]\label{lem:2.2}
Let us denote $f_{K} - f_{0}$ by $[f_{k}]_{0}^{K}$. 
The following second-order summation by parts formulas hold: 
\begin{align*}
\sum_{k=0}^{K-1}(\delta_{k}^{+}f_{k})(\delta_{k}^{+}g_{k})\Delta x & = -\sum_{k=0}^{K-1}(\delta_{k}^{\langle 2 \rangle}f_{k})g_{k}\Delta x + [(\delta_{k}^{-}f_{k})g_{k}]_{0}^{K}, \\
\sum_{k=1}^{K}(\delta_{k}^{-}f_{k})(\delta_{k}^{-}g_{k})\Delta x & = -\sum_{k=1}^{K}(\delta_{k}^{\langle 2 \rangle}f_{k})g_{k}\Delta x + [(\delta_{k}^{+}f_{k})g_{k}]_{0}^{K}
\end{align*}
for all $\{f_{k}\}_{k=-1}^{K+1}, \{g_{k}\}_{k=-1}^{K+1} \in \mathbb{R}^{K+3}$. 
\end{lemma}

\begin{corollary}[Second-order summation by parts formula]\label{col:2.1}
The following second-order summation by parts formula holds: 
\begin{equation*}
\sum_{k=0}^{K-1}\left(\delta_{k}^{+}f_{k}\right)\left(\delta_{k}^{+}g_{k}\right)\Delta x = -\sum_{k=0}^{K}{}^{\prime\prime}\left(\delta_{k}^{\langle 2 \rangle}f_{k}\right)g_{k}\Delta x + \left[\left(\delta_{k}^{\langle 1 \rangle}f_{k}\right) g_{k}\right]_{0}^{K} \quad \mbox{for\ all\ } \{f_{k}\}_{k=-1}^{K+1}, \{g_{k}\}_{k=-1}^{K+1} \in \mathbb{R}^{K+3}. 
\end{equation*}
\end{corollary}
\noindent 
Using Lemma \ref{lem:2.1}, we obtain the following lemma. 
\begin{lemma}\label{lem:2.5} The definition \eqref{Jd} of $J_{\rm d}$ is rewritten as follows: 
\begin{equation*}
J_{\rm d}(\bm{U}) = \sum_{k=0}^{K-1}\frac{\gamma}{2}\left(\delta_{k}^{+}U_{k}\right)^{2}\Delta x + \sum_{k=0}^{K}{}^{\prime\prime} F\left( U_{k}\right) \Delta x \quad \mbox{for\ all\ } \bm{U} \in \mathbb{R}^{K+3}. 
\end{equation*}
\end{lemma}\noindent
By using Lemma \ref{lem:2.5} and Corollary \ref{col:2.1}, we have the following lemma: 
\begin{lemma}\label{lem:2.6}
The following equality holds: 
\begin{equation}
J_{\rm d}(\bm{U}) \! - \! J_{\rm d}(\bm{V}) 
	= \sum _ {k=0}^{K}\!{}^{\prime \prime}\!\!\left\{\! -\gamma\delta_{k}^{\langle 2 \rangle}\!\!\left(\! \frac{U_{k} \! + \! V_{k}}{2}\!\right) \! + \! \frac{d F}{d (U_{k},V_{k})} \!\right\}\!\left( U_{k} \! - \! V_{k}\right) \! \Delta x + \! \left[\gamma\!\left\{\!\delta_{k}^{\langle 1 \rangle}\!\!\left(\! \frac{U_{k} \! + \! V_{k}}{2} \!\right)\!\!\right\}\!\left( U_{k} \! - \! V_{k} \right)\right]_{0}^{K} \label{disvari} \vspace{-1mm}
\end{equation}
for all $\bm{U}, \bm{V} \in \mathbb{R}^{K+3}$. 
\end{lemma}
\noindent
This equality \eqref{disvari} is essential for the discrete energy dissipation (Theorem \ref{thm:2.1}). 

\subsection{Proposed scheme}
The concrete form of our scheme for \eqref{Eq_i}--\eqref{BCii} is 
\begin{gather}
\frac{U_ {k}^{(n+1)} - U_ {k}^{(n)}}{\Delta t} = \delta _ {k}^{\langle 2 \rangle}P_ {k}^{(n)} \quad (k=0,\ldots, K,\ n=0,1, \ldots), \label{dCH1}\\
P_ {k}^{(n)} = -\gamma \delta_{k}^{\langle 2 \rangle}\!\left( \frac{U_{k}^{(n+1)} + U_{k}^{(n)}}{2}\right) + \frac{dF}{d(U_{k}^{(n+1)},U_{k}^{(n)})} \quad (k=0,\ldots, K,\ n=0,1, \ldots), \label{dCH2}\\
\frac{U_ {0}^{( n+1 )} - U_ {0}^{( n )}}{\Delta t} = \left. \delta _ {k}^{\langle 1 \rangle}\left( \frac{U_ {k}^{( n+1 )} + U_ {k}^{( n )}}{2} \right) \right|_{k=0} \quad (n = 0,1, \ldots), \label{dDBC}\\
\frac{U_ {K}^{( n+1 )} - U_ {K}^{( n )}}{\Delta t} = - \left. \delta _ {k}^{\langle 1 \rangle}\left( \frac{U_ {k}^{( n+1 )} + U_ {k}^{( n )}}{2} \right) \right|_{k=K} \quad (n = 0,1, \ldots), \label{dDBC'}\\
\delta _ {k}^{\langle 1 \rangle}P_ {k}^{(n)} = 0 \quad (k=0,K,\ n=0,1, \ldots). \label{dNBC}
\end{gather}
\begin{remark}
In previous result \cite{A}, Fukao et al.\ constructed another structure-preserving scheme. 
They used a forward difference operator as an approximation of an outward normal derivative on the boundary conditions. 
On the other hand, we have constructed a structure-preserving scheme in which we used a central difference operator as an approximation of an outward normal derivative. 
\end{remark} 
\noindent 
Then, the proposed scheme \eqref{dCH1}--\eqref{dNBC} has the following property corresponding to \eqref{D}, i.e., 

\begin{theorem}\label{thm:2.1}
The solution of the scheme \eqref{dCH1}--\eqref{dNBC} satisfies the following inequality: 
\begin{equation}
\delta_{n}^{+}J_{\rm d}(\bm{U}^{(n)}) = - \gamma \left| \delta_{n}^{+}U_{0}^{(n)}\right|^{2} - \gamma \left| \delta_{n}^{+}U_{K}^{(n)}\right|^{2} - \sum_{k=0}^{K-1}\left|\delta_{k}^{+}P_{k}^{(n)}\right|^{2}\Delta x \leq 0 \quad (n = 0,1,\ldots).  \label{dD} \vspace{-1mm}
\end{equation} 
\end{theorem}

\begin{proof}
Using Corollary \ref{col:2.1}, Lemma \ref{lem:2.6}, \eqref{dCH1}--\eqref{dNBC}, we have 
\begin{align*}
\delta_{n}^{+}J_{\rm d}(\bm{U}^{(n)}) 
	& = \frac{ J_{\rm d}\left(\bm{U}^{(n+1)}\right) - J_{\rm d}\left(\bm{U}^{(n)}\right) }{\Delta t} \\[-1pt]
	& = \sum_{k=0}^{K}{}^{\prime\prime} \left\{ -\gamma\delta_{k}^{\langle 2 \rangle}\left(\frac{U_{k}^{(n+1)} + U_{k}^{(n)}}{2}\right) + \frac{dF}{d (U_{k}^{(n+1)},U_{k}^{(n)})} \right\}\frac{ U_{k}^{(n+1)} - U_{k}^{(n)} }{\Delta t}\Delta x \\[-2pt]
	& \quad + \left[\gamma\left\{ \delta_{k}^{\langle 1 \rangle}\left(\frac{U_{k}^{(n+1)} + U_{k}^{(n)}}{2} \right)\right\}\frac{ U_{k}^{(n+1)} - U_{k}^{(n)} }{\Delta t}\right]_{0}^{K} \\[-2pt]
	& = \sum_{k=0}^{K}{}^{\prime\prime} P_{k}^{(n)}\left(\delta_{k}^{\langle 2 \rangle}P_{k}^{(n)}\right)\Delta x -\gamma \left|\delta_{n}^{+}U_{0}^{(n)} \right|^{2} -\gamma \left|\delta_{n}^{+}U_{K}^{(n)} \right|^{2} \\[-3pt]
	& = \left[ \left(\delta_{k}^{\langle 1 \rangle}\! P_{k}^{(n)}\right)\! P_{k}^{(n)} \right]_{0}^{K} \! - \sum_{k=0}^{K-1}\left|\delta_{k}^{+}P_{k}^{(n)}\right|^{2} \Delta x -\gamma \! \left|\delta_{n}^{+}U_{0}^{(n)} \right|^{2} -\gamma \! \left|\delta_{n}^{+}U_{K}^{(n)} \right|^{2} \\[-3pt]
	& = - \sum_{k=0}^{K-1}\left|\delta_{k}^{+}P_{k}^{(n)}\right|^{2}\Delta x -\gamma \left|\delta_{n}^{+}U_{0}^{(n)} \right|^{2} -\gamma \left|\delta_{n}^{+}U_{K}^{(n)} \right|^{2}
\end{align*}
for all $n = 0,1,\ldots$. \hfill $\Box$\\
\end{proof} 
\noindent 
Furthermore, the proposed scheme \eqref{dCH1}--\eqref{dNBC} has the following property corresponding to \eqref{MC}, i.e.,
\begin{theorem}\label{thm:2.2}
The solution of the scheme \eqref{dCH1}--\eqref{dNBC} satisfies the following equality. 
\begin{equation*}
\delta_{n}^{+}M_{\rm d}(\bm{U}^{(n)}) = 0 \quad (n= 0,1,\ldots).
\end{equation*}
\end{theorem}
\noindent 
For the proof, we use the following lemma. 
This lemma can be shown by direct calculation and here omitted. 
\begin{lemma}[Summation of a difference {\cite[Propositon 3.1]{B}}]\label{lem:2.3}
The following fundamental formula holds: 
\begin{equation*}
\sum_{k=0}^{K}{}^{\prime\prime}\delta_{k}^{\langle 2 \rangle}f_{k}\Delta x = \left[ \delta_{k}^{\langle 1 \rangle}f_{k} \right]_{0}^{K} \quad \mbox{for\ all\ } \{f_{k}\}_{k=-1}^{K+1} \in \mathbb{R}^{K+3}. 
\end{equation*}
\end{lemma}

\paragraph{Proof of Theorem \ref{thm:2.2}.}
Using \eqref{dCH1}, \eqref{dNBC}, and Lemma \ref{lem:2.3}, we obtain 
\begin{equation*}
\delta_{n}^{+}M_{\rm d}(\bm{U}^{(n)}) 
	= \sum_{k=0}^{K}{}^{\prime\prime}\frac{U_{k}^{(n+1)} - U_{k}^{(n)}}{\Delta t}\Delta x  
	= \sum_{k=0}^{K}{}^{\prime\prime}\delta_{k}^{\langle 2 \rangle}P_{k}^{(n)}\Delta x \\
	= \left[\delta_{k}^{\langle 1 \rangle}P_{k}^{(n)}\right]_{0}^{K} \\
	= 0
\end{equation*}
for all $n= 0,1,\ldots$. \hfill $\Box$\\

\section{Stability of the proposed scheme}
In this section, we show that, if the proposed scheme has a solution, then it satisfies the global boundedness. 
Firstly, let us give definitions of the discrete $L^{\infty}$-norm and the discrete Dirichlet semi-norm. 

\begin{definition}
We define the discrete $L^{\infty}$-norm $\|\cdot\|_{L_{\rm d}^{\infty}}$ by 
\begin{equation*}
\|\bm{f}\|_{L_{\rm d}^{\infty}} := \max_{0 \leq k \leq K}|f_{k}| \quad \mbox{for\ all\ } \bm{f} = \{f_{k}\}_{k=0}^{K} \in \mathbb{R}^{K+1}. 
\end{equation*}
For all $\bm{f} = \{f_{k}\}_{k=0}^{K} \in \mathbb{R}^{K+1}$, we define the discrete Dirichlet semi-norm $\|D\bm{f}\|$ of $\bm{f}$ by 
\begin{equation*}
\| D\bm{f}\| := \sqrt{ \sum _ {k=0}^{K-1}|\delta_{k}^{+}f_{k}|^{2}\Delta x}, 
\end{equation*}
where $D\bm{f}$ is denoted by $D\bm{f} := \{\delta_{k}^{+}f_{k}\}_{k=0}^{K-1} \in \mathbb{R}^{K}$. 
\end{definition} 
\noindent 
For the proof of the global boundedness of the numerical solution, we use the following lemmas. 
\begin{lemma}\label{lem:3.1}
The solution to the scheme \eqref{dCH1}--\eqref{dNBC} satisfies the following inequality: 
\begin{equation}
\left\| D\bm{U}^{(n)}\right\| \leq \left\{\frac{2}{\gamma}\left( J_{\rm d}(\bm{U}^{(0)}) + L\left| \min\left\{\inf_{\xi \in \mathbb{R}}F(\xi), 0\right\} \right| \right)\right\}^{\frac{1}{2}} \quad (n = 0,1,\ldots). \label{DS_bound}
\end{equation}
\end{lemma}

\begin{proof}
From the dissipative property (Theorem \ref{thm:2.1}) and the assumption on $F$, we obtain 
\begin{align*}
J_{\rm d}( \bm{U}^{(0)} ) \geq J_{\rm d}( \bm{U}^{(n)} ) \! 
	= \! \frac{\gamma}{2}\sum_{k=0}^{K-1}\left|\delta_{k}^{+}U_{k}^{(n)}\right|^{2}\!\Delta x + \sum_{k=0}^{K}\!{}^{\prime\prime}F\!\left(\! U_{k}^{(n)}\!\right) \!\Delta x \! 
	& \geq \! \frac{\gamma}{2}\left\| D\bm{U}^{(n)}\right\|^{2} \! + \! \min\!\left\{ \inf_{\xi \in \mathbb{R}}F(\xi), 0\right\}\sum_{k=0}^{K}\!{}^{\prime\prime}\Delta x \\
	& = \! \frac{\gamma}{2}\left\| D\bm{U}^{(n)}\right\|^{2} \! + \! \min\!\left\{ \inf_{\xi \in \mathbb{R}}F(\xi), 0\right\}L. 
\end{align*}
for $n=0,1,\ldots$. 
Namely, 
\begin{equation*}
\frac{\gamma}{2}\left\| D\bm{U}^{(n)}\right\|^{2} \leq J_{\rm d}(\bm{U}^{(0)}) - \min\left\{ \inf_{\xi \in \mathbb{R}}F(\xi), 0\right\}L \leq J_{\rm d}(\bm{U}^{(0)}) + \left|\min\left\{ \inf_{\xi \in \mathbb{R}}F(\xi), 0\right\}\right| L \quad (n=0,1,\ldots). 
\end{equation*}
Therefore, the inequality \eqref{DS_bound} holds. \hfill $\Box$\\
\end{proof}

\begin{lemma}[Discrete Poincar\'{e}--Wirtinger inequality {\cite[Lemma 3.3]{B}}]\label{lem:3.2}
The following inequality holds: 
\begin{equation}
\left| f_{l} - \frac{1}{L}\sum _ {k=0}^{K}{}^{\prime \prime}f_{k} \Delta x \right|^{2} \leq L\left\| D\bm{f}\right\|^{2} \quad (l=0, \ldots, K) \label{PW_ineq}
\end{equation}
for all $\bm{f} = \{f_{k}\}_{k=0}^{K} \in \mathbb{R}^{K+1}$. 
\end{lemma}

\begin{proof}
We can obtain \eqref{PW_ineq} from the proof of Furihata and Matsuo \cite{B}. \hfill $\Box$\\
\end{proof}
\noindent 
From Lemma \ref{lem:3.1} and Lemma \ref{lem:3.2}, we can obtain the following global boundedness. 

\begin{theorem}[Global boundedness]\label{thm:3.1} 
The solution to the scheme \eqref{dCH1}--\eqref{dNBC} satisfies the following inequality: \vspace{-1mm}
\begin{equation}
\left\| \bm{U}^{(n)}\right\|_{L_{\rm d}^{\infty}} \leq \frac{1}{L}\left| M_{\rm d}( \bm{U}^{(0)} ) \right| + \left\{\frac{2L}{\gamma}\left( J_{\rm d}(\bm{U}^{(0)}) + L\left| \min\left\{\inf_{\xi \in \mathbb{R}}F(\xi), 0\right\} \right| \right)\right\}^{\frac{1}{2}} \quad (n=0,1,\ldots). \label{GB} \vspace{-0.5mm}
\end{equation}
\end{theorem}
\begin{proof}

From the discrete Poincar\'{e}--Wirtinger inequality \eqref{PW_ineq} and the conservative property (Theorem \ref{thm:2.2}), we have \vspace{-1mm}
\begin{equation}
\left\| \bm{U}^{(n)} \right\|_{L_{\rm d}^{\infty}} 
	\leq \frac{1}{L}M_{\rm d}(\bm{U}^{(n)}) + L^{\frac{1}{2}}\left\| D\bm{U}^{(n)} \right\| 
	= \frac{1}{L}M_{\rm d}(\bm{U}^{(0)}) + L^{\frac{1}{2}}\left\| D\bm{U}^{(n)} \right\| 
	\leq \frac{1}{L}\left| M_{\rm d}(\bm{U}^{(0)}) \right| + L^{\frac{1}{2}}\left\| D\bm{U}^{(n)} \right\| \label{Linf_bound} \vspace{-0.5mm}
\end{equation}
for $n=0,1,\ldots$. 
By applying Lemma \ref{lem:3.1} to \eqref{Linf_bound}, we can obtain \eqref{GB}. \hfill $\Box$\\
\end{proof}
\begin{remark}
Theorem \ref{thm:3.1} means that our proposed scheme is numerically stable for any time step $n$. 
Note that we can obtain a more precise evaluation if the initial data is sufficiently smooth (see Appendix A for details). 
\end{remark}

\section{Existence and uniqueness of the solution for the proposed scheme}
In this section, using the energy method in \cite{A,Oku,Yan,Yos1,C,Yos3}, we prove that the proposed scheme \eqref{dCH1}--\eqref{dNBC} has a unique solution under a specific condition on $\Delta t$. 

\subsection{Preparation}
Let $\Omega$ be a domain in $\mathbb{R}$. 
We define $\bar{F}''$ for $F \in C^{2}(\Omega )$ and give several lemmas necessary for the proof of the existence and uniqueness of the solution. 
\begin{definition}
For a function $F \in C^{2}(\Omega )$, we define $\bar{F}''$: $\Omega^{4} \to \mathbb{R}$ by \vspace{-1.5mm}
\begin{numcases}
    {\bar{F}'' (\xi, \tilde{\xi}; \eta, \tilde{\eta}) :=} 
      \frac{1}{\xi - \tilde{\xi}}\left\{ \left( \frac{dF}{d(\xi, \eta)} + \frac{dF}{d(\xi, \tilde{\eta})} \right) - \left( \frac{dF}{d(\tilde{\xi}, \eta)} + \frac{dF}{d(\tilde{\xi}, \tilde{\eta})} \right) \right\} & ($\xi \neq \tilde{\xi}$), \nonumber\\
      {\partial}_{\xi} \left. \left( \frac{dF}{d(\xi, \eta)} + \frac{dF}{d(\xi, \tilde{\eta})} \right) \right|_{\xi = \tilde{\xi}} & ($\xi = \tilde{\xi}$). \nonumber
\end{numcases}~\vspace{-4mm}\\
for all $(\xi, \tilde{\xi}, \eta, \tilde{\eta}) \in \Omega^{4}$. 
\end{definition}
\noindent 
Since proofs of the following lemmas can be found in \cite{C}, we omit them. 
\begin{lemma}[{\cite[Lemma 2.4]{C}}] \label{lem:4.1}
If $F \in C^{2}(\Omega)$, then $\bar{F}'' \in C(\Omega^{4})$. 
Moreover, we have \vspace{-1mm}
\begin{equation*}
\left| \bar{F}'' (\xi, \tilde{\xi}; \eta, \tilde{\eta}) \right| \leq \sup_{\eta, \tilde{\eta} \in \Omega}\sup_{\xi \in \Omega}\left| \frac{\partial}{\partial \xi} \left( \frac{dF}{d(\xi, \eta)} + \frac{dF}{d(\xi, \tilde{\eta})} \right) \right| \leq \sup_{\xi \in \Omega}\left| F'' (\xi) \right| \quad \mbox{for\ all\ } (\xi, \tilde{\xi}, \eta, \tilde{\eta}) \in \Omega^{4}. \vspace{-0.5mm}
\end{equation*}
\end{lemma}

\begin{lemma}[{\cite[Proposition 2.5]{C}}] \label{lem:4.2}
Assume that $F \in C^{2}(\Omega)$. For any $\xi$, $\tilde{\xi}$, $\eta$, $\tilde{\eta} \in \Omega$, we have \vspace{-1mm}
\begin{equation*}
\frac{dF}{d(\xi, \eta)} - \frac{dF}{d(\tilde{\xi}, \tilde{\eta})} = \frac{1}{2}\bar{F}'' (\xi, \tilde{\xi}; \eta, \tilde{\eta})(\xi - \tilde{\xi}) + \frac{1}{2}\bar{F}'' (\eta, \tilde{\eta}; \xi, \tilde{\xi})(\eta - \tilde{\eta}). \vspace{-0.5mm}
\end{equation*}
\end{lemma}

\begin{lemma}[{\cite[Lemma 2.3]{C}}] \label{lem:4.3}
The following inequality holds: \vspace{-0.5mm}
\begin{equation*}
\left\| D(\bm{f}\bm{g})\right\| \leq \|\bm{f}\|_{L_{\rm d}^{\infty}}\left\| D\bm{g}\right\| + \|\bm{g}\|_{L_{\rm d}^{\infty}} \left\| D\bm{f}\right\| \quad \mbox{for\ all\ } \bm{f} = \{f_{k}\}_{k=0}^{K}, \bm{g} = \{g_{k}\}_{k=0}^{K} \in \mathbb{R}^{K+1}, \vspace{-0.5mm}
\end{equation*}
where $\bm{fg} = \{f_{k}g_{k}\}_{k=0}^{K} \in \mathbb{R}^{K+1}$. 
\end{lemma}
\noindent 
The following lemma follows from the same argument as Lemma 2.6 in \cite{C}. 
\begin{lemma}[{\cite[Lemma 3.3 (2)]{A}}] \label{lem:4.4}
Assume that $F \in C^{3}(\Omega)$. 
For any $\bm{f}_{1} = \{f_{1,k}\}_{k=0}^{K}$, $\bm{f}_{2} = \{f_{2,k}\}_{k=0}^{K}$, $\bm{f}_{3} = \{f_{3,k}\}_{k=0}^{K}$, $\bm{f}_{4} = \{f_{4,k}\}_{k=0}^{K} \in \mathbb{R}^{K+1}$, all the elements of which are in $\Omega$, we have \vspace{-1mm}
\begin{equation*}
\left\| D \bar{F}'' (\bm{f}_{1}, \bm{f}_{2}; \bm{f}_{3}, \bm{f}_{4}) \right\| \leq \frac{1}{6}\sup_{\xi \in \Omega}\left| F'''(\xi )\right| \left( 2\left\| D\bm{f}_{1}\right\| + 2\left\| D\bm{f}_{2}\right\| + \left\| D\bm{f}_{3}\right\| + \left\| D\bm{f}_{4}\right\| \right). \vspace{-0.5mm}
\end{equation*}
\end{lemma}

\subsection{Existence and uniqueness of the solution}

\begin{theorem} 
Assume that the potential function $F$ is in $C^{3}$.  
For any given $\bm{U}^{(0)} = \{U_{k}^{(0)}\}_{k=-1}^{K+1} \in \mathbb{R}^{K+3}$, let \vspace{-0.5mm}
\begin{equation*}
B_{0} := \left\{\frac{2}{\gamma}\left( J_{\rm d}(\bm{U}^{(0)}) + L\left| \min\left\{\inf_{\xi \in \mathbb{R}}F(\xi), 0\right\} \right| \right)\right\}^{\frac{1}{2}}, \quad \tilde{B}_{0} := \frac{1}{L}\left| M_{\rm d}(\bm{U}^{(0)}) \right| + L^{\frac{1}{2}}B_{0}. \vspace{-0.5mm}
\end{equation*}
If $\Delta t$ satisfies \vspace{-0.5mm}
\begin{equation}
\max\left\{\frac{3}{2}\max_{|\xi| \leq 2\tilde{B}_{0}}\left| F''(\xi) \right| , \ \frac{1}{2}\max_{|\xi| \leq 2\tilde{B}_{0}}\left| F''(\xi) \right| + \frac{ 5L^{\frac{1}{2}}B_{0} }{6}\max_{|\xi| \leq 2\tilde{B}_{0}}\left| F'''(\xi) \right| \right\} \sqrt{\frac{\Delta t}{2\gamma}} <1, \label{tcon} \vspace{-0.5mm}
\end{equation}
then there exists a unique solution $\{U_{k}^{(n)}\}_{k=-1}^{K+1} \in \mathbb{R}^{K+3} \ (n = 1,2,\ldots)$ satisfying \eqref{dCH1}--\eqref{dNBC}. 
\end{theorem}

\begin{remark} 
The assumption \eqref{tcon} is independent of the space mesh size $\Delta x$. 
Also, it is one of the advantages of the numerical method we apply that the condition on $\Delta t$ can be derived explicitly as above. 
\end{remark}

\begin{proof} 
We show the existence of a $(K+3)$-vector $\bm{U}^{(n+1)} = \{U_{k}^{(n+1)}\}_{k=-1}^{K+1} \in \mathbb{R}^{K+3}$ for any given $\bm{U}^{(n)} = \{U_{k}^{(n)}\}_{k=-1}^{K+1} \in \mathbb{R}^{K+3}$ that satisfies \eqref{dCH1}--\eqref{dNBC}. 
For the purpose, we define the nonlinear mapping $\Psi$: $\{U_{k}\}_{k=0}^{K} \mapsto \{\tilde{U}_{k}\}_{k=-1}^{K+1}$ by 
\begin{gather}
\frac{\tilde{U}_ {k} - U_ {k}^{(n)}}{\Delta t} = \delta_{k}^{\langle 2 \rangle}\tilde{P}_{k}^{(n)} \quad (k = 0,\ldots, K), \label{NLM1}\\
\tilde{P}_{k}^{(n)} = -\gamma \delta_{k}^{\langle 2 \rangle} \left(\frac{\tilde{U}_{k} + U_{k}^{(n)}}{2}\right) + \frac{dF}{d( U_{k},U_{k}^{(n)})} \quad (k = 0,\ldots, K), \label{NLM2}\\
\frac{\tilde{U}_ {0} - U_ {0}^{(n)}}{\Delta t} = \left. \delta_{k}^{\langle 1 \rangle} \left( \frac{\tilde{U}_{k} \! + \! U_{k}^{(n)}}{2}\right)  \right|_{k=0}  , \label{NLM3}\\
\frac{\tilde{U}_ {K} - U_ {K}^{(n)}}{\Delta t}  =  -  \left. \delta_{k}^{\langle 1 \rangle} \left(\frac{\tilde{U}_{k} + U_{k}^{(n)}}{2}\right) \right|_{k=K} , \label{NLM3'}\\
\delta_{k}^{\langle 1 \rangle}\tilde{P}_{k}^{(n)} = 0 \quad (k=0,K). \label{NLM4}
\end{gather} 
Firstly, we show that the mapping $\Psi$ is well-defined. 
For any fixed $\bm{U} = \{ U_{k} \}_{k=0}^{K} \in \mathbb{R}^{K+1}$, from \eqref{NLM3} and \eqref{NLM3'}, $\tilde{U}_{-1}$ and $\tilde{U}_{K+1}$ can be explicitly written as 
\begin{gather}
\tilde{U}_{-1} = - U_{-1}^{(n)} + \tilde{U}_{1} + U_{1}^{(n)} - \frac{4\Delta x}{\Delta t}\left( \tilde{U}_ {0} - U_ {0}^{(n)} \right), \label{U_-1}\\
\tilde{U}_{K+1} = - U_{K+1}^{(n)} + \tilde{U}_{K-1} + U_{K-1}^{(n)} - \frac{4\Delta x}{\Delta t}\left( \tilde{U}_ {K} - U_ {K}^{(n)} \right). \label{U_K+1}
\end{gather} 
Thus, it is sufficient to show that $\tilde{U}_{k} \ (k=0,\ldots,K)$ can be explicitly written by given $\bm{U}$ and $\bm{U}^{(n)}$.  
Using \eqref{NLM4}--\eqref{U_K+1}, we eliminate terms at $k=-1, K+1$ in \eqref{NLM1} and \eqref{NLM2}. 
Thus, we have 
\begin{gather}
\frac{\tilde{U}_ {0} - U_ {0}^{(n)}}{\Delta t} = \frac{2}{(\Delta x)^{2}}\left( \tilde{P}_{1}^{(n)} - \tilde{P}_{0}^{(n)} \right), \label{NLM1-1}\\
\frac{\tilde{U}_ {k} - U_ {k}^{(n)}}{\Delta t} = \delta_{k}^{\langle 2 \rangle}\tilde{P}_{k}^{(n)} \quad (k = 1,\ldots, K-1), \label{NLM1-2}\\
\frac{\tilde{U}_ {K} - U_ {K}^{(n)}}{\Delta t} = \frac{2}{(\Delta x)^{2}}\left( \tilde{P}_{K-1}^{(n)} - \tilde{P}_{K}^{(n)} \right), \label{NLM1-3}\\
\tilde{P}_{0}^{(n)} = -\frac{2\gamma}{(\Delta x)^{2}}\left\{ \left(\frac{\tilde{U}_{1} + U_{1}^{(n)}}{2}\right) - \left(\frac{\tilde{U}_{0} + U_{0}^{(n)}}{2}\right) \right\} + \frac{2\gamma}{\Delta x\Delta t}\left( \tilde{U}_ {0} - U_ {0}^{(n)} \right) + \frac{dF}{d( U_{0},U_{0}^{(n)})}, \label{NLM2-1}\\
\tilde{P}_{k}^{(n)} = -\gamma \delta_{k}^{\langle 2 \rangle} \left(\frac{\tilde{U}_{k} + U_{k}^{(n)}}{2}\right) + \frac{dF}{d( U_{k},U_{k}^{(n)})} \quad (k = 1,\ldots, K-1), \label{NLM2-2}\\
\tilde{P}_{K}^{(n)} = -\frac{2\gamma}{(\Delta x)^{2}}\left\{ \left(\frac{\tilde{U}_{K-1} + U_{K-1}^{(n)}}{2}\right) - \left(\frac{\tilde{U}_{K} + U_{K}^{(n)}}{2}\right) \right\} + \frac{2\gamma}{\Delta x\Delta t}\left( \tilde{U}_ {K} - U_ {K}^{(n)} \right) + \frac{dF}{d( U_{K},U_{K}^{(n)})}. \label{NLM2-3}
\end{gather}
Here, we give the following matrix expression of \eqref{NLM1-1}--\eqref{NLM2-3}: 
\begin{equation*}
A\bm{\tilde{U}} = f(\bm{U}, \bm{U}^{(n)}). 
\end{equation*}
By using \eqref{NLM2-1}--\eqref{NLM2-3}, we eliminate $\tilde{P}_{k}^{(n)}$ in \eqref{NLM1-1}--\eqref{NLM1-3}. 
Then, the $(K+1) \times (K+1)$ matrix $A$ is defined by 
\begin{equation}
A := I + \beta \begin{pmatrix}
\displaystyle 6 + \frac{2}{\alpha} & -8 & 2 & & & & & & \\[5pt]
\displaystyle -4-\frac{1}{\alpha} & 7 & -4 & 1 & & & & & \\
1 & -4 & 6 & -4 & 1 & & & \\
& 1 & -4 & 6 & -4 & 1 & & & \\
& & \ddots & \ddots & \ddots & \ddots & \ddots & & \\
& & & 1 & -4 & 6 & -4 & 1 & \\
& & & & 1 & -4 & 6 & -4 & 1 \\
& & & & & 1 & -4 & 7 & \displaystyle -4-\frac{1}{\alpha} \\[5pt]
& & & & & & 2 & -8 & \displaystyle 6 + \frac{2}{\alpha} \\
\end{pmatrix}, \label{matrixA}
\end{equation}
where $I$ is the $(K+1)$-dimensional identity matrix. 
Besides, $\alpha$ and $\beta$ are defined by $\alpha := \Delta t/(4\Delta x)$ and $\beta := \gamma\Delta t/(2(\Delta x)^{4})$, respectively. 
If the matrix $A$ is nonsingular, then the mapping $\Psi$ is well-defined. 
In fact, $A$ is nonsingular (see Appendix B, Lemma \ref{Ap:E}). 

Next, we prove the existence and uniqueness of the solution to the proposed scheme by the fixed-point theorem for a contraction mapping. 
From \eqref{U_-1} and \eqref{U_K+1}, it is sufficient to show the existence of a $(K+1)$-dimensional vector $\bm{U} = \{U_{k}\}_{k=0}^{K} \in \mathbb{R}^{K+1}$ satisfying $\tilde{U}_{k} = U_{k} \ (k=0, \ldots, K)$. 
For the purpose, we define the mapping $\Theta: \mathbb{R}^{K+1} \to \mathbb{R}^{K+1}$ by 
\begin{equation}
\Theta\left(\bm{V}\right) := \bm{V} + \frac{1}{L}M_{\rm d}(\bm{U}^{(0)})\bm{1} \quad \mbox{for\ all\ } \bm{V} \in \mathbb{R}^{K+1}, \label{The}
\end{equation}
where $\bm{1} := (1,1,\ldots,1)^{\top} \in \mathbb{R}^{K+1}$. 
Then, its inverse mapping $\Theta^{-1}$ is written as 
\begin{equation}
\Theta^{-1}\left(\bm{V}\right) = \bm{V} - \frac{1}{L}M_{\rm d}(\bm{U}^{(0)})\bm{1} \quad \mbox{for\ all\ } \bm{V} \in \mathbb{R}^{K+1}. \label{Theinv}
\end{equation}
Let us define the mapping $\Phi: \mathbb{R}^{K+1} \to \mathbb{R}^{K+1}$ by 
\begin{equation}
\Phi\left(\bm{V}\right) := \Theta^{-1} \left(\left\{ \Psi_{k} \left(\Theta\left(\bm{V}\right)\right)\right\}_{k=0}^{K}\right) \quad \mbox{for\ all\ } \bm{V} \in \mathbb{R}^{K+1}, \label{Phi}
\end{equation}
where $\Psi_{k} \left(\Theta\left(\bm{V}\right)\right)$ is the $k$-th element of the vector $(\Psi \circ \Theta)(\bm{V}) = \Psi (\Theta(\bm{V}))$. 
Moreover, let 
\begin{equation*}
X_{0} := \left\{\bm{f} \in \mathbb{R}^{K+1};  \left\| D\bm{f}\right\| \leq 2B_{0}, \ M_{\rm d}(\bm{f})=0\right\}. 
\end{equation*}
We show that $\Phi$ is a contraction mapping on $X_{0}$ under the assumption \eqref{tcon} on $\Delta t$. 
If $\Phi$ is a contraction mapping, $\Phi$ has a unique fixed-point $\bm{V}^{\ast}$ in the closed ball $X_{0}$ from the Banach fixed point theorem.
That is, $\bm{V}^{\ast}$ satisfies $\Phi(\bm{V}^{\ast}) = \bm{V}^{\ast}$. 
From \eqref{Theinv} and \eqref{Phi}, we have 
\begin{equation}
\Phi(\bm{V}^{\ast}) = \Theta^{-1} \left(\left\{ \Psi_{k} \left(\Theta\left(\bm{V}^{\ast}\right)\right)\right\}_{k=0}^{K}\right) = \left\{ \Psi_{k} \left(\Theta\left(\bm{V}^{\ast}\right)\right)\right\}_{k=0}^{K} - \frac{1}{L}M_{\rm d}(\bm{U}^{(0)})\bm{1}. \label{fixed1}
\end{equation}
Furthermore, from \eqref{The}, we obtain 
\begin{equation}
\bm{V}^{\ast} = \Theta\left(\bm{V}^{\ast}\right) - \frac{1}{L}M_{\rm d}(\bm{U}^{(0)})\bm{1}. \label{fixed2}
\end{equation}
Hence, from \eqref{fixed1} and \eqref{fixed2}, it holds that $\left\{ \Psi_{k} \left(\Theta\left(\bm{V}^{\ast}\right)\right)\right\}_{k=0}^{K} = \Theta\left(\bm{V}^{\ast}\right)$. 
Namely, $\Theta\left(\bm{V}^{\ast}\right)$ is the solution $\bm{U}^{(n+1)}$ to the scheme \eqref{dCH1}--\eqref{dNBC}. 
Firstly, we show that $\Phi(X_{0}) \subset X_{0}$. 
For the purpose, we check $\| D( \Phi(\bm{V}) ) \| \leq 2B_{0}$ and $M_{\rm d}( \Phi(\bm{V}) ) = 0$ for any fixed $\bm{V} \in X_{0}$. 
Let $\bm{U} := \Theta (\bm{V})$. 
Then, from \eqref{The}, we have 
\begin{equation}
U_{k} = V_{k} + \frac{1}{L}M_{\rm d}(\bm{U}^{(0)}) \quad (k=0,\ldots,K). \label{ThetaV}
\end{equation}
Hence, it holds that 
\begin{equation}
\delta_{k}^{+}U_{k} = \delta_{k}^{+}\left( V_{k} + \frac{1}{L}M_{\rm d}(\bm{U}^{(0)}) \right) = \delta_{k}^{+}V_{k} \quad (k=0,\ldots,K-1). \label{DUV}
\end{equation}
Let us define $\tilde{\bm{U}} := \Psi (\bm{U})$ and $\tilde{\bm{V}} := \Phi(\bm{V}) = \Theta^{-1} (\{ \tilde{U}_{k} \}_{k=0}^{K})$. 
Then, from \eqref{Theinv}, we obtain 
\begin{equation}
\tilde{V}_{k} = \tilde{U}_{k} - \frac{1}{L}M_{\rm d}(\bm{U}^{(0)}) \quad (k=0,\ldots,K). \label{Vtil}
\end{equation}
Hence, we have 
\begin{equation}
\delta_{k}^{+}\tilde{V}_{k} = \delta_{k}^{+}\left(\tilde{U}_{k} - \frac{1}{L}M_{\rm d}(\bm{U}^{(0)} ) \right) = \delta_{k}^{+}\tilde{U}_{k} \quad (k=0,\ldots,K-1). \label{DUVt}
\end{equation}
In addition, it follows from \eqref{NLM1}, \eqref{NLM4}, Lemma \ref{lem:2.3}, and Theorem \ref{thm:2.2} that 
\begin{equation*}
M_{\rm d}(\tilde{\bm{U}}) = \sum_{k=0}^{K}{}^{\prime\prime}\tilde{U}_{k}\Delta x 
	= \sum_{k=0}^{K}{}^{\prime\prime}U_{k}^{(n)} \Delta x + \Delta t\sum_{k=0}^{K}{}^{\prime\prime}\delta_{k}^{\langle 2 \rangle}\tilde{P}_{k}^{(n)} \Delta x 
	= M_{\rm d}(\bm{U}^{(n)}) + \Delta t\left[\delta_{k}^{\langle 1 \rangle}\tilde{P}_{k}^{(n)}\right]_{0}^{K} 
	= M_{\rm d}(\bm{U}^{(0)}). 
\end{equation*}
Thus, it holds from \eqref{Vtil} and this inequality that 
\begin{equation*}
M_{\rm d}(\tilde{\bm{V}}) 
	= \sum_{k=0}^{K}{}^{\prime\prime}\tilde{V}_{k}\Delta x 
	= \sum_{k=0}^{K}{}^{\prime\prime}\left(\tilde{U}_{k} - \frac{1}{L}M_{\rm d}(\bm{U}^{(0)}) \right)\Delta x 
	= M_{\rm d}(\tilde{\bm{U}})- M_{\rm d}(\bm{U}^{(0)})
	= M_{\rm d}(\bm{U}^{(0)}) - M_{\rm d}(\bm{U}^{(0)}) 
	= 0.
\end{equation*}
Therefore, all that is left is to show $\| D\tilde{ \bm{V} } \| \leq 2B_{0}$. 
From Corollary \ref{col:2.1} and \eqref{NLM1}--\eqref{NLM3'}, we have 
\begin{align*}
& \frac{1}{\Delta t}\left( \left\| D\bm{\tilde{U}}\right\|^{2} - \left\| D\bm{U}^{(n)}\right\|^{2} \right) = 2\sum _ {k=0}^{K-1}\!\left\{ \delta_{k}^{+}\left(\frac{\tilde{U}_{k} + U_{k}^{(n)}}{2}\right) \right\} \left\{ \delta_{k}^{+}\left(\frac{\tilde{U}_{k} - U_{k}^{(n)}}{\Delta t}\right)\!\right\}\Delta x \\
	& = -2\sum _ {k=0}^{K}{}^{\prime \prime}\left\{\delta_{k}^{\langle 2 \rangle}\left(\frac{\tilde{U}_{k} + U_{k}^{(n)}}{2}\right)\right\}\frac{\tilde{U}_ {k} - U_ {k}^{(n)}}{\Delta t}\Delta x + 2\left[\left\{\delta_{k}^{\langle 1 \rangle}\left(\frac{\tilde{U}_{k} + U_{k}^{(n)}}{2}\right)\right\}\frac{\tilde{U}_ {k} - U_ {k}^{(n)}}{\Delta t}\right]_{0}^{K} \\
	& = -2\sum _ {k=0}^{K}{}^{\prime \prime}\left\{ -\frac{1}{\gamma}\tilde{P}_{k}^{(n)} + \frac{1}{\gamma}\frac{dF}{d( U_{k},U_{k}^{(n)})} \right\}\left(\delta_{k}^{\langle 2 \rangle}\tilde{P}_{k}^{(n)}\right)\Delta x - 2\left( \frac{\tilde{U}_ {K} - U_ {K}^{(n)}}{\Delta t}\right)^{2} - 2\left( \frac{\tilde{U}_ {0} - U_ {0}^{(n)}}{\Delta t}\right)^{2} \\
	& \leq \frac{2}{\gamma}\sum _ {k=0}^{K}{}^{\prime \prime}\tilde{P}_{k}^{(n)}\left(\delta_{k}^{\langle 2 \rangle}\tilde{P}_{k}^{(n)}\right)\Delta x - \frac{2}{\gamma}\sum _ {k=0}^{K}{}^{\prime \prime}\frac{dF}{d( U_{k},U_{k}^{(n)})}\left(\delta_{k}^{\langle 2 \rangle}\tilde{P}_{k}^{(n)}\right)\Delta x. 
\end{align*}
Now, from Corollary \ref{col:2.1} and \eqref{NLM4}, it holds that 
\begin{equation*}
\frac{2}{\gamma}\sum _ {k=0}^{K}{}^{\prime \prime}\tilde{P}_{k}^{(n)}\left(\delta_{k}^{\langle 2 \rangle}\tilde{P}_{k}^{(n)}\right)\Delta x 
	= -\frac{2}{\gamma} \sum _ {k=0}^{K-1}\left(\delta_{k}^{+}\tilde{P}_{k}^{(n)}\right)^{2}\Delta x + \frac{2}{\gamma}\left[\left( \delta_{k}^{\langle 1 \rangle}\tilde{P}_{k}^{(n)} \right)\tilde{P}_{k}^{(n)}\right]_{0}^{K} 
	= -\frac{2}{\gamma}\left\| D\tilde{\bm{P}}^{(n)}\right\|^{2}. 
\end{equation*}
Furthermore, from Corollary \ref{col:2.1}, \eqref{NLM4}, and the Young inequality:
$ab \leq (\varepsilon/2)a^{2} + (1/(2\varepsilon))b^{2}$ for all $a,b \in \mathbb{R}$, and $\varepsilon >0$, 
we obtain 
\begin{align*}
- \frac{2}{\gamma}\sum _ {k=0}^{K}\!{}^{\prime \prime}\!\frac{dF}{d( U_{k},U_{k}^{(n)})}\!\left(\!\delta_{k}^{\langle 2 \rangle}\tilde{P}_{k}^{(n)}\!\right)\!\Delta x 
	& = \frac{2}{\gamma} \sum _ {k=0}^{K-1}\!\!\left(\!\delta_{k}^{+}\tilde{P}_{k}^{(n)}\!\right)\!\left\{\! \delta_{k}^{+}\!\!\left(\! \frac{dF}{d( U_{k},U_{k}^{(n)})} \!\right)\!\!\right\} \!\Delta x - \frac{2}{\gamma}\!\left[\!\left(\! \delta_{k}^{\langle 1 \rangle}\tilde{P}_{k}^{(n)} \!\right)\!\frac{dF}{d( U_{k},U_{k}^{(n)})}\!\right]_{0}^{K} \\
	& \leq \frac{2}{\gamma}\sum _ {k=0}^{K-1}\left[\left(\delta_{k}^{+}\tilde{P}_{k}^{(n)}\right)^{2} + \frac{1}{4}\left\{ \delta_{k}^{+}\left( \frac{dF}{d( U_{k},U_{k}^{(n)})}\right) \right\}^{2}\right]\Delta x \\
	& = \frac{2}{\gamma}\left\| D\tilde{\bm{P}}^{(n)}\right\|^{2} + \frac{1}{2\gamma}\left\| D\left( \frac{dF}{d\left( \bm{U},\bm{U}^{(n)}\right)} \right)\right\|^{2}.
\end{align*}
From the above, we have 
\begin{equation*}
\frac{1}{\Delta t}\left( \left\| D\bm{\tilde{U}}\right\|^{2} - \left\| D\bm{U}^{(n)}\right\|^{2} \right) \leq \frac{1}{2\gamma}\left\| D\left( \frac{dF}{d\left( \bm{U},\bm{U}^{(n)}\right)} \right)\right\|^{2}.
\end{equation*}
Consequently, using the triangle inequality: $\sqrt{a^{2} + b^{2}} \leq |a| + |b|$ for all $a,b \in \mathbb{R}$, we get 
\begin{equation}
\left\| D\bm{\tilde{U}}\right\| \leq \left\| D\bm{U}^{(n)}\right\| + \sqrt{\frac{\Delta t}{2\gamma}}\left\| D\left( \frac{dF}{d\left( \bm{U},\bm{U}^{(n)}\right)} \right)\right\|. \label{seminorm_a}
\end{equation}
Thus, from \eqref{DUVt} and \eqref{seminorm_a}, it is sufficient to show that the right-hand side of \eqref{seminorm_a} is not greater than $2B_{0}$. 
For $k=0,\ldots, K-1$, using Lemma \ref{lem:4.2}, we have 
\begin{equation}
\delta_{k}^{+}\left(\frac{dF}{d(U_{k}, U_{k}^{(n)})}\right) = \frac{1}{2}\bar{F}'' (U_{k+1},U_{k} ; U_{k+1}^{(n)}, U_{k}^{(n)})\delta_{k}^{+}U_{k} + \frac{1}{2}\bar{F}'' (U_{k+1}^{(n)}, U_{k}^{(n)} ; U_{k+1}, U_{k})\delta_{k}^{+}U_{k}^{(n)}. \label{dp_df}
\end{equation}
Hence, using \eqref{dp_df} and the Minkowski inequality, we obtain 
\begin{align*}
& \left\| D\!\left( \frac{dF}{d\!\left(\! \bm{U},\bm{U}^{(n)}\right)} \!\right) \!\right\| \! = \! \left\{ \sum _ {k=0}^{K-1}\!\left| \frac{1}{2}\bar{F}'' (U_{k+1},U_{k} ; U_{k+1}^{(n)}, U_{k}^{(n)})\delta_{k}^{+}U_{k} + \frac{1}{2}\bar{F}'' (U_{k+1}^{(n)}, U_{k}^{(n)} ; U_{k+1}, U_{k})\delta_{k}^{+}U_{k}^{(n)} \right|^{2} \!\! \Delta x \!\right\}^{\! \frac{1}{2}} \\
	& \leq \left\{ \sum _ {k=0}^{K-1}\left| \frac{1}{2}\bar{F}'' (U_{k+1},U_{k} ; U_{k+1}^{(n)}, U_{k}^{(n)})\delta_{k}^{+}U_{k} \right|^{2}\!\! \Delta x \!\right\}^{\! \frac{1}{2}} + \left\{ \sum _ {k=0}^{K-1}\left| \frac{1}{2}\bar{F}'' (U_{k+1}^{(n)}, U_{k}^{(n)} ; U_{k+1}, U_{k})\delta_{k}^{+}U_{k}^{(n)} \right|^{2}\!\! \Delta x \!\right\}^{\! \frac{1}{2}} \\
%	& \leq \frac{1}{2}\left\{ \max_{0 \leq k \leq K-1}\left| \bar{F}'' (U_{k+1},U_{k} ; U_{k+1}^{(n)}, U_{k}^{(n)}) \right|^{2}\sum _ {k=0}^{K-1}\left| \delta_{k}^{+}U_{k} \right|^{2} \Delta x \right\}^{\frac{1}{2}} \\ 
%	& \quad + \frac{1}{2}\left\{ \max_{0 \leq k \leq K-1}\left| \bar{F}'' (U_{k+1}^{(n)}, U_{k}^{(n)} ; U_{k+1}, U_{k}) \right|^{2}\sum _ {k=0}^{K-1}\left| \delta_{k}^{+}U_{k}^{(n)} \right|^{2}\Delta x \right\}^{\frac{1}{2}} \\
	& \leq \frac{1}{2}\max_{0 \leq k \leq K-1}\left| \bar{F}'' (U_{k+1},U_{k} ; U_{k+1}^{(n)}, U_{k}^{(n)}) \right| \left\| D\bm{U} \right\| + \frac{1}{2}\max_{0 \leq k \leq K-1}\left| \bar{F}'' (U_{k+1}^{(n)}, U_{k}^{(n)} ; U_{k+1}, U_{k}) \right| \left\| D\bm{U}^{(n)} \right\| . 
\end{align*}
Next, we consider $|\bar{F}'' (U_{k+1},U_{k} ; U_{k+1}^{(n)}, U_{k}^{(n)})|$ and $|\bar{F}'' (U_{k+1}^{(n)}, U_{k}^{(n)} ; U_{k+1}, U_{k})|$. 
It follows from Lemma \ref{lem:3.2} that 
\begin{equation*}
\left\| \bm{U} \right\|_{L_{\rm d}^{\infty}} \leq \frac{1}{L}M_{\rm d}(\bm{U}) + L^{\frac{1}{2}}\left\| D\bm{U} \right\| .
\end{equation*}
Since it holds from $\bm{V} \in X_{0}$ that $M_{\rm d}(\bm{V}) = 0$, we have 
\begin{equation*}
M_{\rm d}(\bm{U}) 
	= \sum_{k=0}^{K}{}^{\prime\prime}U_{k}\Delta x 
	= M_{\rm d}(\bm{V}) + M_{\rm d}(\bm{U}^{(0)}) 
	= M_{\rm d}(\bm{U}^{(0)})
\end{equation*}
by \eqref{ThetaV}. 
Because it holds from $\bm{V} \in X_{0}$ and \eqref{DUV} that $\| D\bm{U}\| = \| D\bm{V}\| \leq 2B_{0}$, we obtain 
\begin{equation*}
\left\| \bm{U} \right\|_{L_{\rm d}^{\infty}} \leq \frac{1}{L}M_{\rm d}(\bm{U}) + 2L^{\frac{1}{2}}B_{0} \leq \frac{1}{L}\left| M_{\rm d}(\bm{U}^{(0)})\right| + 2L^{\frac{1}{2}}B_{0} \leq 2\tilde{B}_{0}.
\end{equation*}
Also, using Theorem \ref{thm:3.1}, we get $\| \bm{U}^{(n)} \|_{L_{\rm d}^{\infty}} \leq \tilde{B}_{0}$.
Therefore, from Lemma \ref{lem:4.1}, we obtain 
\begin{equation*}
\left|\bar{F}'' (U_{k+1},U_{k} ; U_{k+1}^{(n)}, U_{k}^{(n)})\right| \leq \max_{|\xi| \leq 2\tilde{B}_{0}}\left| F'' (\xi)\right| , \quad \left|\bar{F}'' (U_{k+1}^{(n)}, U_{k}^{(n)} ; U_{k+1}, U_{k})\right| \leq \max_{|\xi| \leq 2\tilde{B}_{0}}\left| F'' (\xi)\right|  
\end{equation*}
for $k=0,\ldots, K-1$. 
Hence, it holds that 
\begin{equation}
\left\| D\left( \frac{dF}{d\!\left( \bm{U},\bm{U}^{(n)}\right)} \right) \right\| \leq \frac{1}{2}\max_{|\xi| \leq 2\tilde{B}_{0}}\left| F'' (\xi)\right| \left( \left\| D\bm{U} \right\| + \| D\bm{U}^{(n)} \| \right) . \label{seminorm_b}
\end{equation}
Consequently, using \eqref{seminorm_a}, \eqref{seminorm_b}, and Lemma \ref{lem:3.1}, the following estimate holds: 
\begin{equation*}
\left\| D\bm{\tilde{U}}\right\| 
	\leq \| D\bm{U}^{(n)}\| + \sqrt{\frac{\Delta t}{2\gamma}} \cdot \frac{1}{2}\max_{|\xi| \leq 2\tilde{B}_{0}}\left| F'' (\xi)\right|\left( \| D\bm{U}\| + \| D\bm{U}^{(n)}\| \right) 
	\leq B_{0} + \frac{\displaystyle 3\max_{|\xi| \leq 2\tilde{B}_{0}}\left| F'' (\xi)\right|}{2}\sqrt{\frac{\Delta t}{2\gamma}}B_{0}. 
\end{equation*}
Now, from \eqref{DUVt} and the assumption \eqref{tcon}, we have $\| D\bm{\tilde{V}}\| = \| D\bm{\tilde{U}}\| \leq 2B_{0}$. 
From the above, it holds that $\Phi(\bm{V}) = \bm{\tilde{V}} \in X_{0}$, i.e., $\Phi(X_{0}) \subset X_{0}$.  
Next, we prove that $\Phi$ is contractive. 
For any $\bm{V}_{1}, \bm{V}_{2} \in X_{0}$, let $\bm{U}_{1} := \Theta (\bm{V}_{1})$ and $\bm{U}_{2} := \Theta (\bm{V}_{2})$. 
From \eqref{The}, it holds that 
\begin{equation}
U_{i,k} = V_{i,k} + \frac{1}{L}M_{\rm d}(\bm{U}^{(0)}) \quad (k=0,\ldots,K, \ i=1,2). \label{U12} 
\end{equation} 
It follows from \eqref{U12} that 
\begin{equation}
\delta_{k}^{+}U_{i,k} = \delta_{k}^{+}\left( V_{i,k} + \frac{1}{L}M_{\rm d}(\bm{U}^{(0)}) \right) = \delta_{k}^{+}V_{i,k} \quad (k=0,\ldots, K-1,\ i=1,2).  \label{DUV12}
\end{equation}
Furthermore, from \eqref{U12} and $\bm{V}_{i} \in X_{0} \ (i=1,2)$, we have 
\begin{equation}
M_{\rm d}(\bm{U}_{i}) 
	= \sum_{k=0}^{K}{}^{\prime\prime}U_{i,k}\Delta x 
	= M_{\rm d}(\bm{V}_{i}) + M_{\rm d}(\bm{U}^{(0)}) 
	= M_{\rm d}(\bm{U}^{(0)}) \quad (i=1,2). \label{MasC} 
\end{equation}
Moreover, it follows from $U_{1,k} - U_{2,k} = V_{1,k} - V_{2,k} \ (k=0,\ldots,K)$ that $\|D(\bm{U}_{1} - \bm{U}_{2})\| = \|D(\bm{V}_{1} - \bm{V}_{2})\|$. 
Now, let us define $\tilde{\bm{U}}_{i} := \Psi (\bm{U}_{i})$ and $\tilde{\bm{V}}_{i} := \Phi (\bm{V}_{i}) = \Theta^{-1} (\{ \tilde{U}_{i,k} \}_{k=0}^{K}) \ (i=1,2)$. 
Then, from \eqref{Theinv}, we have 
\begin{equation*}
\tilde{V}_{i,k} = \tilde{U}_{i,k} - \frac{1}{L}M_{\rm d}(\bm{U}^{(0)}) \quad (k=0,\ldots, K,\ i=1,2). 
\end{equation*}
Hence, it holds that $\tilde{V}_{1,k} - \tilde{V}_{2,k} = \tilde{U}_{1,k} - \tilde{U}_{2,k} \ (k=0,\ldots,K)$. 
Namely, $\|D(\tilde{\bm{V}}_{1} - \tilde{\bm{V}}_{2})\| = \|D(\tilde{\bm{U}}_{1} - \tilde{\bm{U}}_{2})\|$. 
Now, from the definition of $\Psi$, the vector $\{\tilde{U}_{i,k}\}_{k=-1}^{K+1} = \{\Psi_{k}(\bm{U}_{i})\}_{k=-1}^{K+1}$ satisfies \eqref{NLM1}--\eqref{NLM4} ($i=1,2$). 
Subtracting these relations, we obtain \vspace{-1mm}
\begin{gather}
\frac{\tilde{U}_ {1,k} - \tilde{U}_ {2,k}}{\Delta t} = \delta_{k}^{\langle 2 \rangle}\left(\tilde{P}_{1,k}^{(n)} - \tilde{P}_{2,k}^{(n)} \right) \quad (k = 0,\ldots, K), \label{NLM1-12}\\[-2pt]
\delta_{k}^{\langle 2 \rangle} \left( \tilde{U}_{1,k} - \tilde{U}_{2,k}\right) = \frac{2}{\gamma}\left\{ -\left(\tilde{P}_{1,k}^{(n)} - \tilde{P}_{2,k}^{(n)}\right) + \left( \frac{dF}{d( U_{1,k},U_{k}^{(n)})} - \frac{dF}{d( U_{2,k},U_{k}^{(n)})}\right) \right\} \quad (k = 0,\ldots, K), \label{NLM2-12}\\[-2pt]
\frac{\tilde{U}_ {1,0} - \tilde{U}_ {2,0}}{\Delta t} = \left. \delta_{k}^{\langle 1 \rangle} \left( \frac{\tilde{U}_{1,k} - \tilde{U}_{2,k}}{2}\right)  \right|_{k=0}, \label{NLM3-12}\\[-2pt]
\frac{\tilde{U}_ {1,K} - \tilde{U}_ {2,K}}{\Delta t} = - \left. \delta_{k}^{\langle 1 \rangle} \left(\frac{\tilde{U}_{1,k} - \tilde{U}_{2,k}}{2}\right) \right|_{k=K}, \label{NLM3'-12}\\[-2pt]
\delta_{k}^{\langle 1 \rangle} \left( \tilde{P}_{1,k}^{(n)} - \tilde{P}_{2,k}^{(n)} \right) = 0 \quad (k = 0,K). \label{NLM4-12}
\end{gather}
From Corollary \ref{col:2.1}, \eqref{NLM1-12}--\eqref{NLM4-12}, and the Young inequality, we have \vspace{-1mm}
\begin{align*}
& \left\| D(\bm{\tilde{U}}_{1} - \bm{\tilde{U}}_{2})\right\|^{2} = \sum _ {k=0}^{K}{}^{\prime \prime}\left\{\delta_{k}^{+}(\tilde{U}_ {1,k} - \tilde{U}_ {2,k})\right\}^{2}\Delta x \\[-2pt]
	& = -\sum _ {k=0}^{K}\!{}^{\prime \prime}\!\left\{\delta_{k}^{\langle 2 \rangle}\!\!\left(\tilde{U}_{1,k} - \tilde{U}_ {2,k}\right)\!\right\}\!\left( \tilde{U}_ {1,k} - \tilde{U}_ {2,k}\right)\Delta x + \left[\left\{\delta_{k}^{\langle 1 \rangle}\!\!\left(\tilde{U}_{1,k} - \tilde{U}_ {2,k}\right)\!\right\}\!\left( \tilde{U}_ {1,k} - \tilde{U}_ {2,k}\right) \right]_{0}^{K} \\[-2pt]
	& = \frac{2\Delta t}{\gamma}\sum _ {k=0}^{K}\!{}^{\prime \prime}\!\left(\tilde{P}_{1,k}^{(n)} - \tilde{P}_{2,k}^{(n)}\right) \! \left\{\delta_{k}^{\langle 2 \rangle}\!\!\left(\tilde{P}_{1,k}^{(n)} - \tilde{P}_ {2,k}^{(n)}\right)\!\right\}\Delta x \\[-2pt]
	& \quad - \frac{2\Delta t}{\gamma}\sum _ {k=0}^{K}\!{}^{\prime \prime} \! \left(\! \frac{dF}{d( U_{1,k},U_{1,k}^{(\! n\! )})} \! - \! \frac{dF}{d( U_{2,k},U_{2,k}^{(\! n\! )})}\!\right) \!\! \left\{\! \delta_{k}^{\langle 2 \rangle}\!\!\left(\! \tilde{P}_{1,k}^{(\! n\! )} \! - \! \tilde{P}_ {2,k}^{(\! n\! )}\!\right)\!\right\}\!\Delta x \! - \! \frac{2}{\Delta t}\!\!\left( \tilde{U}_ {1,K} \! - \! \tilde{U}_ {2,K}\right)^{\! 2} \! - \! \frac{2}{\Delta t}\!\!\left( \tilde{U}_ {1,0} \! - \! \tilde{U}_ {2,0}\right)^{\! 2} \\[-2pt] 
	& \leq -\frac{2\Delta t}{\gamma}\sum _ {k=0}^{K-1}\left\{\delta_{k}^{+}\left(\tilde{P}_{1,k}^{(n)} - \tilde{P}_ {2,k}^{(n)}\right)\right\}^{2}\Delta x + \frac{2\Delta t}{\gamma}\left[ \left\{\delta_{k}^{\langle 1 \rangle}\!\!\left(\tilde{P}_{1,k}^{(n)} - \tilde{P}_ {2,k}^{(n)}\right)\!\right\} \! \left(\tilde{P}_{1,k}^{(n)} - \tilde{P}_ {2,k}^{(n)}\right) \! \right]_{0}^{K}\\[-2pt]
	& \quad + \frac{2\Delta t}{\gamma}\sum _ {k=0}^{K-1}\! \left\{\delta_{k}^{+}\left(\! \frac{dF}{d( U_{1,k},U_{1,k}^{(n)})} - \frac{dF}{d( U_{2,k},U_{2,k}^{(n)})}\!\right) \! \right\}\! \left\{\delta_{k}^{+}\!\!\left(\tilde{P}_{1,k}^{(n)} - \tilde{P}_ {2,k}^{(n)}\right)\!\right\}\Delta x \\[-2pt]
	& \quad - \frac{2\Delta t}{\gamma}\left[\left\{\delta_{k}^{\langle 1 \rangle}\!\!\left(\tilde{P}_{1,k}^{(n)} - \tilde{P}_ {2,k}^{(n)}\right)\!\right\} \left(\! \frac{dF}{d( U_{1,k},U_{1,k}^{(n)})} - \frac{dF}{d( U_{2,k},U_{2,k}^{(n)})}\!\right)\right]_{0}^{K} \\[-2pt]
	& \leq -\frac{2\Delta t}{\gamma}\!\left\| D\! \left(\! \tilde{\bm{P}}_{1}^{(n)} \! - \! \tilde{\bm{P}}_{2}^{(n)} \!\right)\right\|^{2} \! + \! \frac{2\Delta t}{\gamma}\!\sum _ {k=0}^{K-1}\!\! \left[\frac{1}{4}\!\left\{\! \delta_{k}^{+}\!\!\left(\! \frac{dF}{d( U_{1,k},U_{1,k}^{(n)})} - \frac{dF}{d( U_{2,k},U_{2,k}^{(n)})}\!\right) \! \right\}^{\! 2} \!\! + \! \left\{\! \delta_{k}^{+}\!\!\left(\! \tilde{P}_{1,k}^{(n)} \! - \! \tilde{P}_ {2,k}^{(n)} \!\right)\!\right\}^{\! 2}\right] \! \Delta x \\[-2pt]
	& = \frac{\Delta t}{2\gamma}\left\| D \left( \frac{dF}{d( \bm{U}_{1},\bm{U}_{1}^{(n)})} - \frac{dF}{d( \bm{U}_{2},\bm{U}_{2}^{(n)})} \right)\right\|^{2}. 
\end{align*}
Namely, 
\begin{equation}
\left\| D(\bm{\tilde{U}}_{1} - \bm{\tilde{U}}_{2})\right\| \leq \sqrt{\frac{\Delta t}{2\gamma}}\left\| D \left( \frac{dF}{d( \bm{U}_{1},\bm{U}_{1}^{(n)})} - \frac{dF}{d( \bm{U}_{2},\bm{U}_{2}^{(n)})} \right)\right\|. \label{difD12}
\end{equation}
Using Lemma \ref{lem:4.2}, we get 
\begin{equation*}
\frac{dF}{d(U_{1,k}, U_{k}^{(n)})} - \frac{dF}{d(U_{2,k}, U_{k}^{(n)})} = \frac{1}{2}\bar{F}'' (U_{1,k}, U_{2,k}; U_{k}^{(n)}, U_{k}^{(n)})(U_{1,k} - U_{2,k})  \quad (k=0,\ldots, K). 
\end{equation*}
Hence, it follows from Lemma \ref{lem:4.3} that 
\begin{align}
& \left\| D\left(\frac{d F}{d (\bm{U}_ {1},\bm{U}^{(n)})} - \frac{d F}{d (\bm{U}_ {2},\bm{U}^{(n)})}\right)\right\| = \frac{1}{2}\left\| D\left\{\bar{F}'' (\bm{U}_{1}, \bm{U}_{2}; \bm{U}^{(n)}, \bm{U}^{(n)})(\bm{U}_{1} - \bm{U}_{2})\right\}\right\| \nonumber\\
& \leq \frac{1}{2}\left\|\bar{F}'' (\bm{U}_{1}, \bm{U}_{2}; \bm{U}^{(n)}, \bm{U}^{(n)})\right\|_{L_{\rm d}^{\infty}}\left\| D(\bm{U}_{1} - \bm{U}_{2})\right\| + \frac{1}{2}\left\|\bm{U}_{1} - \bm{U}_{2}\right\|_{L_{\rm d}^{\infty}}\left\| D\bar{F}'' (\bm{U}_{1}, \bm{U}_{2}; \bm{U}^{(n)}, \bm{U}^{(n)})\right\|. \label{nonlinear_a}
\end{align}
We consider $\|\bar{F}'' (\bm{U}_{1}, \bm{U}_{2}; \bm{U}^{(n)}, \bm{U}^{(n)})\|_{L_{\rm d}^{\infty}}$ and $\| D\bar{F}'' (\bm{U}_{1}, \bm{U}_{2}; \bm{U}^{(n)}, \bm{U}^{(n)})\|$. 
From $\bm{V}_{i} \in X_{0} \ (i=1,2)$ and \eqref{DUV12}, we have $\| D\bm{U}_{i}\| = \| D\bm{V}_{i}\| \leq 2B_{0} \ (i=1,2)$. 
Therefore, using Lemma \ref{lem:3.2} and \eqref{MasC}, we obtain 
\begin{equation*}
\left\| \bm{U}_{i} \right\|_{L_{\rm d}^{\infty}} \leq \frac{1}{L}M_{\rm d}(\bm{U}_{i}) + L^{\frac{1}{2}}\left\| D\bm{U}_{i} \right\| \leq \frac{1}{L}\left| M_{\rm d}(\bm{U}^{(0)})\right| + 2L^{\frac{1}{2}}B_{0} \leq 2\tilde{B}_{0} \quad (i=1,2).
\end{equation*}
Hence, using Theorem \ref{thm:3.1} and Lemma \ref{lem:4.1}, we get 
\begin{equation}
\left\|\bar{F}'' (\bm{U}_{1}, \bm{U}_{2}; \bm{U}^{(n)}, \bm{U}^{(n)})\right\|_{L_{\rm d}^{\infty}} \leq \max_{|\xi| \leq 2\tilde{B}_{0}}\left| F'' (\xi)\right|. \label{nonlinear_b}
\end{equation}
Furthermore, from Lemma \ref{lem:3.1} and Lemma \ref{lem:4.4}, the following estimate holds: 
\begin{equation}
\left\| D\bar{F}'' (\bm{U}_{1}, \bm{U}_{2}; \bm{U}^{(n)}, \bm{U}^{(n)})\right\| 
	\leq \frac{1}{3}\max_{|\xi| \leq 2\tilde{B}_{0}}\left| F''' (\xi)\right| \left( \left\| D\bm{U}_{1} \right\| + \left\| D\bm{U}_{2} \right\| + \left\| D\bm{U}^{(n)} \right\| \right) 
	\leq \frac{5}{3}\max_{|\xi| \leq 2\tilde{B}_{0}}\left| F''' (\xi)\right| B_{0}. \label{nonlinear_c}
\end{equation}
Now, it follows from $U_{1,k} - U_{2,k} = V_{1,k} - V_{2,k} \ (k=0,\ldots,K)$ and $\bm{V}_{i} \in X_{0} \ (i=1,2)$ that 
\begin{equation*}
M_{\rm d}(\bm{U}_{1} - \bm{U}_{2}) = M_{\rm d}(\bm{V}_{1} - \bm{V}_{2}) = M_{\rm d}(\bm{V}_{1}) - M_{\rm d}(\bm{V}_{2}) = 0. 
\end{equation*} 
Hence, from Lemma \ref{lem:3.2}, we have 
\begin{equation}
\left\| \bm{U}_{1} - \bm{U}_{2} \right\|_{L_{\rm d}^{\infty}} \leq \frac{1}{L}M_{\rm d}(\bm{U}_{1} - \bm{U}_{2}) + L^{\frac{1}{2}}\left\| D\left(\bm{U}_{1} - \bm{U}_{2} \right)\right\| = L^{\frac{1}{2}}\left\| D\left(\bm{U}_{1} - \bm{U}_{2} \right)\right\|. \label{Uinf}
\end{equation}
Thus, using \eqref{nonlinear_a}--\eqref{Uinf}, we get the following estimate: 
\begin{equation}
\left\| D\!\left(\! \frac{d F}{d (\bm{U}_ {1},\bm{U}^{(n)})} - \frac{d F}{d (\bm{U}_ {2},\bm{U}^{(n)})}\!\right) \!\right\| 
	\leq \left( \frac{\displaystyle \max_{|\xi| \leq 2\tilde{B}_{0}}\left| F'' (\xi)\right|}{2} + \frac{\displaystyle 5L^{\frac{1}{2}}B_{0}\max_{|\xi| \leq 2\tilde{B}_{0}}\left| F''' (\xi)\right|}{6} \right) \left\| D\left(\bm{U}_{1} - \bm{U}_{2} \right)\right\|. \label{difDF12}
\end{equation}
Consequently, from \eqref{difD12} and \eqref{difDF12}, we obtain 
\begin{align*}
\left\| D(\bm{\tilde{V}}_{1} - \bm{\tilde{V}}_{2})\right\| 
	= \left\| D(\bm{\tilde{U}}_{1} - \bm{\tilde{U}}_{2})\right\| 
	& \leq \left( \frac{\displaystyle \max_{|\xi| \leq 2\tilde{B}_{0}}\left| F'' (\xi)\right|}{2} + \frac{\displaystyle 5L^{\frac{1}{2}}B_{0}\max_{|\xi| \leq 2\tilde{B}_{0}}\left| F''' (\xi)\right|}{6} \right) \sqrt{\frac{\Delta t}{2\gamma}} \left\| D\left(\bm{U}_{1} - \bm{U}_{2} \right)\right\| \\ 
	& = \left( \frac{\displaystyle \max_{|\xi| \leq 2\tilde{B}_{0}}\left| F'' (\xi)\right|}{2} + \frac{\displaystyle 5L^{\frac{1}{2}}B_{0}\max_{|\xi| \leq 2\tilde{B}_{0}}\left| F''' (\xi)\right|}{6} \right) \sqrt{\frac{\Delta t}{2\gamma}} \left\| D\left(\bm{V}_{1} - \bm{V}_{2} \right)\right\|. 
\end{align*}
Since it holds from the assumption \eqref{tcon} on $\Delta t$ that 
\begin{equation*}
\left( \frac{\displaystyle \max_{|\xi| \leq 2\tilde{B}_{0}}\left| F'' (\xi)\right|}{2} + \frac{\displaystyle 5L^{\frac{1}{2}}B_{0}\max_{|\xi| \leq 2\tilde{B}_{0}}\left| F''' (\xi)\right|}{6} \right) \sqrt{\frac{\Delta t}{2\gamma}} < 1, 
\end{equation*}
the mapping $\Phi$ is contraction into $X_{0}$. 
This completes the proof. \hfill $\Box$\\ 
\end{proof} 
\noindent 
The following corollary follows from the same argument as Corollary 3.3 in \cite{C}. 
\begin{corollary}\label{col:4.2}
Assume that $F(s) = (q/4)s^{4} - (r/2)s^{2}$ for all $s \in \mathbb{R}$, where $q$ and $r$ are positive constants. 
If $\Delta t$ satisfies \vspace{-1mm}
\begin{equation*}
\max\left\{\frac{3r}{2},\ \frac{17q}{2}\tilde{B}_{0}^{2} + \frac{r}{2},\ \frac{51q}{4}\tilde{B}_{0}^{2} - \frac{r}{2} \right\} \sqrt{\frac{\Delta t}{2\gamma}} <1, \vspace{-1mm}
\end{equation*}
then there exists a unique solution $\{U_{k}^{(n)}\}_{k=-1}^{K+1} \in \mathbb{R}^{K+3} \ (n = 1,2,\ldots)$ satisfying \eqref{dCH1}--\eqref{dNBC}. 
\end{corollary}

\section{Error estimate}
In this section, we show the error estimate. 
We also use the energy method in \cite{A,Oku,Yan,Yos1,C,Yos3}. 
Fix a natural number $N \in \mathbb{N}$. 
We compute $\bm{U}^{(n)}$ up to $n=N$ by our proposed scheme \eqref{dCH1}--\eqref{dNBC} and estimate the error between it and the solution to the problem \eqref{Eq_i}--\eqref{BCii} up to $T=N\Delta t$.
Let $u$ and $p$ be the solutions to the problem \eqref{Eq_i}--\eqref{BCii} with an initial condition. 
Besides, assume that $u \in C^{4}([0,L] \times [0,T])$. 
Also, we assume that the potential $F$ is sufficiently smooth. 
Moreover, for all $t \in [0,T]$, we denote a function $\tilde{p}(\cdot ,t)$ in the domain $[-\Delta x,L + \Delta x] \times [0,T]$ which satisfies that $\partial_{x}^{5}\tilde{p}$ exists and is continuous on $[-\Delta x,L + \Delta x] \times [0,T]$ and that the following properties hold: 
\begin{gather}
\tilde{p}(x,t) = p(x,t) \quad \mbox{for\ all\ } x \in [0,L], \nonumber\\
\tilde{p}(-\Delta x, t) = \tilde{p}(\Delta x, t), \quad \tilde{p}(L + \Delta x, t) = \tilde{p}(L - \Delta x, t). \label{pb}
\end{gather}
Also, we extend the solution $u$ in $[0,L] \times [0,T]$ to a function $\tilde{u}$ in $[-\Delta x,L + \Delta x] \times [0,T]$ that satisfies $\tilde{u} \in C^{4}( [-\Delta x,L + \Delta x] \times [0,T] )$. 
For example, we define an extension $\tilde{u}$ of $u$ by 
\begin{equation}
  \tilde{u}(x,t) := \begin{cases}
    \displaystyle u(-x,t) + 2x\partial_{x}u(0,t) + \frac{x^{3}}{3}\partial_{x}^{3}u(0,t) & (-\Delta x \leq x < 0), \\[8pt]
    u(x,t) & (0 \leq x \leq L), \\
    \displaystyle u(2L - x,t) + 2(x - L)\partial_{x}u(L,t) + \frac{(x - L)^{3}}{3}\partial_{x}^{3}u(L,t) & (L < x \leq L + \Delta x), 
  \end{cases} \label{util}
\end{equation}
for $t \in [0,T]$, where $\partial_{x}f(a)$ means $\left. \partial_{x}f(x)\right|_{x=a}$.  
Let $U_{k}^{(0)} = \tilde{u}(k\Delta x, 0) \ (k = -1,0,\ldots, K,K+1)$. 
In addition, we define the errors $e_{u,k}^{(n)}$ and $e_{p,k}^{(n)}$ by 
\begin{gather*}
e_{u,k}^{(n)} := U_{k}^{(n)} - \tilde{u}(k\Delta x, n\Delta t) \quad (k=-1,0,\ldots, K,K+1,\ n=0,1,\ldots, N), \\
e_{p,k}^{(n)} := P_{k}^{(n)} - \tilde{p}\left( k\Delta x, \left( n+\frac{1}{2}\right)\Delta t\right) \quad (k=-1,0,\ldots, K,K+1,\ n=0,1,\ldots, N-1).
\end{gather*}~\vspace{-2mm}\\
For simplicity, we use the expression $\tilde{u}_{k}^{(n)} := \tilde{u}(k\Delta x, n\Delta t)$ from now on.  
Also, the expression $\delta_{k}^{\ast}f_{l}$ means $\left. \delta_{k}^{\ast}f_{k} \right|_{k=l}$, where the symbol ``$\ast$" denotes $+$, $\langle 1 \rangle$, or $\langle 2 \rangle$. 
Then, the following lemmas hold (the proofs can be found in Appendix C): 

\begin{lemma} \label{lem:5.1}
Assume that $u \in C^{4}([0,L] \times [0,T])$. 
Then, we obtain the following equations on the errors: 
\begin{gather*}
\frac{e_{u,k}^{(n+1)} - e_{u,k}^{(n)}}{\Delta t} = \delta_{k}^{\langle 2 \rangle}e_{p,k}^{(n)} + \xi_{1,k}^{\left( n + \frac{1}{2}\right)} \quad (k=0,\ldots, K,\ n=0,1,\ldots, N-1), \\
e_{p,k}^{(\! n\! )} \! = \! - \gamma \delta_{k}^{\langle 2 \rangle}\!\!\left(\!\! \frac{e_{u,k}^{(\! n+1\! )} \!\! + \! e_{u,k}^{(\! n\! )}}{2}\!\!\right) \! + \! \left(\! \frac{dF}{d(U_{k}^{(\! n+1\! )}\! ,U_{k}^{(\! n\! )})} \! - \! \frac{dF}{d(u_{k}^{(\! n+1\! )}\! ,u_{k}^{(\! n\! )})} \!\right) \! + \! \xi_{2,k}^{\left(\! n + \frac{1}{2}\!\right)} \quad (k \! = \! 0,\ldots, K,\ n \! = \! 0,1,\ldots, N \! - \! 1), \\
\frac{e_{u,0}^{(n+1)} - e_{u,0}^{(n)}}{\Delta t} = \delta_{k}^{\langle 1 \rangle}\left(\frac{e_{u,0}^{(n+1)} + e_{u,0}^{(n)}}{2}\right) + \xi_{3,0}^{\left( n + \frac{1}{2}\right)} \quad (n=0,1,\ldots, N-1), \\
\frac{e_{u,K}^{(n+1)} - e_{u,K}^{(n)}}{\Delta t} = -\delta_{k}^{\langle 1 \rangle}\left(\frac{e_{u,K}^{(n+1)} + e_{u,K}^{(n)}}{2}\right) + \xi_{3,K}^{\left( n + \frac{1}{2}\right)} \quad (n=0,1,\ldots, N-1), \\
\delta_{k}^{\langle 1 \rangle}e_{p,k}^{(n)} = 0 \quad (k=0,K,\ n=0,1,\ldots, N-1). 
\end{gather*}
where $\xi_{1}$, $\xi_{2}$, and $\xi_{3}$ are defined as follows: 
\begin{gather*}
\xi_{1,k}^{\left( n + \frac{1}{2}\right)} := \partial_{t}u_{k}^{\left( n + \frac{1}{2}\right)}  - \frac{u_{k}^{(n+1)} - u_{k}^{(n)}}{\Delta t} + \delta_{k}^{\langle 2 \rangle}\tilde{p}_{k}^{\left( n + \frac{1}{2}\right)}  - \partial_{x}^{2}p_{k}^{\left( n + \frac{1}{2}\right)} \quad (k=0,\ldots, K), \\
\xi_{2,k}^{\left( n + \frac{1}{2}\right)} := \gamma\left\{\partial_{x}^{2}u_{k}^{\left( n + \frac{1}{2}\right)} - \delta_{k}^{\langle 2 \rangle}\left(\frac{\tilde{u}_{k}^{(n+1)} + \tilde{u}_{k}^{(n)}}{2}\right) \right\} + \frac{dF}{d(u_{k}^{(n+1)},u_{k}^{(n)})} - F'( u_{k}^{\left( n + \frac{1}{2}\right)}) \quad (k=0,\ldots, K), \\
\xi_{3,0}^{\left( n + \frac{1}{2}\right)} := \partial_{t}u_{0}^{\left( n + \frac{1}{2}\right)}  - \frac{u_{0}^{(n+1)} - u_{0}^{(n)}}{\Delta t} + \delta_{k}^{\langle 1 \rangle}\left(\frac{\tilde{u}_{0}^{(n+1)} + \tilde{u}_{0}^{(n)}}{2}\right) - \partial_{x}u_{0}^{\left( n + \frac{1}{2}\right)}, \\
\xi_{3,K}^{\left( n + \frac{1}{2}\right)} := \partial_{t}u_{K}^{\left( n + \frac{1}{2}\right)}  - \frac{u_{K}^{(n+1)} - u_{K}^{(n)}}{\Delta t} + \partial_{x}u_{K}^{\left( n + \frac{1}{2}\right)} - \delta_{k}^{\langle 1 \rangle}\left(\frac{\tilde{u}_{K}^{(n+1)} + \tilde{u}_{K}^{(n)}}{2}\right). 
\end{gather*}
\end{lemma}

\begin{lemma} \label{lem:5.2} 
Assume that $u \in C^{4}([0,L] \times [0,T])$. 
Furthermore, we suppose that the potential function $F$ is in $C^{3}$.  
Denote the bounds by 
\begin{equation}
\max_{0 \leq n \leq N}\left\{ \left\| D\bm{U}^{(n)}\right\|, \left\| D\bm{u}^{(n)}\right\| \right\} \leq C_{1}, \quad \max_{0 \leq n \leq N}\left\{ \left\| \bm{U}^{(n)}\right\|_{L_{\rm d}^{\infty}}, \left\| \bm{u}^{(n)}\right\|_{L_{\rm d}^{\infty}} \right\} \leq C_{2}. \label{Bou}
\end{equation}
Also, let 
\begin{equation*}
C_{3} := \frac{\displaystyle C_{1}L^{\frac{1}{2}}\max_{|\xi| \leq C_{2}}\left| F'''(\xi)\right| + \max_{|\xi| \leq C_{2}}\left| F''(\xi)\right|}{2}. 
\end{equation*}
Then, for any fixed $\varepsilon > 0$, the following inequality holds: 
\begin{equation*}
\left\{ 1 - \Delta t\!\left( \frac{C_{3}^{2}}{\gamma} + \varepsilon \right) \right\} \left\| D\bm{e}_{u}^{(n+1)}\right\|^{2} 
	\leq \left\{ 1 + \Delta t\left( \frac{C_{3}^{2}}{\gamma} + \varepsilon \right) \right\}\left\| D\bm{e}_{u}^{(n)}\right\|^{2} + \Delta t R^{\left( n + \frac{1}{2}\right)} \quad (n=0,1, \ldots, N-1), 
\end{equation*}
where 
\begin{align}
R^{\left( n + \frac{1}{2}\right)} 
	& := \frac{1}{2\gamma}\left( 1 + \frac{2C_{3}^{2}}{\varepsilon\gamma} \right)\left(\Delta t C_{1}\max_{|\xi| \leq C_{2}}\left| F'''(\xi)\right| \sum_{j=0}^{n}\left\|\bm{\xi}_{1}^{\left( j + \frac{1}{2}\right)}\right\|_{L_{\rm d}^{\infty}} + \left\| D\bm{\xi}_{2}^{\left( n + \frac{1}{2}\right)}\right\| \right)^{2} + \frac{1}{\varepsilon}\left\| D\bm{\xi}_{1}^{\left( n + \frac{1}{2}\right)}\right\|^{2} \nonumber\\
	& \quad + \left|\xi_{1,0}^{\left( n + \frac{1}{2}\right)}\right|^{2} + \left|\xi_{1,K}^{\left( n + \frac{1}{2}\right)}\right|^{2} + \left|\xi_{3,0}^{\left( n + \frac{1}{2}\right)}\right|^{2} + \left|\xi_{3,K}^{\left( n + \frac{1}{2}\right)}\right|^{2}. \label{Rn}
\end{align}
\end{lemma}

\begin{theorem} \label{thm:5.1}
Assume that $u \in C^{5}([0,L] \times [0,T])$. 
Furthermore, we suppose that the potential function $F$ is in $C^{4}$. 
In the same manner, as Lemma \ref{lem:5.2}, denote the bounds by \eqref{Bou}. 
Fix $B \in (0,(\gamma / C_{3}^{2}) )$. 
If $\Delta t$ satisfies 
\begin{equation}
\Delta t < B \left( < \frac{\gamma}{ C_{3}^{2} } \right), \label{dt_ass}
\end{equation}
then, there exists a constant $C := C(B)$ dependent on $B$ and independent of $k$ and $n$ such that 
\begin{equation*}
\left\| (\varPi_{\Delta x, \Delta t}U)(\cdot, t) - u(\cdot, t) \right\|_{L^{\infty}(0,L)} \leq C\left( (\Delta x)^{2} + (\Delta t)^{2} \right) \quad \mbox{for\ all\ } t \in [0,T], 
\end{equation*}
%where $C(B) \to \infty \ (B \to (\gamma / C_{3}^{2}) )$. 
where $\varPi_{\Delta x, \Delta t}U$ is the function which interpolates the grid value point $U_{k}^{(n)}$ and is defined by 
\begin{align*}
(\varPi_{\Delta x, \Delta t}U)(x,t) 
	& := \left( \varPi_{\Delta x}(\varPi_{\Delta t}U_{k}) \right) (x,t) \\
	& = \left( \varPi_{\Delta t}(\varPi_{\Delta x}U^{(n)}) \right) (x,t) \\
	& = \left( k+1 - \frac{x}{\Delta x}\right) \left( n+1 - \frac{t}{\Delta t}\right) U_{k}^{(n)} + \left( k+1 - \frac{x}{\Delta x}\right) \left( \frac{t}{\Delta t} - n \right) U_{k}^{(n+1)} \\
	& \quad + \left( \frac{x}{\Delta x} - k \right) \left( n+1 - \frac{t}{\Delta t}\right) U_{k+1}^{(n)} + \left( \frac{x}{\Delta x} - k\right)\left( \frac{t}{\Delta t} - n\right) U_{k+1}^{(n+1)}
\end{align*}
for $(x,t) \in [k\Delta x, (k+1)\Delta x] \times [n\Delta t, (n+1)\Delta t], \ k = 0,1, \ldots, K-1, \ n = 0,1,\ldots, N-1$. 
Also, $\varPi_{\Delta x}$ is the function that interpolates the grid value point $f_{k}$ and is defined as follows: 
\begin{align*}
(\varPi_{\Delta x}f)(x) 
	& := f_{k} + \frac{ f_{k+1} - f_{k} }{ \Delta x } (x - k\Delta x) \\
%	& = \left\{ 1 - \left( \frac{x}{\Delta x} - k\right)\right\}f_{k} + \left( \frac{x}{\Delta x} - k\right) f_{k+1} \\
	& = \left( k+1 - \frac{x}{\Delta x}\right) f_{k} + \left( \frac{x}{\Delta x} - k\right) f_{k+1} \quad  \mbox{for\ } x \in [k\Delta x, (k+1)\Delta x], \quad k=0,1,\ldots, K-1, 
\end{align*}
and $\varPi_{\Delta t}$ is the function that interpolates the grid value point $f^{(n)}$ and is defined as follows: 
\begin{align*}
(\varPi_{\Delta t}f)(t) 
	& := f^{(n)} + \frac{ f^{(n+1)} - f^{(n)} }{ \Delta t } (t - n\Delta t) \\
%	& = \left\{ 1 - \left( \frac{t}{\Delta t} - n\right)\right\}f^{(n)} + \left( \frac{t}{\Delta t} - n\right) f^{(n+1)} \\
	& = \left(\! n+1 - \frac{t}{\Delta t}\!\right) f^{(n)} + \left(\! \frac{t}{\Delta t} - n\!\right) f^{(n+1)} \quad \mbox{for\ } t \in [n\Delta t, (n+1)\Delta t], \quad n=0,1,\ldots, N-1. 
\end{align*}
\end{theorem}

\begin{proof}
{\bf Step1.} 
Let $\varepsilon$ be an arbitrarily fixed positive number satisfying 
\begin{equation*}
\varepsilon < \frac{1}{B}\left( 1 - \frac{C_{3}^{2}}{\gamma}B \right). 
\end{equation*}
In other words, we have $B < 1/C_{4}$ for $C_{4} := (C_{3}^{2}/\gamma) + \varepsilon$. 
Let $\tilde{C}_{4} := (2C_{4})/( 1 - C_{4}B )$. 
Then, it follows from \eqref{dt_ass} that 
\begin{equation}
\frac{1 + C_{4}\Delta t}{1 - C_{4}\Delta t} < 1 + \tilde{C}_{4}\Delta t < \exp(\tilde{C}_{4}\Delta t). \label{dt_ineq}
\end{equation}
Actually, since $C_{4}$ is positive, it holds from \eqref{dt_ass} that $1 - C_{4}\Delta t > 1 - C_{4}B$. 
Also, from the definition of $\tilde{C}_{4}$, we obtain $\tilde{C}_{4}( 1 - C_{4}B ) = 2C_{4}$. 
Thus, we have 
\begin{equation*}
(1 + \tilde{C}_{4}\Delta t)(1 - C_{4}\Delta t) = 1 - C_{4}\Delta t + \tilde{C}_{4}\Delta t(1 - C_{4}\Delta t) > 1 - C_{4}\Delta t + \tilde{C}_{4}\Delta t(1 - C_{4}B) = 1 + C_{4}\Delta t. 
\end{equation*}
From this inequality, the first inequality in \eqref{dt_ineq} holds. 
The second inequality in \eqref{dt_ineq} holds from the following inequality: $1 + x < \exp(x)$ for all $x >0$. 
Using Lemma \ref{lem:5.2}, \eqref{dt_ass}, and \eqref{dt_ineq}, we obtain 
\begin{equation}
\left\| D\bm{e}_{u}^{(n+1)}\right\|^{2} \leq \frac{ 1 \! + \! C_{4}\Delta t }{ 1 \! - \! C_{4}\Delta t }\left\| D\bm{e}_{u}^{(n)}\right\|^{2} + \frac{\Delta t}{1 \! - \! C_{4}\Delta t} R^{\left( n + \frac{1}{2} \right)} \leq \exp(\tilde{C}_{4}\Delta t)\left\| D\bm{e}_{u}^{(n)}\right\|^{2} + \frac{\Delta t}{1 \! - \! C_{4}B} R^{\left( n + \frac{1}{2} \right)} \label{tool} 
\end{equation}
for $n=0,1,\ldots, N-1$. 
Using \eqref{tool} repeatedly, we have 
\begin{align*}
\left\| D\bm{e}_{u}^{(n)}\right\|^{2} 
	& \leq \exp(\tilde{C}_{4}\Delta t)\left\| D\bm{e}_{u}^{(n-1)}\right\|^{2} + \frac{\Delta t}{1 - C_{4}B} R^{\left( n - 1 + \frac{1}{2} \right)} \\
	& \leq \exp(2\tilde{C}_{4}\Delta t)\left\| D\bm{e}_{u}^{(n-2)}\right\|^{2} + \frac{\Delta t}{1 - C_{4}B}\left[ \left\{\exp(\tilde{C}_{4}\Delta t) \right\} R^{\left( n - 2 + \frac{1}{2} \right)} + R^{\left( n - 1 + \frac{1}{2} \right)}\right] \\
	& \leq \cdots \\
	& \leq \exp(n\tilde{C}_{4}\Delta t)\left\| D\bm{e}_{u}^{(0)}\right\|^{2} + \frac{\Delta t}{1 - C_{4}B}\sum_{j=1}^{n}\left[\exp\{(j-1)\tilde{C}_{4}\Delta t\}\right]R^{\left( n - j + \frac{1}{2} \right)} \\
	& = \frac{\Delta t}{1 - C_{4}B}\sum_{j=1}^{n}\left[\exp\{\tilde{C}_{4}(j-1)\Delta t\}\right]R^{\left( n - j + \frac{1}{2} \right)} \quad (n=1, \ldots, N), 
\end{align*} 
where the last equality holds from $\bm{e}_{u}^{(0)} = \bm{0}$. 
For any $j=1,2,\ldots, n$, it holds from $j-1 \leq n -1 < N$ that 
\begin{equation*}
\exp\{\tilde{C}_{4}(j-1)\Delta t\} < \exp(\tilde{C}_{4}N\Delta t) = \exp\left(\tilde{C}_{4}N \cdot \frac{T}{N} \right) = \exp(\tilde{C}_{4}T). 
\end{equation*}
Therefore, we obtain 
\begin{equation}
\left\| D\bm{e}_{u}^{(n)}\right\|^{2} \leq \frac{\exp(\tilde{C}_{4}T)}{1 - C_{4}B}\Delta t\sum_{j=1}^{n}R^{\left( n - j + \frac{1}{2} \right)} \quad (n=1,\ldots, N). \label{Derr_est} 
\end{equation}
Now, using Lemma \ref{lem:5.1} and Lemma \ref{lem:2.3}, we have 
\begin{equation*}
\delta_{n}^{+}\!\!\left(\sum_{k=0}^{K}\!{}^{\prime\prime}e_{u,k}^{(\! n\! )}\Delta x\!\right) \!
	= \! \sum_{k=0}^{K}\!{}^{\prime\prime}\delta_{k}^{\langle 2 \rangle}e_{p,k}^{(\! n\! )}\Delta x + \sum_{k=0}^{K}\!{}^{\prime\prime}\xi_{1,k}^{\left(\! n + \frac{1}{2}\!\right)}\!\Delta x \! 
	\leq \! \left[\delta_{k}^{\langle 1 \rangle}e_{p,k}^{(\! n\! )}\right]_{0}^{K} + \max_{0 \leq k \leq K} \left|\xi_{1,k}^{\left(\! n + \frac{1}{2}\!\right)}\right|\sum_{k=0}^{K}\!{}^{\prime\prime}\Delta x \!
	= \! L\left\|\bm{\xi}_{1}^{\left(\! n + \frac{1}{2}\!\right)}\right\|_{L_{\rm d}^{\infty}} 
\end{equation*}
for $n = 0,1,\ldots, N-1$. 
That is, 
\begin{equation*}
\sum_{k=0}^{K}{}^{\prime\prime}e_{u,k}^{(n+1)}\Delta x \leq \sum_{k=0}^{K}{}^{\prime\prime}e_{u,k}^{(n)}\Delta x + \Delta tL\left\|\bm{\xi}_{1}^{\left( n + \frac{1}{2}\right)}\right\|_{L_{\rm d}^{\infty}} \quad (n=0,1,\ldots, N-1). 
\end{equation*}
Using this inequality iteratively, we obtain 
\begin{align*}
\sum_{k=0}^{K}{}^{\prime\prime}e_{u,k}^{(n)}\Delta x 
	& \leq \sum_{k=0}^{K}{}^{\prime\prime}e_{u,k}^{(n-1)}\Delta x + \Delta tL\left\|\bm{\xi}_{1}^{\left( n - 1 + \frac{1}{2}\right)}\right\|_{L_{\rm d}^{\infty}} \\
	& \leq \sum_{k=0}^{K}{}^{\prime\prime}e_{u,k}^{(n-2)}\Delta x + \Delta tL\left\|\bm{\xi}_{1}^{\left( n - 2 + \frac{1}{2}\right)}\right\|_{L_{\rm d}^{\infty}} + \Delta tL\left\|\bm{\xi}_{1}^{\left( n - 1 + \frac{1}{2}\right)}\right\|_{L_{\rm d}^{\infty}} \\
	& \leq \cdots \\
	& \leq \sum_{k=0}^{K}{}^{\prime\prime}e_{u,k}^{(0)}\Delta x + \Delta tL\sum_{j=0}^{n-1}\left\|\bm{\xi}_{1}^{\left( j + \frac{1}{2}\right)}\right\|_{L_{\rm d}^{\infty}} \\
	& = \Delta tL\sum_{j=0}^{n-1}\left\|\bm{\xi}_{1}^{\left( j + \frac{1}{2}\right)}\right\|_{L_{\rm d}^{\infty}} \quad (n=1,\ldots, N), 
\end{align*}
where the last equality holds from $\bm{e}_{u}^{(0)} = \bm{0}$. 
Hence, from Lemma \ref{lem:3.2}, the above inequality, and \eqref{Derr_est}, we have 
\begin{align}
\left\|\bm{e}_{u}^{(n)}\right\|_{L_{\rm d}^{\infty}} 
	& \leq \frac{1}{L}\sum_{k=0}^{K}{}^{\prime\prime}e_{u,k}^{(n)}\Delta x + L^{\frac{1}{2}}\left\| D\bm{e}_{u}^{(n)}\right\| \nonumber\\
	& \leq \Delta t\sum_{j=0}^{n-1}\left\|\bm{\xi}_{1}^{\left(\! j + \frac{1}{2} \!\right)}\right\|_{L_{\rm d}^{\infty}} + L^{\frac{1}{2}}\left\| D\bm{e}_{u}^{(n)}\right\| \nonumber\\
	& \leq \Delta t\sum_{j=0}^{n-1}\left\|\bm{\xi}_{1}^{\left(\! j + \frac{1}{2} \!\right)}\right\|_{L_{\rm d}^{\infty}} \! + \! \left\{ \frac{L\exp(\tilde{C}_{4}T)}{1 - C_{4}B}\Delta t\sum_{j=1}^{n}R^{\left( n - j + \frac{1}{2} \right)} \right\}^{ \frac{1}{2} } \quad (n=1,\ldots, N). \label{err_Linf1}
\end{align}
Next, we estimate $\bm{\xi}_{i}^{(n + 1/2)} \ (i=1,2)$ and $\xi_{3,k}^{(n + 1/2)} \ (k=0,K)$.  
Let us define 
\begin{gather*}
M_{i,j}(v) := \max\left\{\left|\frac{\partial^{i+j}v}{\partial x^{i}\partial t^{j}}\right| ; (x,t) \in [0,L] \times [0,T]\right\} \quad {\rm for\ all\ } i,j \in \mathbb{Z}, \\
\tilde{M}_{i,j}(\tilde{v}) := \max\left\{\left|\frac{\partial^{i+j}\tilde{v} }{\partial x^{i}\partial t^{j}}\right| ; (x,t) \in [-\Delta x,L + \Delta x] \times [0,T]\right\} \quad {\rm for\ all\ } i,j \in \mathbb{Z}, \\
C_{F,i} := \max_{|\eta | \leq C_{2}}\left| F^{(i)}(\eta )\right| \quad (i=2,3,4). 
\end{gather*} 
Firstly, we consider $\xi_{3,0}$ and $\xi_{3,K}$. 
Applying the Taylor theorem to $\tilde{u}$ and using \eqref{util}, we obtain the following estimate (for details, see \cite{Oku}): 
\begin{equation*}
\left| \delta_{k}^{\langle 1 \rangle}\!\!\left(\!\frac{\tilde{u}_{k}^{(n+1)} \! + \! \tilde{u}_{k}^{(n)}}{2}\!\right) \! - \! \partial_{x}u_{k}^{\left(\! n + \frac{1}{2} \!\right)} \right| 
	\leq C\left\{ (\Delta x)^{2}M_{3,0}(u) \! + \! (\Delta t)^{2}M_{1,2}(u) \! + \! (\Delta t)^{2}(\Delta x)^{2}M_{3,2}(u) \right\} \quad (k=0,K). 
\end{equation*} 
As a remark, throughout this proof, we need the reader to keep in mind that the meaning of $C$ changes from line to line, whereas $C$ always denote those constants. 
From the assumption \eqref{dt_ass} on $\Delta t$, we obtain the following estimate: 
\begin{equation*}
\left| \delta_{k}^{\langle 1 \rangle}\!\!\left(\!\frac{\tilde{u}_{k}^{(n+1)} + \tilde{u}_{k}^{(n)}}{2}\!\right) \! - \! \partial_{x}u_{k}^{\left( n + \frac{1}{2}\right)} \right| 
	\leq C\left\{ (\Delta x)^{2}\!\left( M_{3,0}(u) + \frac{ \gamma^{2} }{ C_{3}^{4} }M_{3,2}(u) \!\right) \! + \! (\Delta t)^{2}M_{1,2}(u) \right\} \quad (k = 0,K). 
\end{equation*}
Furthermore, using the Taylor theorem, we have the following estimate: 
\begin{equation}
\left| \partial_{t}u_{k}^{\left( n + \frac{1}{2}\right)}  - \frac{u_{k}^{(n+1)} - u_{k}^{(n)}}{\Delta t} \right| \leq CM_{0,3}(u)(\Delta t)^{2} \quad (k=0,\ldots, K). \label{t_u_est}
\end{equation}
From the above, we estimate $\xi_{3,0}$ and $\xi_{3,K}$ as follows: 
\begin{align*}
\left| \xi_{3,k}^{\left( n + \frac{1}{2}\right)} \right| 
	& \leq \left| \partial_{t}u_{k}^{\left( n + \frac{1}{2}\right)}  - \frac{u_{k}^{(n+1)} - u_{k}^{(n)}}{\Delta t} \right| + \left| \delta_{k}^{\langle 1 \rangle}\!\!\left(\!\frac{\tilde{u}_{k}^{(n+1)} + \tilde{u}_{k}^{(n)}}{2}\!\right) - \partial_{x}u_{k}^{\left( n + \frac{1}{2}\right)} \right| \\
	& \leq C(\Delta x)^{2}\!\left( M_{3,0}(u) + \frac{\gamma^{2} }{C_{3}^{4} }M_{3,2}(u) \right) + C(\Delta t)^{2}\left( M_{0,3}(u) + M_{1,2}(u) \right) \quad (k = 0,K). 
\end{align*}
Next, we consider $\bm{\xi}_{1}$. 
For any $t \in [0,T]$ and $k=0,\ldots, K$, applying the Taylor theorem to $\tilde{p}$, there exists $\theta_{1} \in (0,1)$ such that  
\begin{align}
\frac{ \tilde{p}((k \! + \! 1)\Delta x, t) -2\tilde{p}(k\Delta x, t)  + \tilde{p}((k \! - \! 1)\Delta x, t) }{(\Delta x)^{2}} 
	& = \partial_{x}^{2}\tilde{p}( k\Delta x, t) \nonumber\\
	& \quad + \frac{(\Delta x)^{2}}{24}\!\left\{ \partial_{x}^{4}\tilde{p}( (k \! + \! \theta_{1})\Delta x, t) \! + \! \partial_{x}^{4}\tilde{p}( (k \! - \! \theta_{1})\Delta x, t) \right\}. \label{cd2}
\end{align}
Substituting $(n + 1/2)\Delta t$ into $t$ in \eqref{cd2}, we obtain 
\begin{equation}
\delta_{k}^{\langle 2 \rangle}\tilde{p}_{k}^{ \left(\! n + \frac{1}{2}\!\right) } - \partial_{x}^{2}p_{k}^{ \left(\! n + \frac{1}{2}\!\right)} 
	= \delta_{k}^{\langle 2 \rangle}\tilde{p}_{k}^{ \left(\! n + \frac{1}{2}\!\right) } - \partial_{x}^{2}\tilde{p}_{k}^{\left(\! n + \frac{1}{2}\!\right)} 
	= \frac{( \Delta x )^{2}}{24}\!\!\left(\! \partial_{x}^{4}\tilde{p}_{k+\theta_{1}}^{ \left(\! n + \frac{1}{2}\!\right) } + \partial_{x}^{4}\tilde{p}_{k-\theta_{1}}^{ \left(\! n + \frac{1}{2}\!\right) } \!\right) \quad (k=0,\ldots, K). \label{lap_err_1}
\end{equation}
Therefore, we have 
\begin{equation}
\left| \delta_{k}^{\langle 2 \rangle}\tilde{p}_{k}^{ \left(\! n + \frac{1}{2}\!\right) } - \partial_{x}^{2}p_{k}^{ \left(\! n + \frac{1}{2}\!\right)} \right| \leq \frac{(\Delta x)^{2}}{12}\tilde{M}_{4,0}(\tilde{p}) \quad (k=0, \ldots, K). \label{lap_p_est}
\end{equation}
Hence, using \eqref{t_u_est} and \eqref{lap_p_est}, the following estimate holds: 
\begin{equation*}
\left| \xi_{1,k}^{\left( n + \frac{1}{2}\right)} \right| 
	\leq \left| \partial_{t}u_{k}^{\left(\! n + \frac{1}{2} \!\right)}  - \frac{u_{k}^{(n+1)} - u_{k}^{(n)}}{\Delta t} \right| + \left| \delta_{k}^{\langle 2 \rangle}\tilde{p}_{k}^{ \left(\! n + \frac{1}{2}\!\right) } - \partial_{x}^{2}p_{k}^{ \left(\! n + \frac{1}{2}\!\right)} \right| 
	\leq CM_{0,3}(u)(\Delta t)^{2} + \frac{(\Delta x)^{2}}{12}\tilde{M}_{4,0}(\tilde{p})
\end{equation*}
for $k=0, \ldots, K$. 
Next, for $k=0, \ldots, K-1$, from \eqref{lap_err_1}, we have 
\begin{align}
\delta_{k}^{+} \! \left(\! \delta_{k}^{\langle 2 \rangle}\tilde{p}_{k}^{ \left(\! n + \frac{1}{2}\!\right) } \! - \! \partial_{x}^{2}p_{k}^{\left(\! n + \frac{1}{2}\!\right)} \!\right) 
	& = \frac{1}{\Delta x}\left\{ \left( \delta_{k}^{\langle 2 \rangle}\tilde{p}_{k+1}^{ \left(\! n + \frac{1}{2}\!\right) } \! - \! \partial_{x}^{2}p_{k+1}^{\left(\! n + \frac{1}{2}\!\right)} \right)  - \left( \delta_{k}^{\langle 2 \rangle}\tilde{p}_{k}^{ \left(\! n + \frac{1}{2}\!\right) } \! - \! \partial_{x}^{2}p_{k}^{\left(\! n + \frac{1}{2}\!\right)} \right) \right\} \nonumber\\
	& = \frac{1}{\Delta x}\left\{ \frac{(\Delta x)^{2}}{24}\!\!\left(\! \partial_{x}^{4}\tilde{p}_{k+1+\theta_{1}}^{ \left(\! n + \frac{1}{2}\!\right) } + \partial_{x}^{4}\tilde{p}_{k+1-\theta_{1}}^{ \left(\! n + \frac{1}{2}\!\right) } \!\right) - \frac{(\Delta x)^{2}}{24}\!\!\left(\! \partial_{x}^{4}\tilde{p}_{k+\theta_{1}}^{ \left(\! n + \frac{1}{2}\!\right) } + \partial_{x}^{4}\tilde{p}_{k-\theta_{1}}^{ \left(\! n + \frac{1}{2}\!\right) } \!\right) \right\} \nonumber\\
	& = \frac{(\Delta x)^{2}}{24} \frac{ \partial_{x}^{4}\tilde{p}_{k+1+\theta_{1}}^{ \left(\! n + \frac{1}{2}\!\right) } - \partial_{x}^{4}\tilde{p}_{k+\theta_{1}}^{ \left(\! n + \frac{1}{2}\!\right) } }{\Delta x} + \frac{(\Delta x)^{2}}{24}\frac{ \partial_{x}^{4}\tilde{p}_{k+1-\theta_{1}}^{ \left(\! n + \frac{1}{2}\!\right) } - \partial_{x}^{4}\tilde{p}_{k-\theta_{1}}^{ \left(\! n + \frac{1}{2}\!\right) }}{\Delta x}. \label{D_lap}
\end{align}
Since $\tilde{p}$ satisfies $\tilde{p}(\cdot, t) \in C^{5}([0,L])$ for any fixed $t \in [0,T]$, applying the mean value theorem to $\partial_{x}^{4}\tilde{p}(\cdot, t)$ and using \eqref{D_lap}, we obtain 
\begin{equation*}
\left| \delta_{k}^{+} \left( \delta_{k}^{\langle 2 \rangle}\tilde{p}_{k}^{ \left(\! n + \frac{1}{2}\!\right) } \! - \! \partial_{x}^{2}p_{k}^{\left(\! n + \frac{1}{2}\!\right)} \right) \right| 
	\leq \frac{(\Delta x)^{2}}{12}\tilde{M}_{5,0}(\tilde{p}) \quad (k=0, \ldots, K-1). 
\end{equation*}
Besides, applying the Taylor theorem to $u$ and using the mean value theorem, we have 
\begin{equation*}
\left| \delta_{k}^{+}\left(\partial_{t}u_{k}^{\left( n + \frac{1}{2}\right)}  - \frac{u_{k}^{(n+1)} - u_{k}^{(n)}}{\Delta t}\right) \right| \leq CM_{1,3}(u)(\Delta t)^{2} \quad (k=0,\ldots, K-1). 
\end{equation*}
Hence, we have the following estimate: 
\begin{equation*}
\left| \delta_{k}^{+}\xi_{1,k}^{\left(\! n + \frac{1}{2} \!\right)} \!\right| 
	\leq \left| \delta_{k}^{+}\!\!\left(\! \partial_{t}u_{k}^{\left(\! n + \frac{1}{2} \!\right)} \! - \! \frac{u_{k}^{(\! n+1\! )} \! - \! u_{k}^{(\! n\! )}}{\Delta t} \!\right) \!\right| + \left| \delta_{k}^{+} \!\! \left(\! \delta_{k}^{\langle 2 \rangle}\tilde{p}_{k}^{ \left(\! n + \frac{1}{2}\!\right) } \! - \! \partial_{x}^{2}p_{k}^{\left(\! n + \frac{1}{2}\!\right)} \!\right) \!\right| 
	\leq CM_{1,3}(u)(\Delta t)^{2} + \frac{(\Delta x)^{2}}{12}\tilde{M}_{5,0}(\tilde{p})
\end{equation*}
for $k=0, \ldots, K-1$. 
Similarly, from the Taylor theorem and the mean value theorem, we see that 
\begin{gather}
\left| \delta_{k}^{+}\!\left\{\! \partial_{x}^{2}u_{k}^{\left( n + \frac{1}{2}\right)} - \delta_{k}^{\langle 2 \rangle}\!\!\left( \frac{ \tilde{u}_{k}^{ (n + 1) } + \tilde{u}_{k}^{ (n) } }{2}\right) \!\right\} \right| \leq C (M_{5,0}(u) (\Delta x)^{2} + M_{3,2}(u) (\Delta t)^{2} )\quad (k=0,\ldots, K-1), \label{lap_u}\\
\begin{split}
\left| \delta_{k}^{+}\!\left(\! \frac{dF}{d( u_{k}^{(n+1)},u_{k}^{(n)})} - F'( u_{k}^{\left( n + \frac{1}{2}\right)} ) \!\right) \right| 
	& \leq C\left\{ C_{F,2}M_{1,2}(u) + C_{F,3}\left( M_{1,1}(u)M_{0,1}(u) \! + \! M_{0,2}(u) M_{1,0}(u) \right) \right. \\
	& \quad \left. + C_{F,4}M_{1,0}(u)\left( M_{0,1}(u) \right)^{2} \right\}(\Delta t)^{2} \quad (k=0,\ldots, K-1). 
\end{split} \nonumber
\end{gather}
For detail about the above estimate \eqref{lap_u}, see Lemma \ref{Ap:H} in Appendix C. 
From the regularity assumption of the solution $u$ and the potential $F$, we see that $C_{F,i} \ (i=2,3,4)$, $M_{i,j}(u) \ (i,j \in \mathbb{Z}, 0 \leq i+j \leq 5)$, and $\tilde{M}_{i,0}(\tilde{p}) \ (i=4,5)$ are bounded. 
Thus, we obtain the following estimates: 
\begin{gather}
\left| \xi_{1,k}^{\left(\! n + \frac{1}{2} \!\right)} \right| \leq C_{5}( (\Delta x)^{2} \! + \! (\Delta t)^{2} ) \quad (k=0, \ldots, K), \nonumber\\
\left| \xi_{3,k}^{\left(\! n + \frac{1}{2} \!\right)} \right| \leq C_{5}( (\Delta x)^{2} + (\Delta t)^{2} ) \quad (k=0,K), \label{xi3_inf}\\
\left| \delta_{k}^{+}\xi_{i,k}^{\left(\! n + \frac{1}{2} \!\right)} \right| \leq C_{5}( (\Delta x)^{2} \! + \! (\Delta t)^{2} ) \quad (k=0, \ldots, K-1,\ i=1,2), \nonumber
\end{gather}
where $C_{5}$ is a constant independent of $\Delta x$ and $\Delta t$. 
Therefore, the following estimates hold: 
\begin{gather} 
\left\|\bm{\xi}_{1}^{\left( n + \frac{1}{2}\right)}\right\|_{L_{\rm d}^{\infty}} \leq C_{5}( (\Delta x)^{2} + (\Delta t)^{2} ), \label{xi1_inf}\\
\left\| D\bm{\xi}_{i}^{\left(\! n + \frac{1}{2} \!\right)}\right\|^{2} = \sum_{k=0}^{K-1}\left| \delta_{k}^{+}\xi_{i,k}^{\left(\! n + \frac{1}{2} \!\right)} \right|^{2}\Delta x \leq C_{5}^{2}( (\Delta x)^{2} \! + \! (\Delta t)^{2} )^{2}\sum_{k=0}^{K-1}\Delta x = LC_{5}^{2}( (\Delta x)^{2} \! + \! (\Delta t)^{2} )^{2} \quad (i=1,2). \label{xi12_D}
\end{gather}
Furthermore, using \eqref{xi1_inf}, we obtain 
\begin{equation}
\Delta t\sum_{j=0}^{n-1}\left\|\bm{\xi}_{1}^{\left( j + \frac{1}{2}\right)}\right\|_{L_{\rm d}^{\infty}} \leq C_{5}( (\Delta x)^{2} \! + \! (\Delta t)^{2} ) n \Delta t \leq C_{5}T ( (\Delta x)^{2} \! + \! (\Delta t)^{2} ) \quad (n=1, \ldots, N). \label{sum_xi1_inf}
\end{equation}
Hence, from \eqref{Rn}, \eqref{xi3_inf}--\eqref{sum_xi1_inf}, it holds that 
\begin{align*}
R^{\left( n + \frac{1}{2}\right)} 
	& = \frac{1}{2\gamma}\left( 1 + \frac{2C_{3}^{2}}{\varepsilon \gamma} \right)\left( C_{1}C_{F,3}\Delta t\sum_{j=0}^{n}\left\|\bm{\xi}_{1}^{\left( j + \frac{1}{2}\right)}\right\|_{L_{\rm d}^{\infty}} + \left\| D\bm{\xi}_{2}^{\left( n + \frac{1}{2}\right)}\right\| \right)^{2} + \frac{1}{\varepsilon}\left\| D\bm{\xi}_{1}^{\left( n + \frac{1}{2}\right)}\right\|^{2}\\
	& \quad + \left|\xi_{1,0}^{\left( n + \frac{1}{2}\right)}\right|^{2} + \left|\xi_{1,K}^{\left( n + \frac{1}{2}\right)}\right|^{2} + \left|\xi_{3,0}^{\left( n + \frac{1}{2}\right)}\right|^{2} + \left|\xi_{3,K}^{\left( n + \frac{1}{2}\right)}\right|^{2} \\
	& \leq \frac{1}{2\gamma}\left( 1 + \frac{2C_{3}^{2}}{\varepsilon \gamma} \right)\left\{ \left( C_{1}C_{F,3}C_{5}T\right) ((\Delta x)^{2} + (\Delta t)^{2}) + L^{\frac{1}{2} }C_{5}((\Delta x)^{2} + (\Delta t)^{2})\right\}^{2} \\
	& \quad + \frac{C_{5}^{2}L}{\varepsilon}((\Delta x)^{2} + (\Delta t)^{2})^{2} + 4C_{5}^{2}((\Delta x)^{2} + (\Delta t)^{2})^{2} \\
	& \leq C_{5}^{2}\left\{ \frac{1}{2\gamma}\left( 1 + \frac{2C_{3}^{2}}{\varepsilon \gamma} \right)\left( C_{1}C_{F,3}T + L^{\frac{1}{2} } \right)^{2} + \frac{L}{\varepsilon} + 4 \right\} ((\Delta x)^{2} + (\Delta t)^{2})^{2} \quad (n=0,\ldots, N-1). 
\end{align*}
Now, let us define the constant $C_{6}$ as follows: 
\begin{equation*}
C_{6} := C_{5}^{2}\left\{ \frac{1}{2\gamma}\left( 1 + \frac{2C_{3}^{2}}{\varepsilon \gamma} \right)\left( C_{1}C_{F,3}T + L^{\frac{1}{2} } \right)^{2} + \frac{L}{\varepsilon} + 4 \right\}. \vspace{-1mm}
\end{equation*}
Then, we obtain 
\begin{equation}
\frac{L\exp(\tilde{C}_{4}T)}{1 - C_{4}B}\Delta t\sum_{j=1}^{n}\! R^{\left(\! n - j + \frac{1}{2} \!\right)} 
	\leq \frac{C_{6}L\exp(\tilde{C}_{4}T)}{1 - C_{4}B} n\Delta t ((\Delta x)^{2} \! + \! (\Delta t)^{2})^{2} 
	\leq \frac{C_{6}LT\exp(\tilde{C}_{4}T)}{1 - C_{4}B} ((\Delta x)^{2} \! + \! (\Delta t)^{2})^{2} \label{R} \vspace{-1mm}
\end{equation}
for $n=1,\ldots, N$.
From the above, using \eqref{err_Linf1}, \eqref{sum_xi1_inf}, and \eqref{R}, we conclude that \vspace{-1mm}
\begin{equation}
\left\|\bm{e}_{u}^{(n)}\right\|_{L_{\rm d}^{\infty}} \! \leq \! C_{5}T((\Delta x)^{2} + (\Delta t)^{2}) + \left\{\!\! \frac{C_{6}LT\!\exp(\tilde{C}_{4}T)}{1 - C_{4}B} \!\right\}^{\!\! \frac{1}{2}} \!\! ((\Delta x)^{2} + (\Delta t)^{2}) \! = \! C_{7}((\Delta x)^{2} + (\Delta t)^{2}) \quad (n \! = \! 1,\ldots, N), \label{err_est} \vspace{-1mm} 
\end{equation} 
where the constant $C_{7}$ is defined by \vspace{-1mm}
\begin{equation*}
C_{7} := C_{7}(B) := C_{5}T + \left\{ \frac{C_{6}LT\exp(\tilde{C}_{4}T)}{1 - C_{4}B} \right\}^{\frac{1}{2}}. \vspace{-1mm} 
\end{equation*}
{\bf Step2.} 
It holds from the triangle inequality that 
\begin{equation}
\left\| (\! \varPi_{\! \Delta x, \Delta t}U) \! (\cdot, t) \! - \! u(\cdot, t) \right\|_{L^{\! \infty}\! (0,L)} \! 
	\leq \! \left\| (\! \varPi_{\! \Delta x, \Delta t}U) \! (\cdot, t) \! - \! (\! \varPi_{\! \Delta x, \Delta t}u) \! (\cdot, t) \right\|_{L^{\! \infty}\! (0,L)} + \left\| (\! \varPi_{\! \Delta x, \Delta t}u) \! (\cdot, t) \! - \! u(\cdot, t) \right\|_{L^{\! \infty}\! (0,L)} \label{all_err_est_1}
\end{equation} 
for all $t \in [0,T]$. 
Firstly, we estimate the first term on the right-hand side of \eqref{all_err_est_1}. 
For $t \in [n\Delta t, (n+1)\Delta t]$, $n=0,1,\ldots,N-1$, 
there exists $\eta \in [0,1]$ satisfying $t = (n+\eta)\Delta t$. 
Hence, using \eqref{err_est} and the following inequality $\| \varPi_{\Delta x}f \|_{L^{\infty}(0,L)} \leq \| \bm{f}\|_{L_{\rm d}^{\infty}}$ for all $\{f_{k}\}_{k=0}^{K} \in \mathbb{R}^{K+1}$, 
we obtain 
\begin{align}
& \left\| (\varPi_{\Delta x, \Delta t}U)(\cdot, t) - (\varPi_{\Delta x, \Delta t}u)(\cdot, t) \right\|_{L^{\infty}(0,L)} \nonumber\\
	& = \left\| \left\{ \left( n+1 - \frac{(n+\eta)\Delta t}{\Delta t}\right) \left(\varPi_{\Delta x}U^{(n)}\right) + \left( \frac{(n+\eta)\Delta t}{\Delta t} - n\right) \left(\varPi_{\Delta x}U^{(n+1)}\right)\right\} \right. \nonumber\\
	& \quad \left. - \left\{ \left( n+1 - \frac{(n+\eta)\Delta t}{\Delta t}\right) \left(\varPi_{\Delta x}u^{(n)}\right) + \left( \frac{(n+\eta)\Delta t}{\Delta t} - n\right) \left(\varPi_{\Delta x}u^{(n+1)}\right)\right\} \right\|_{L^{\infty}(0,L)} \nonumber\\
	& = \left\| ( 1 - \eta ) \left\{ \left(\varPi_{\Delta x}U^{(n)}\right) - \left(\varPi_{\Delta x}u^{(n)}\right) \right\} + \eta \left\{ \left(\varPi_{\Delta x}U^{(n+1)}\right) - \left(\varPi_{\Delta x}u^{(n+1)}\right) \right\} \right\|_{L^{\infty}(0,L)} \nonumber\\
	& \leq ( 1 - \eta )\left\|  \varPi_{\Delta x}\left( U^{(n)} - u^{(n)} \right) \right\|_{L^{\infty}(0,L)} + \eta\left\| \varPi_{\Delta x}\left( U^{(n+1)} - u^{(n+1)} \right) \right\|_{L^{\infty}(0,L)} \nonumber\\
	& \leq ( 1 - \eta )\left\| \bm{U}^{(n)} - \bm{u}^{(n)} \right\|_{L_{\rm d}^{\infty}} + \eta\left\| \bm{U}^{(n+1)} - \bm{u}^{(n+1)} \right\|_{L_{\rm d}^{\infty}} \nonumber\\
	& \leq C_{7}\left( (\Delta x)^{2} + (\Delta t)^{2} \right). \label{all_est_1}
\end{align}
Next, we estimate the second term on the right-hand side of \eqref{all_err_est_1}. 
For any fixed $(x,t) \in [0,L] \times [0,T]$, 
there exists $k_{0} \in \{0,1, \ldots, K-1\}$ satisfying $x \in [k_{0}\Delta x, (k_{0}+1)\Delta x]$, 
and there exists $n_{0} \in \{0,1, \ldots, N-1\}$ satisfying $t \in [n_{0}\Delta t, (n_{0}+1)\Delta t]$. 
Hence, we have 
\begin{align*}
(\varPi_{\Delta x, \Delta t}u)(x, t) - u(x,t) 
	& = \left( k_{0}+1 - \frac{x}{\Delta x}\right) \left( n_{0}+1 - \frac{t}{\Delta t}\right) \{ u(k_{0}\Delta x, n_{0}\Delta t) - u(x,t) \} \\
	& \quad + \left( k_{0}+1 - \frac{x}{\Delta x}\right) \left( \frac{t}{\Delta t} - n_{0} \right) \{ u(k_{0}\Delta x, (n_{0}+1)\Delta t) - u(x,t) \} \\
	& \quad + \left( \frac{x}{\Delta x} - k_{0} \right) \left( n_{0}+1 - \frac{t}{\Delta t}\right) \{ u((k_{0}+1)\Delta x, n_{0}\Delta t) - u(x,t) \} \\
	& \quad + \left( \frac{x}{\Delta x} - k_{0}\right)\left( \frac{t}{\Delta t} - n_{0}\right) \{ u((k_{0}+1)\Delta x, (n_{0}+1)\Delta t) - u(x,t) \}. 
\end{align*}
Let $C_{8} := (1/8)(M_{2,0}(u) + M_{0,2}(u))$. 
Then, using the Taylor theorem, we obtain 
\begin{equation*}
\left| (\varPi_{\Delta x, \Delta t}u)(x, t) - u(x,t) \right| \leq C_{8}( (\Delta x)^{2} + (\Delta t)^{2} ). 
\end{equation*}
Therefore, we estimate the second term on the right-hand side of \eqref{all_err_est_1} as follows: 
\begin{equation}
\left\| (\varPi_{\Delta x, \Delta t}u)(\cdot, t) - u(\cdot, t) \right\|_{L^{\infty}(0,L)} 
	\leq C_{8}( (\Delta x)^{2} + (\Delta t)^{2} ) \quad \mbox{for\ all\ } t \in [0,T]. \label{all_est_2}
\end{equation}
Hence using \eqref{all_err_est_1}--\eqref{all_est_2}, we conclude that 
\begin{equation*}
\left\| (\varPi_{\Delta x, \Delta t}U)(\cdot, t) - u(\cdot, t) \right\|_{L^{\infty}(0,L)} \leq (C_{7} + C_{8}) ( (\Delta x)^{2} + (\Delta t)^{2} ) \quad \mbox{for\ all\ } t \in [0,T]. 
\end{equation*}
This completes the proof. \hfill $\Box$\\
\end{proof}

\section{Computation examples}
In this section, we demonstrate through computation examples that the numerical solution of our proposed scheme is efficient and that the scheme inherits the conservative property and the dissipative property from the original problem in a discrete sense. 
Also, we compare our scheme with the previous structure-preserving scheme proposed by Fukao--Yoshikawa--Wada \cite{A}. 
Throughout the computation examples, we consider the double-well potential $F(s) = (1/4)s^{4} - (1/2)s^{2}$. 
In the same manner as Section 5, we use the following notation $T = N\Delta t$. 

\subsection{Computation example 1}
As the initial condition, we consider 
\begin{equation*}
u(x,0) = u_{0}(x) = 0.01\cos\left( \frac{\pi}{2}x \right).  
\end{equation*}
We choose $N=20000$ and fix $T=400$ so that $\Delta t = 1/50$. 
Also, we choose $K=40$ and fix $L=20$ so that $\Delta x = 1/2$. 
Besides, we fix the parameter $\gamma = 2.0$. 
Figure \ref{fig:Our1} shows the time development of the solution obtained by our proposed structure-preserving scheme. 
Figure \ref{fig:FYW1} shows the one by the previous structure-preserving scheme proposed by Fukao--Yoshikawa--Wada. 
\begin{figure}[H]
 \begin{minipage}{0.495\hsize}
  \begin{center}
   \includegraphics[width=78mm]{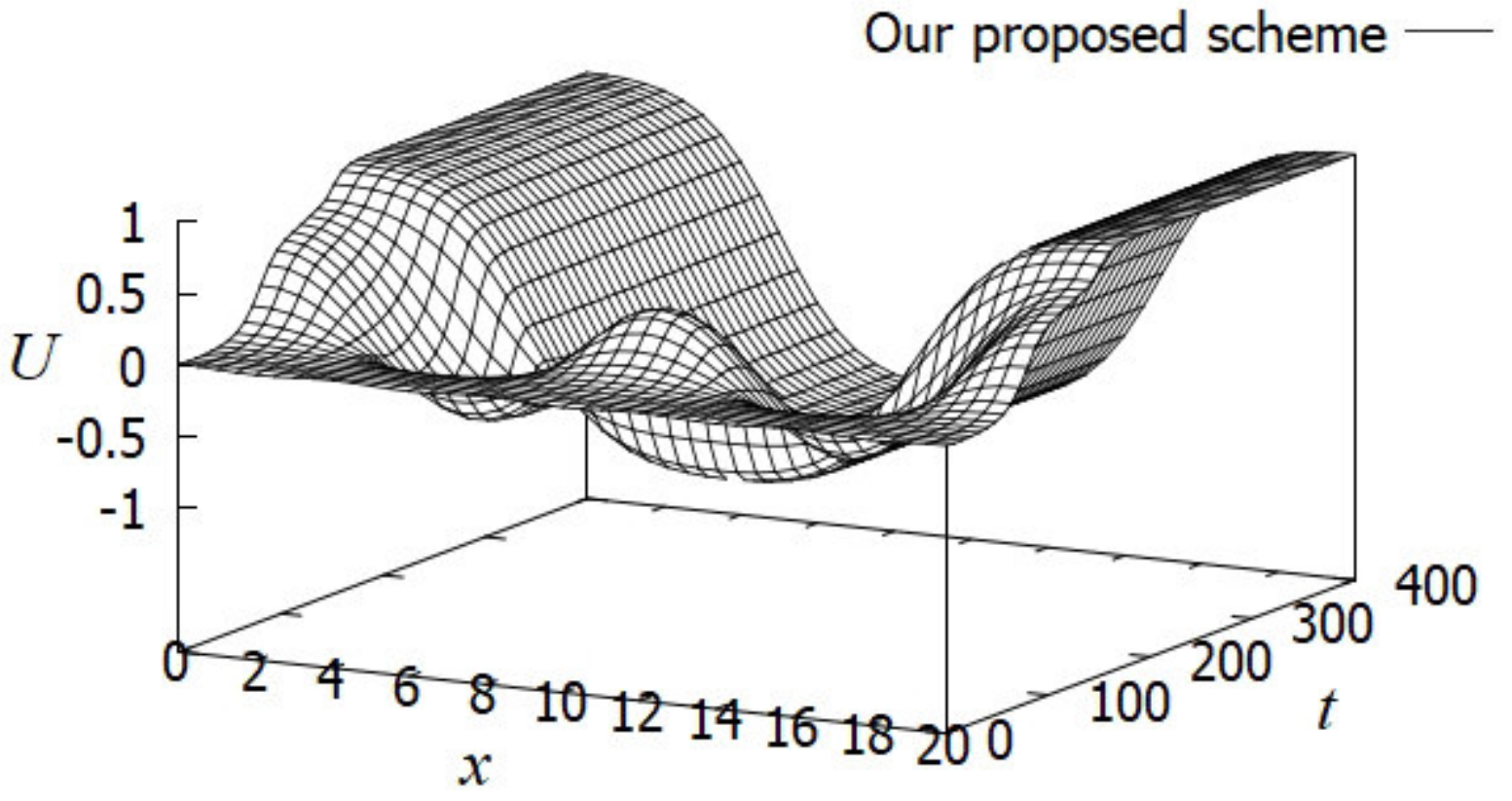} \vspace{-2mm}
  \end{center}
  \caption{Numerical solution by our scheme with $\Delta x = 1/2$.}
  \label{fig:Our1}
 \end{minipage}
 \begin{minipage}{0.495\hsize}
  \begin{center}
   \includegraphics[width=78mm]{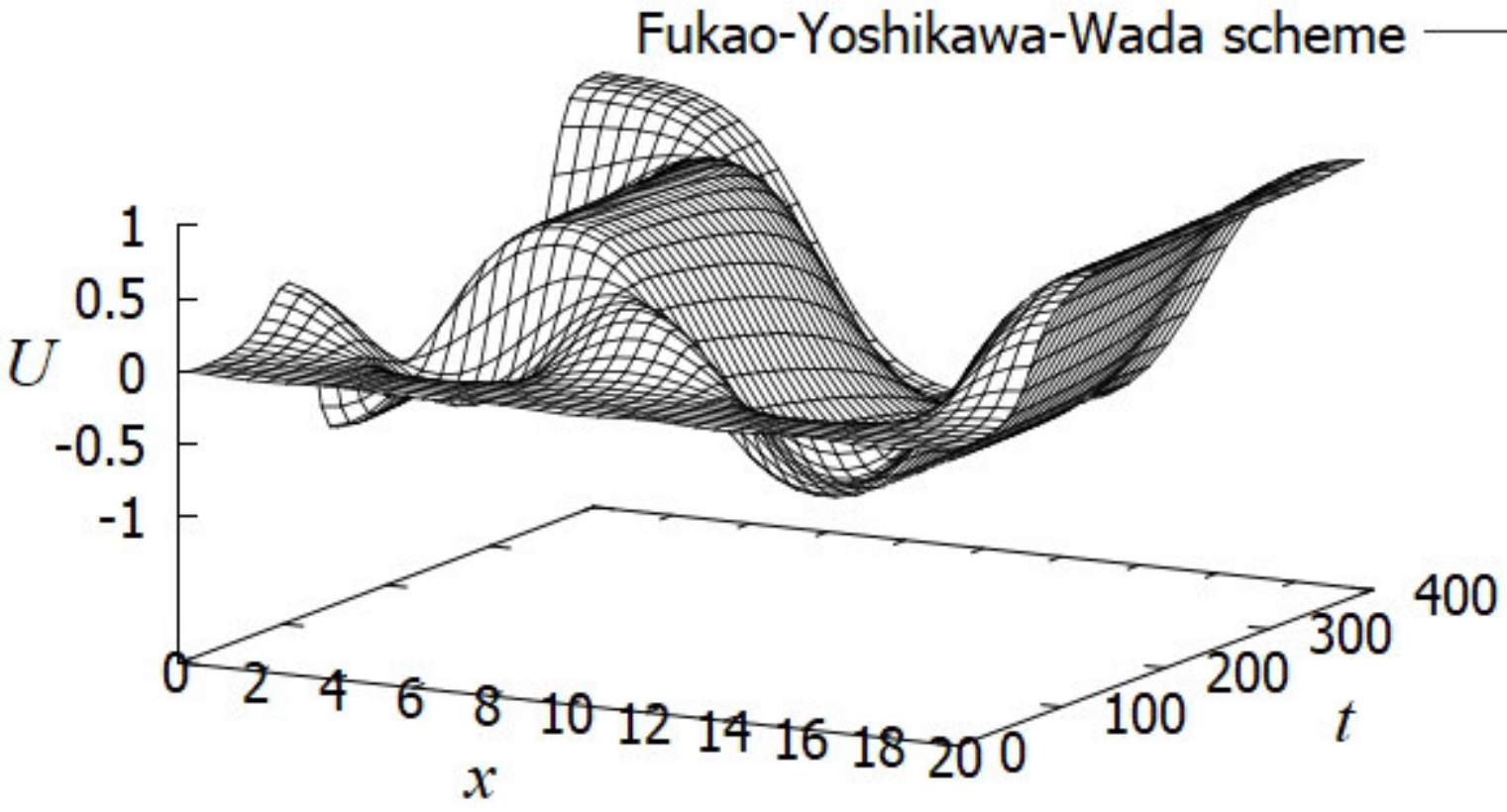} \vspace{-2mm}
  \end{center}
  \caption{Numerical solution by Fukao--Yoshikawa--Wada scheme with $\Delta x = 1/2$.}
  \label{fig:FYW1}
 \end{minipage}
\end{figure}
\noindent 
The behavior of the solution obtained by our scheme is different from the one by the Fukao--Yoshikawa--Wada scheme.
In order to analyze the difference in these results, we refine the space mesh size. 
Specifically, in the following results, we choose $K=800$ so that $\Delta x = 1/40$. 
In this case, the result of the Fukao--Yoshikawa--Wada scheme improves. 
Figure \ref{fig:Our2} shows the time development of the solution obtained by our scheme. 
Also, Figure \ref{fig:FYW2} shows the one by the Fukao--Yoshikawa--Wada scheme. 
Both results are similar to the result obtained by our scheme with $\Delta x = 1/2$. 
Note that we can obtain a valid numerical solution by our proposed scheme even when the space mesh size $\Delta x$ is coarse. 
\begin{figure}[H]
 \begin{minipage}{0.495\hsize}
  \begin{center}
   \includegraphics[width=78mm]{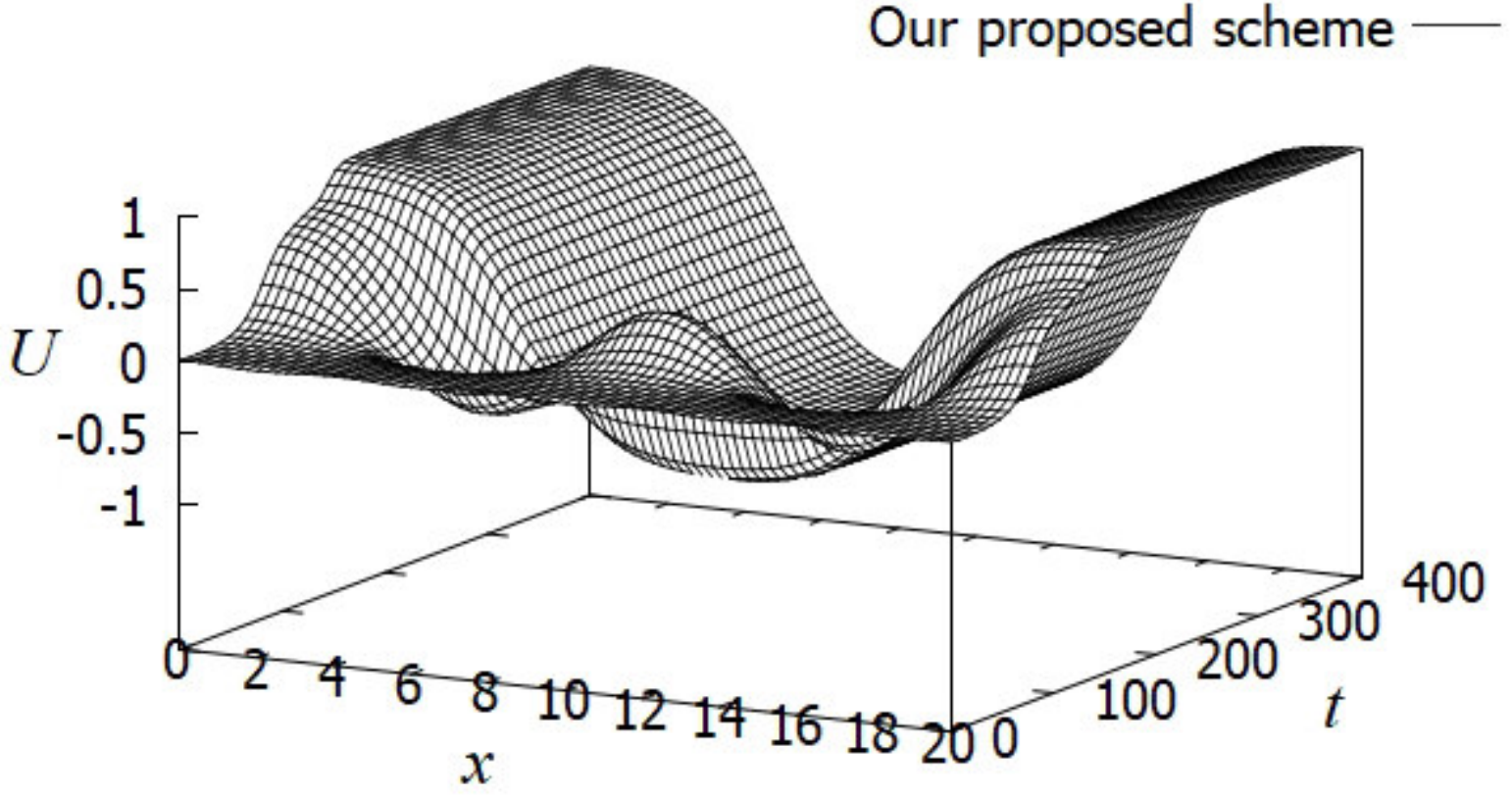} \vspace{-2mm}
  \end{center}
  \caption{Numerical solution by our scheme with $\Delta x = 1/40$.}
  \label{fig:Our2}
 \end{minipage}
 \begin{minipage}{0.495\hsize}
  \begin{center}
   \includegraphics[width=78mm]{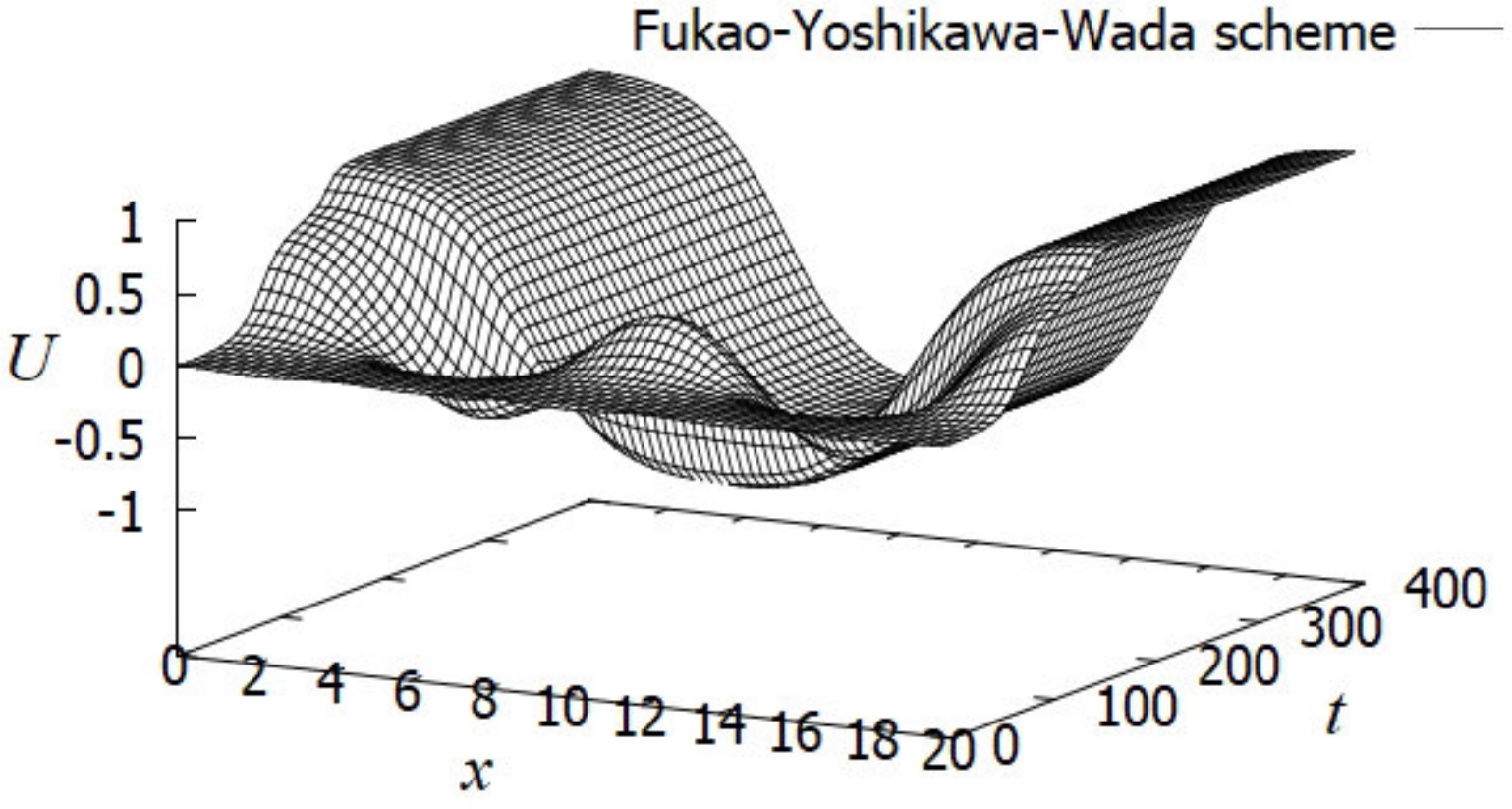} \vspace{-2mm}
  \end{center}
  \caption{Numerical solution by Fukao--Yoshikawa--Wada scheme with $\Delta x = 1/40$.}
  \label{fig:FYW2}
 \end{minipage}
\end{figure}
Next, we confirm the conservative property and the dissipative property. 
Figure \ref{fig:mass1} shows the time development of $M_{\rm d}(\bm{U}^{(n)})$ obtained by our scheme with $\Delta x = 1/40$. 
Figure \ref{fig:energy1} shows the time development of $E_{\rm d}^{(n)} - J_{\rm d}(\bm{U}^{(0)})$ obtained by our scheme with $\Delta x = 1/40$, where 
\begin{equation*}
E_{\rm d}^{(n)} 
	:= J_{\rm d}( \bm{U}^{(n)} )  +  \sum_{l=0}^{n-1}\left\{ \gamma \left| \delta_{n}^{+}U_{0}^{( l )} \right|^{2}  + \gamma \left| \delta_{n}^{+}U_{K}^{( l )} \right|^{2} + \sum_{k=0}^{K-1}\left| \delta_{k}^{+} P_{k}^{( l )} \right|^{2} \Delta x \right\}\Delta t \quad (n=1,2, \ldots). 
\end{equation*}
We remark that the following equality holds from Theorem \ref{thm:2.1} (the dissipative property): \vspace{-1mm}
\begin{equation*}
E_{\rm d}^{(n)} = J_{\rm d}(\bm{U}^{(0)}) \quad (n=1,2,\ldots). \vspace{-5mm}
\end{equation*}
 
\begin{figure}[H]
 \begin{minipage}{0.495\hsize}
  \begin{center}
   \includegraphics[width=78mm]{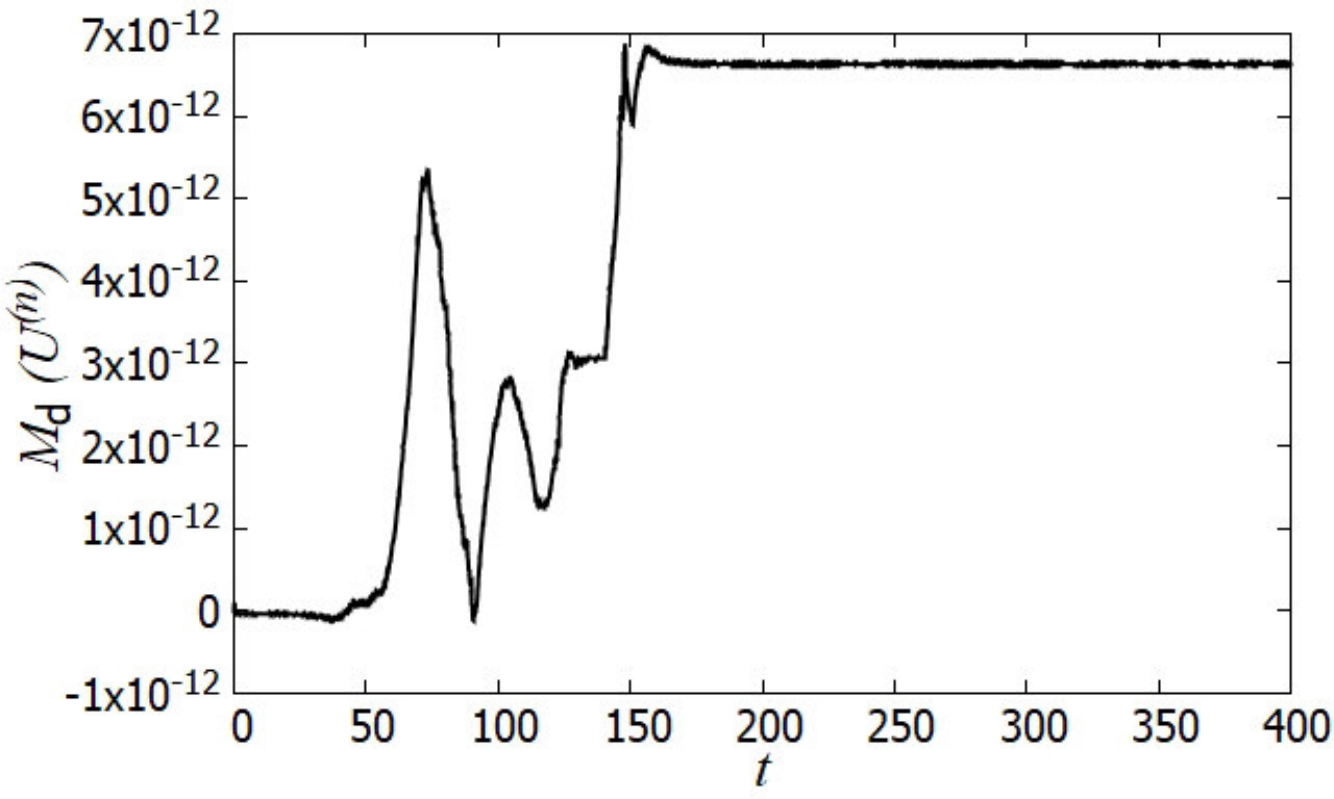} \vspace{-5mm} 
  \end{center}
  \caption{Time development of $M_{\rm d}(\bm{U}^{(n)})$ obtained by our scheme with $\Delta x = 1/40$. $M_{\rm d}(\bm{U}^{(n)})$ does not change by about $11$ orders of magnitude.}
  \label{fig:mass1}
 \end{minipage}
 \begin{minipage}{0.495\hsize}
  \begin{center}
   \includegraphics[width=78mm]{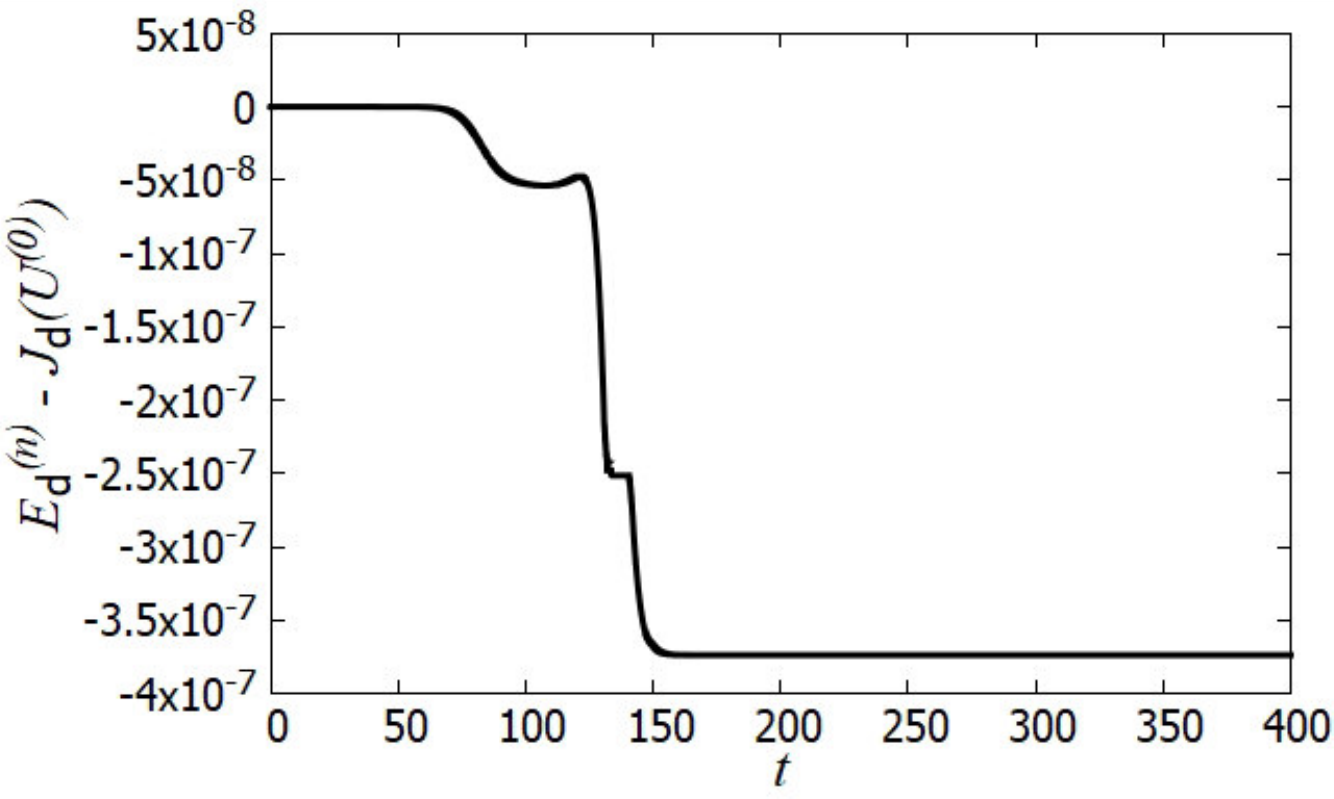} \vspace{-5mm} 
  \end{center}
  \caption{Time development of $E_{\rm d}^{(n)} - J_{\rm d}(\bm{U}^{(0)})$ obtained by our scheme with $\Delta x = 1/40$. $E_{\rm d}^{(n)}$ does not change by about $6$ orders of magnitude.}
  \label{fig:energy1}
 \end{minipage} \vspace{-2mm}
\end{figure}
\noindent 
These graphs show that the quantities $M_{\rm d}(\bm{U}^{(n)})$ and $E_{\rm d}^{(n)}$ are conserved numerically. 
More precisely, $M_{\rm d}(\bm{U}^{(n)})$ does not change by about $11$ orders of magnitude, and $E_{\rm d}^{(n)}$ does not change by about $6$ orders of magnitude. 

\subsection{Computation example 2}
As the initial condition, we consider \vspace{-0.5mm}
\begin{equation*}
u(x,0) = u_{0}(x) = 0.01\sin( 2 \pi x ) + 0.001\cos( 4 \pi x ) + 0.006\sin( 4 \pi x ) + 0.002\cos( 10 \pi x ). \vspace{-0.5mm}
\end{equation*}
We choose $N=50000$ and fix $T=1000$ so that $\Delta t = 1/50$. 
Also, we choose $K=250$ and fix $L=10$ so that $\Delta x = 1/25$. 
In addition, we fix the parameter $\gamma = 1.0$. 
Figure \ref{fig:Our3} shows the time development of the solution obtained by our scheme. 
Figure \ref{fig:FYW3} shows the one by the Fukao--Yoshikawa--Wada scheme. \vspace{-2mm}
\begin{figure}[H]
 \begin{minipage}{0.495\hsize}
  \begin{center}
   \includegraphics[width=78mm]{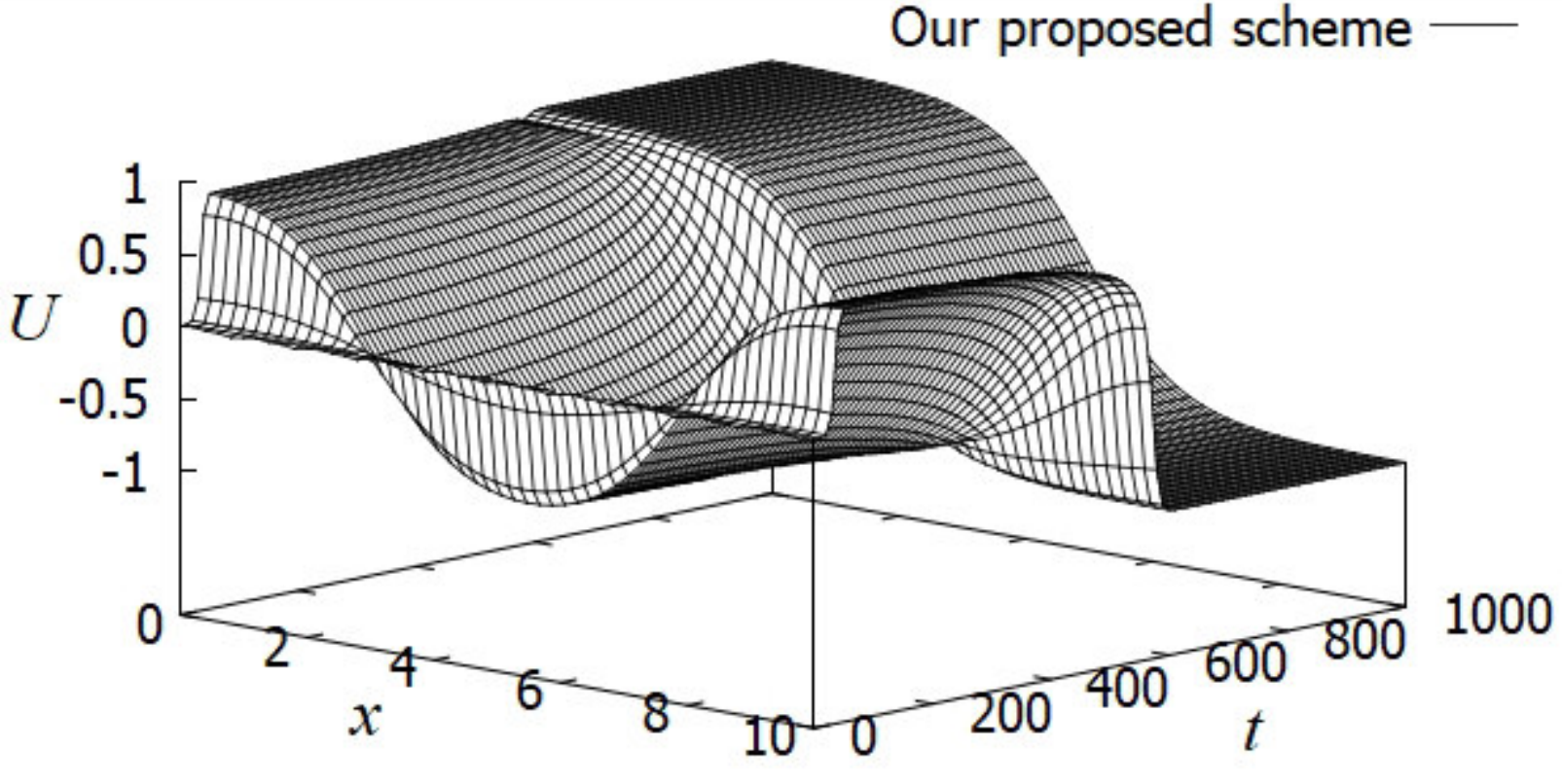} \vspace{-4mm}
  \end{center}
  \caption{Numerical solution by our scheme with $\Delta x = 1/25$.}
  \label{fig:Our3}
 \end{minipage}
 \begin{minipage}{0.495\hsize}
  \begin{center}
   \includegraphics[width=78mm]{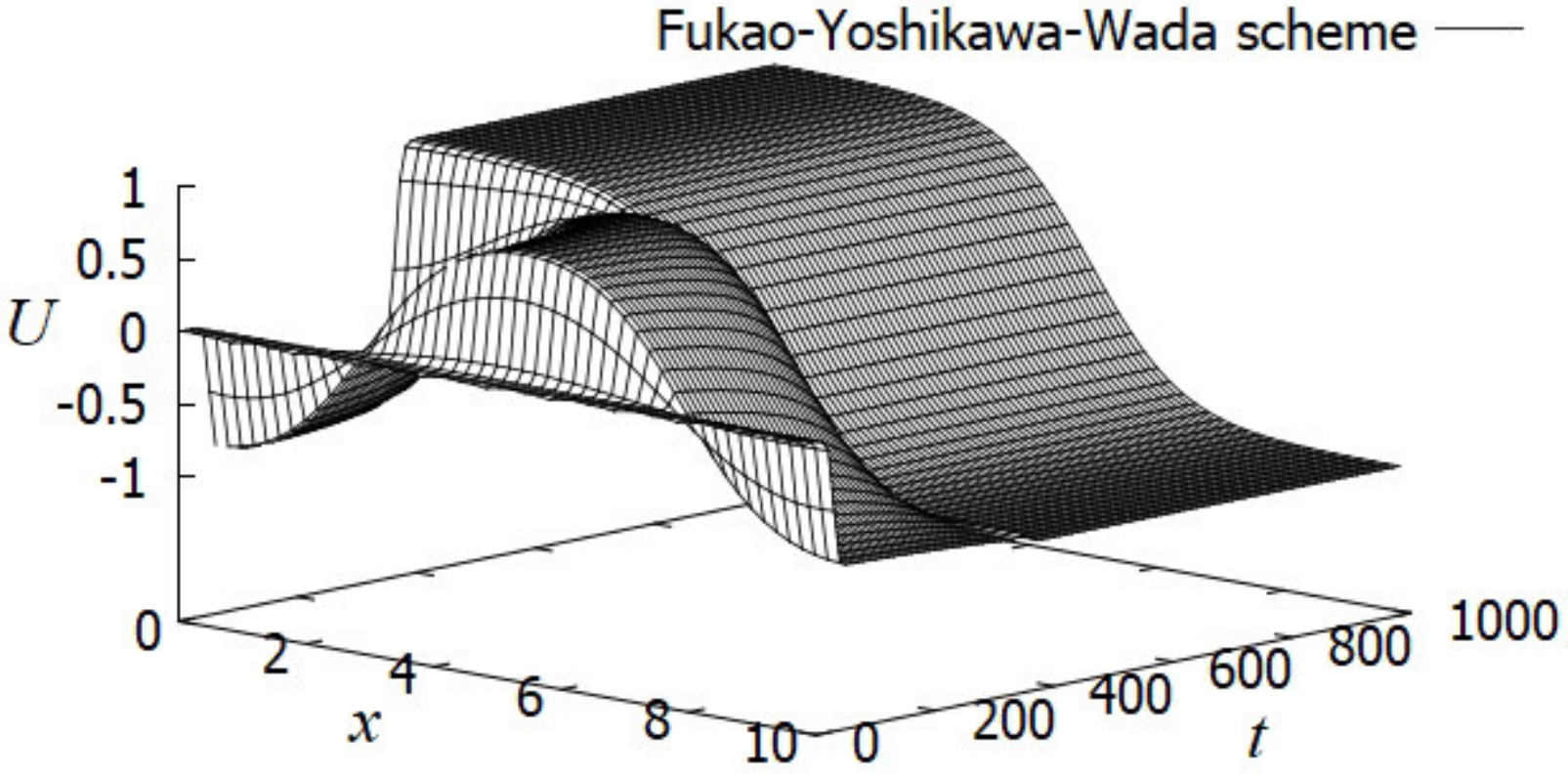} \vspace{-4mm} 
  \end{center}
  \caption{Numerical solution by Fukao--Yoshikawa--Wada scheme with $\Delta x = 1/25$.}
  \label{fig:FYW3}
 \end{minipage} \vspace{-2mm}
\end{figure}
\noindent
The behavior of the solution obtained by our scheme ranging from $t = 0$ to $t = 600$ is different from the one by the Fukao--Yoshikawa--Wada scheme.
In order to analyze the difference in these results, we refine the space mesh size. 
To be specific, in the following results, we choose $K=500$ so that $\Delta x = 1/50$. 
In this case, the result of the Fukao--Yoshikawa--Wada scheme improves, too. 
Also, we remark that we can obtain a valid numerical solution by our proposed scheme even when the space mesh size $\Delta x$ is coarse. \vspace{-2mm}
\begin{figure}[H]
 \begin{minipage}{0.495\hsize}
  \begin{center}
   \includegraphics[width=78mm]{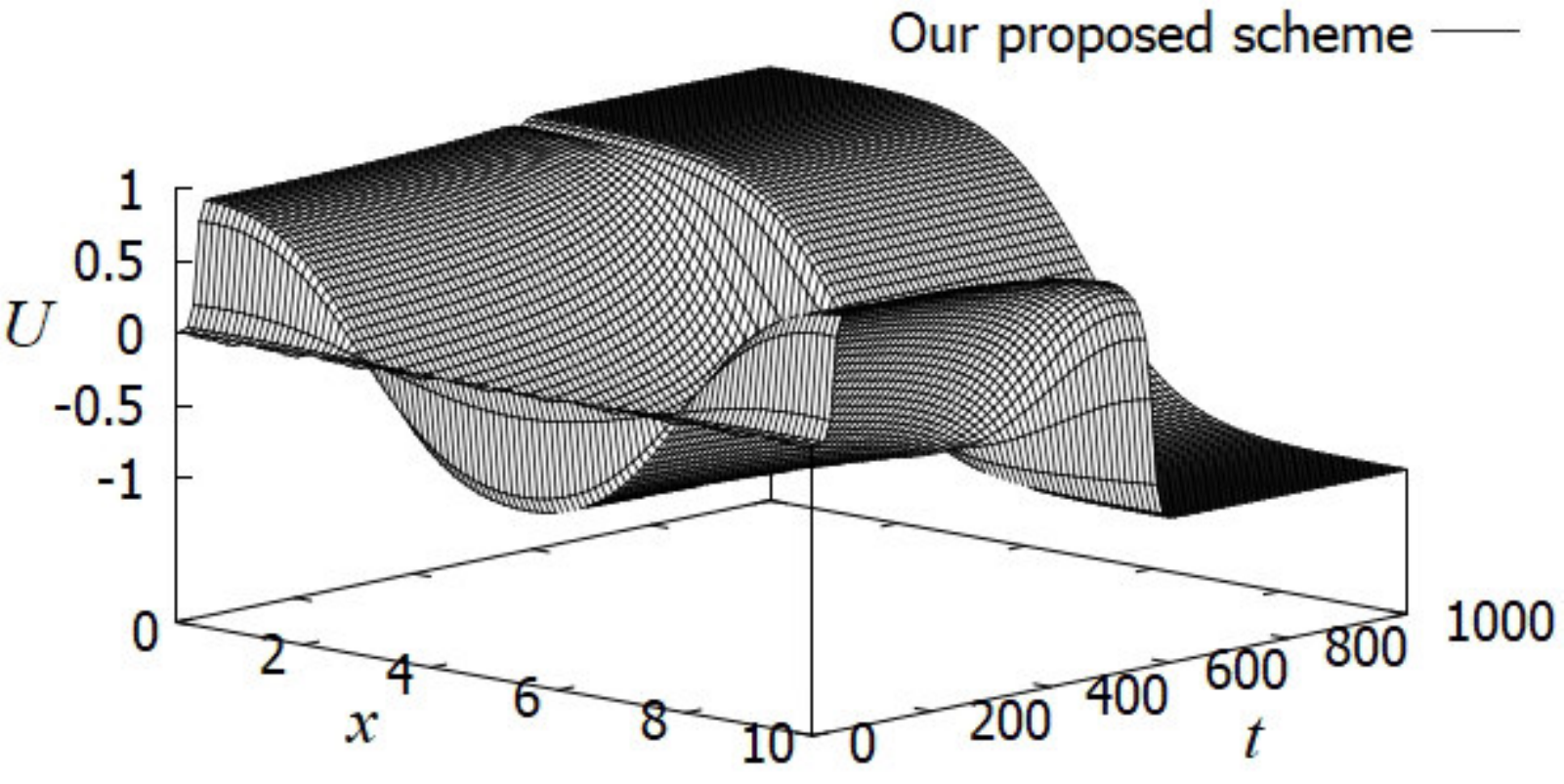} \vspace{-4mm} 
  \end{center}
  \caption{Numerical solution by our scheme with $\Delta x = 1/50$.}
  \label{fig:Our4}
 \end{minipage}
 \begin{minipage}{0.495\hsize}
  \begin{center}
   \includegraphics[width=78mm]{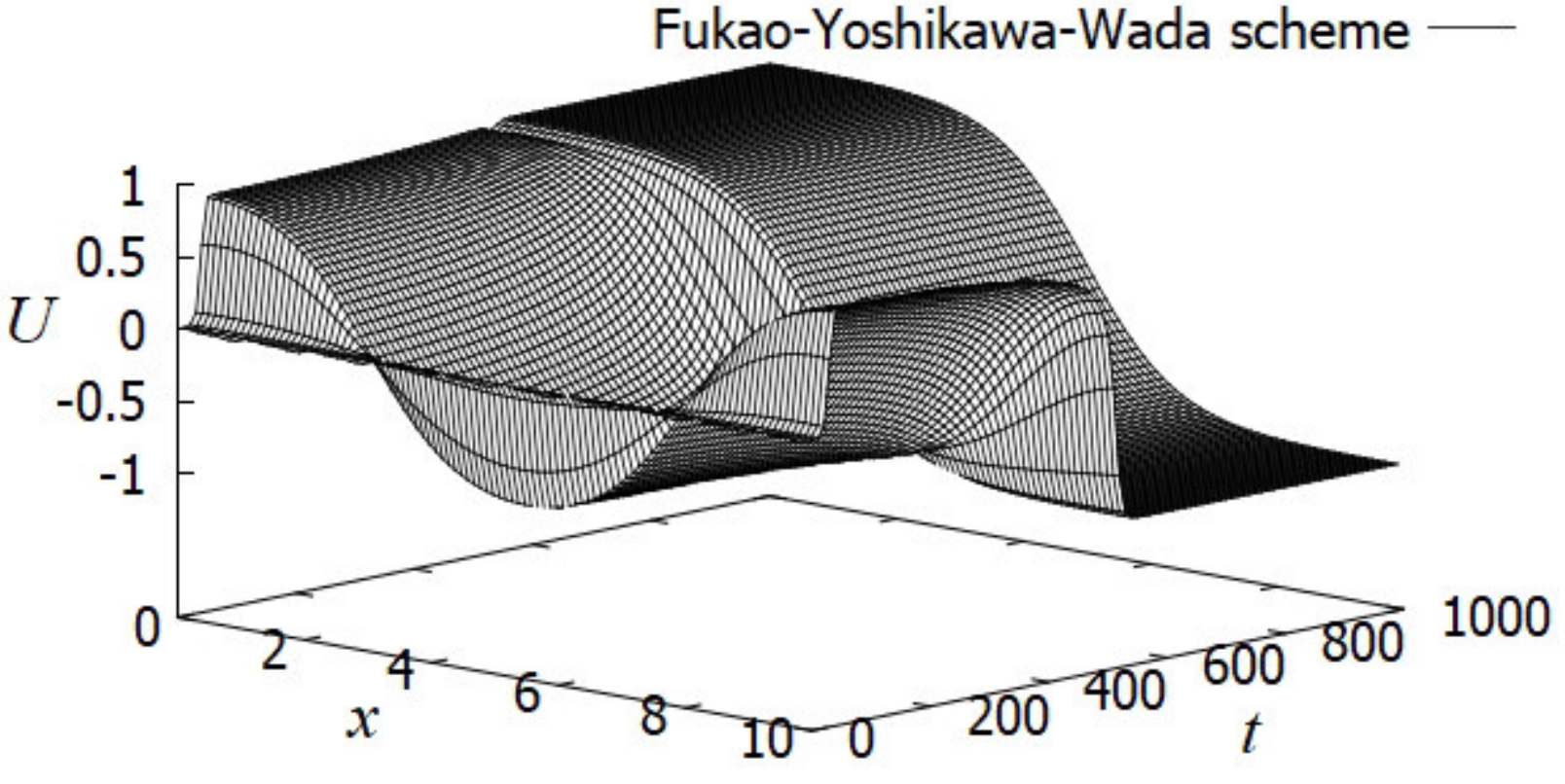} \vspace{-4mm}
  \end{center}
  \caption{Numerical solution by Fukao--Yoshikawa--Wada scheme with $\Delta x = 1/50$.}
  \label{fig:FYW4}
 \end{minipage} \vspace{-2mm}
\end{figure} 
\noindent 
Figure \ref{fig:Our4} shows the time development of the solution obtained by our scheme. 
Also, Figure \ref{fig:FYW4} shows the one by the Fukao--Yoshikawa--Wada scheme. 
Both results are similar to the result obtained by our scheme with $\Delta x = 1/25$. 
Hence, as can be seen from Figures \ref{fig:Our1}--\ref{fig:FYW2} and Figures \ref{fig:Our3}--\ref{fig:FYW4}, we expect that the solution obtained by our proposed scheme is more reliable than that by the Fukao--Yoshikawa--Wada scheme when the space mesh size is coarse.  

Next, we confirm the conservative property and the dissipative property. 
Figure \ref{fig:mass2} shows the time development of $M_{\rm d}(\bm{U}^{(n)})$ obtained by our scheme with $\Delta x = 1/50$. 
Figure \ref{fig:energy2} shows the time development of $E_{\rm d}^{(n)} - J_{\rm d}(\bm{U}^{(0)})$ obtained by our scheme with $\Delta x = 1/50$. \vspace{-2mm}
\begin{figure}[H]
 \begin{minipage}{0.495\hsize}
  \begin{center}
   \includegraphics[width=78mm]{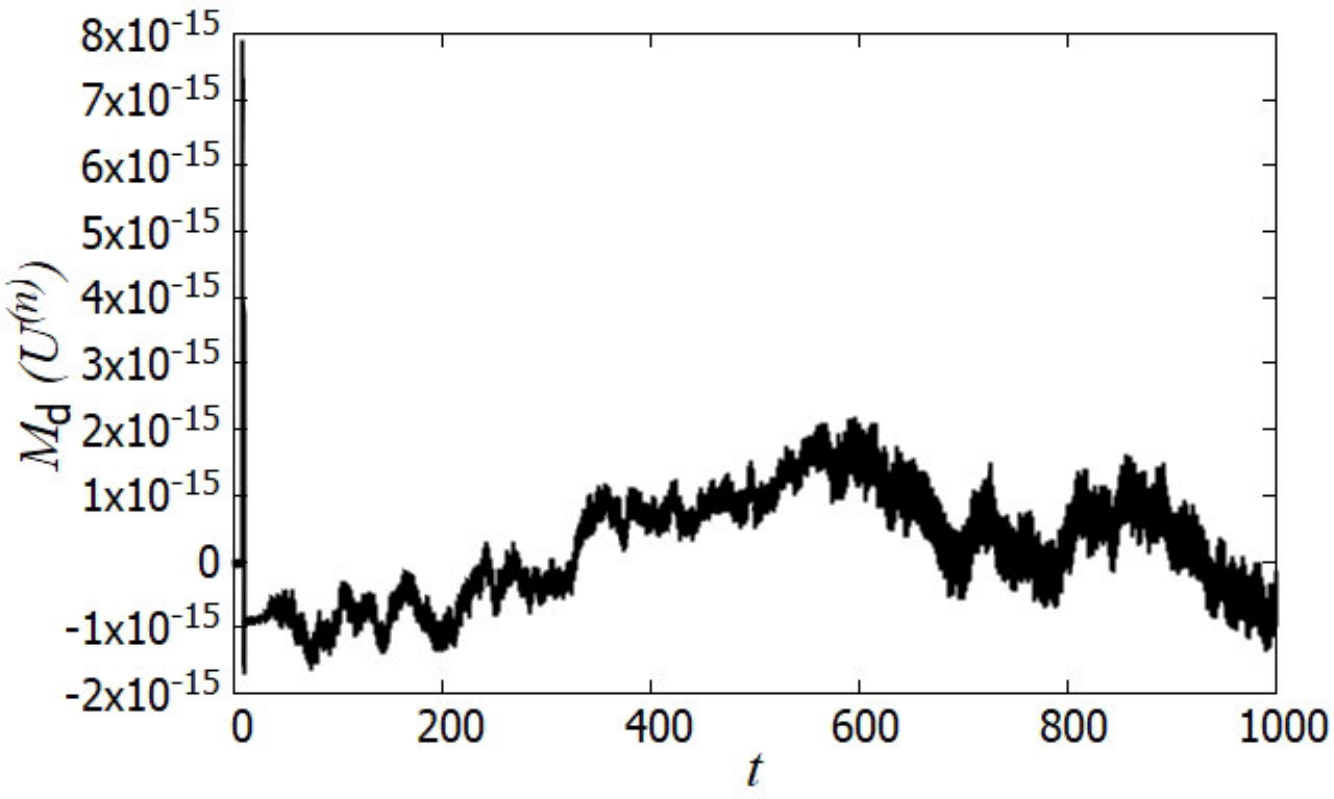} \vspace{-4mm}
  \end{center}
  \caption{Time development of $M_{\rm d}(\bm{U}^{(n)})$ obtained by our scheme with $\Delta x = 1/50$. $M_{\rm d}(\bm{U}^{(n)})$ does not change by about $14$ orders of magnitude.}
  \label{fig:mass2}
 \end{minipage}
 \begin{minipage}{0.495\hsize}
  \begin{center}
   \includegraphics[width=78mm]{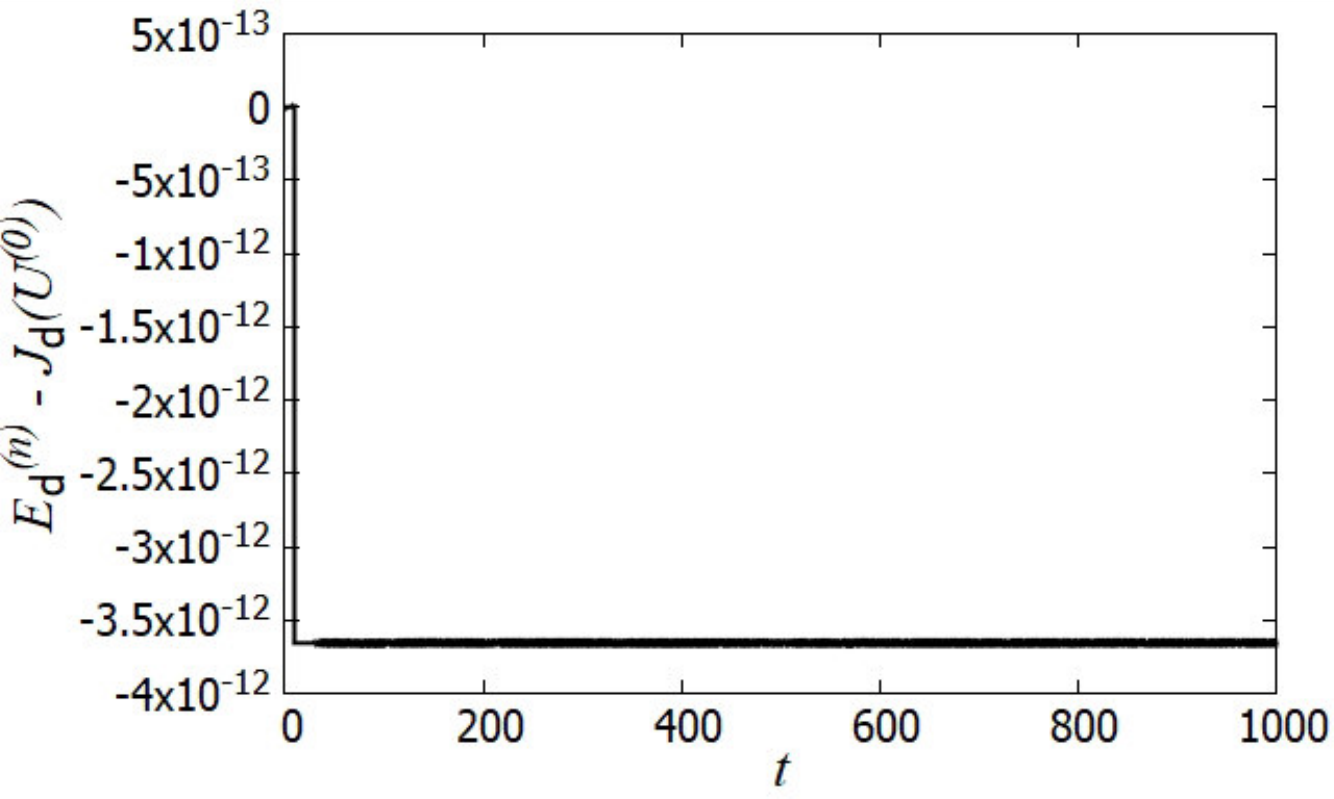} \vspace{-4mm}
  \end{center}
  \caption{Time development of $E_{\rm d}^{(n)} - J_{\rm d}(\bm{U}^{(0)})$ obtained by our scheme with $\Delta x = 1/50$. $E_{\rm d}^{(n)}$ does not change by about $11$ orders of magnitude.}
  \label{fig:energy2}
 \end{minipage} \vspace{-2mm}
\end{figure}
\noindent 
These graphs show that the quantities $M_{\rm d}(\bm{U}^{(n)})$ and $E_{\rm d}^{(n)}$ are conserved numerically. 
More precisely, $M_{\rm d}(\bm{U}^{(n)})$ does not change by about $14$ orders of magnitude, and $E_{\rm d}^{(n)}$ does not change by about $11$ orders of magnitude. 

\subsection{Computation example 3}
We consider the following dynamic boundary condition for the order parameter $u$: 
\begin{align}
\begin{cases}
\varepsilon_{\rm ex}\partial_{t}u(0,t) = \left. \partial_{x}u(x,t)\right|_{x=0} & {\rm in} \ (0,\infty), \\
\varepsilon_{\rm ex}\partial_{t}u(L,t) = -\left. \partial_{x}u(x,t)\right|_{x=L} & {\rm in} \ (0,\infty),  
\end{cases} \label{DBC'}
\end{align}
where $\varepsilon_{\rm ex}$ is a positive constant. 
For the chemical potential $p$, we consider the same homogeneous Neumann boundary condition as before. 
In this computation example, we fix $\varepsilon_{\rm ex} = 1000$. 
We consider 
\begin{equation*}
u(x,0) = u_{0}(x) = 0.05\sin( 2\pi x ) 
\end{equation*}
as the initial condition. 
We choose $K=50$ and fix $L=1$ so that $\Delta x = 1/50$. 
Also, we choose $N=500000$ and fix $T=1000$ so that $\Delta t = 1/500$. 
Besides, we fix the parameter $\gamma = 0.001$. 
Figure \ref{fig:Our5} shows the time development of the solution obtained by our scheme. \vspace{-2mm}
\begin{figure}[H]
  \begin{center}
   \includegraphics[width=78mm]{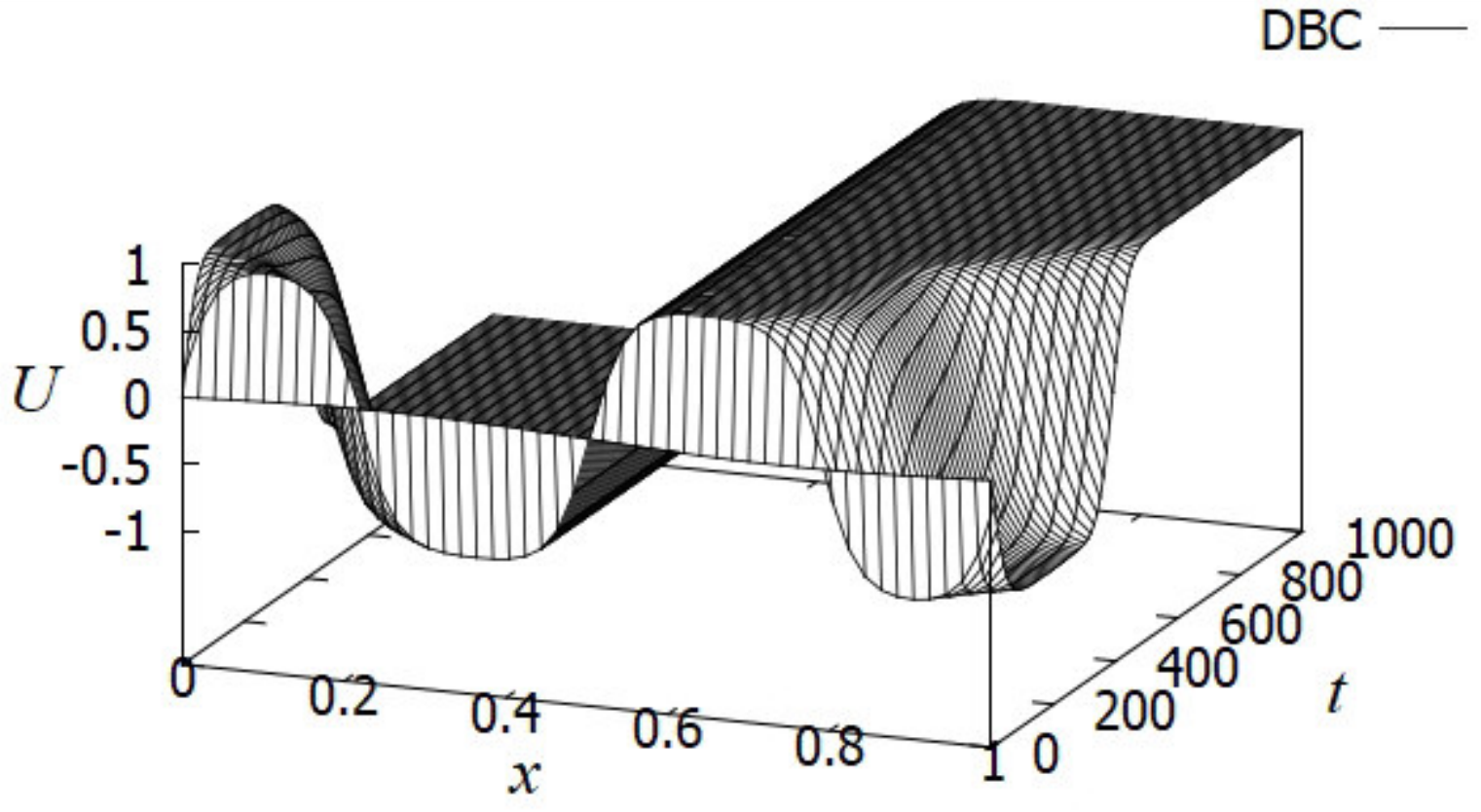} \vspace{-3mm}
  \end{center}
  \caption{Numerical solution to \eqref{Eq_i}--\eqref{Eq_ii} with \eqref{BCii} and \eqref{DBC'} obtained by our scheme.}
  \label{fig:Our5} \vspace{-2mm}
\end{figure}
\noindent 
%As stated in the Introduction, our study for the dynamic boundary condition differs from previous studies for non-dynamical boundary conditions such as the Neumann boundary condition. 
Since there is a term of the time derivative on the boundary, it is natural that the long-time behavior of the solution to \eqref{Eq_i}--\eqref{Eq_ii} with \eqref{BCii} and \eqref{DBC'} may differ from that to \eqref{Eq_i}--\eqref{Eq_ii} with the homogeneous Neumann boundary conditions for the order parameter and the chemical potential. 
In order to assure that the difference occurs, we present the computation example of our structure-preserving scheme for \eqref{Eq_i}--\eqref{Eq_ii} with the Neumann boundary conditions (see Appendix D for details). 

Next, we confirm the conservative property and the dissipative property. 
Figure \ref{fig:mass3} shows the time development of $M_{\rm d}(\bm{U}^{(n)})$ obtained by our scheme. 
Figure \ref{fig:energy3} shows the time development of $E_{\rm d}^{(n)} - J_{\rm d}(\bm{U}^{(0)})$ obtained by our scheme. \vspace{-2mm}
\begin{figure}[H]
 \begin{minipage}{0.495\hsize}
  \begin{center}
   \includegraphics[width=78mm]{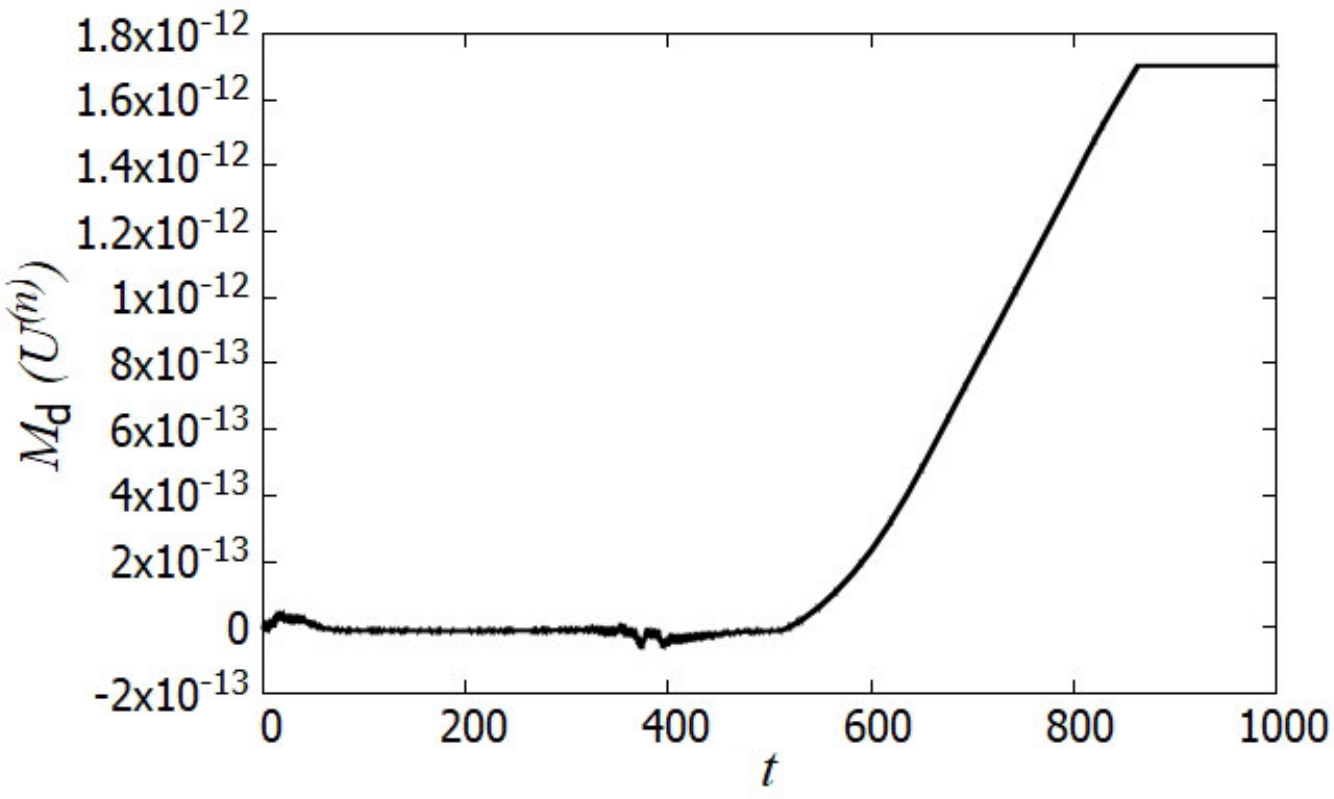} \vspace{-3mm}
  \end{center}
  \caption{Time development of $M_{\rm d}(\bm{U}^{(n)})$ obtained by our scheme. $M_{\rm d}(\bm{U}^{(n)})$ does not change by about $11$ orders of magnitude.}
  \label{fig:mass3}
 \end{minipage}
 \begin{minipage}{0.495\hsize}
  \begin{center}
   \includegraphics[width=78mm]{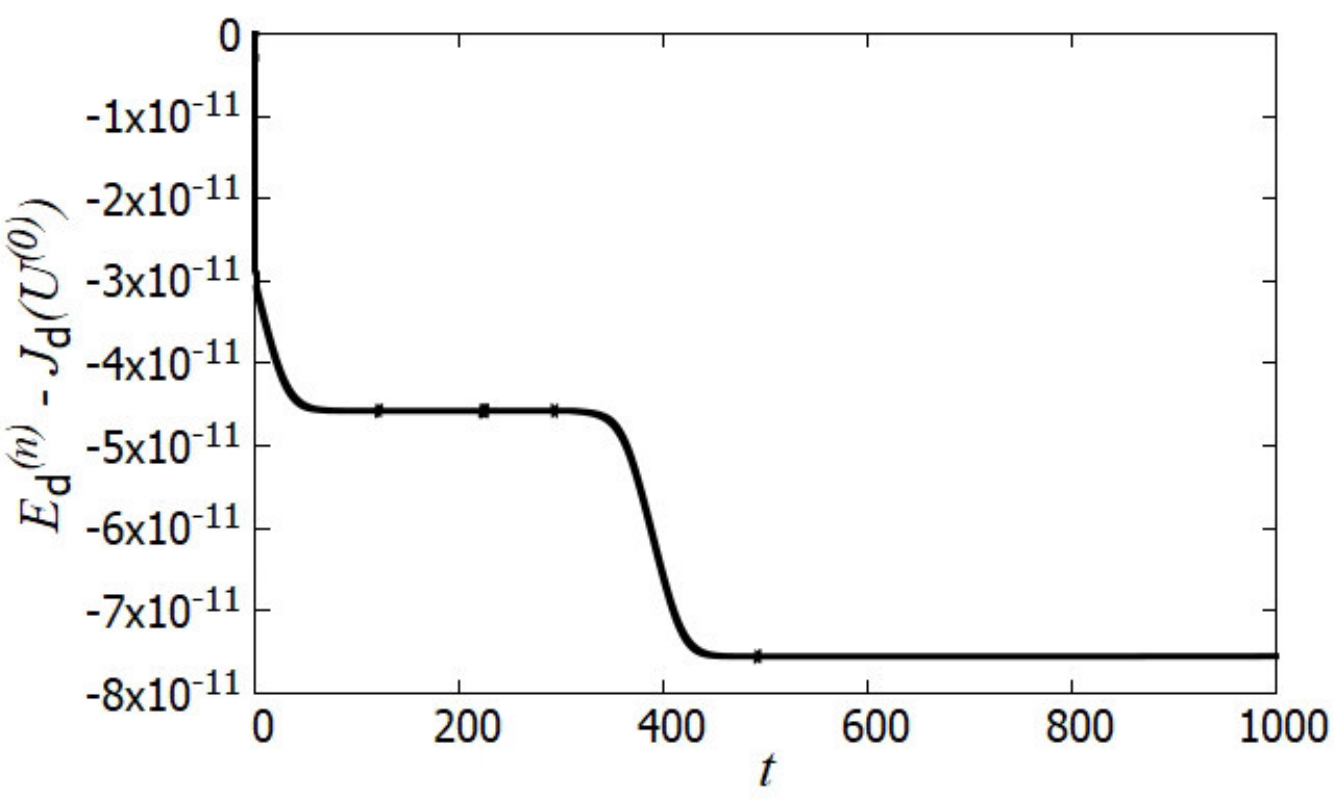} \vspace{-3mm}
  \end{center}
  \caption{Time development of $E_{\rm d}^{(n)} - J_{\rm d}(\bm{U}^{(0)})$ obtained by our scheme. $E_{\rm d}^{(n)}$ does not change by about $10$ orders of magnitude.}
  \label{fig:energy3}
 \end{minipage} \vspace{-2mm}
\end{figure}
\noindent 
These graphs show that the quantities $M_{\rm d}(\bm{U}^{(n)})$ and $E_{\rm d}^{(n)}$ are conserved numerically. 
More precisely, $M_{\rm d}(\bm{U}^{(n)})$ does not change by about $11$ orders of magnitude, and $E_{\rm d}^{(n)}$ does not change by about $10$ orders of magnitude. 
From the above, we can obtain the expected results.

\section{Summary}
We have proposed a structure-preserving finite difference scheme for the Cahn--Hilliard equation with a dynamic boundary condition using the discrete variational derivative method. 
By modifying the conventional manner and using an appropriate summation-by-parts formula, we can use a standard central difference operator as an approximation of an outward normal derivative on the boundary. 
Moreover, we have shown the stability, the solvability of the proposed scheme, and the error estimate. 
Especially, we have shown that our proposed scheme is second-order accurate in space, although the previous structure-preserving scheme by Fukao--Yoshikawa--Wada is first-order accurate in space. 
Also, computation examples have demonstrated the effectiveness of the proposed scheme. 
In particular, through computation examples, we have confirmed that we can obtain a valid numerical solution by our proposed scheme even when the space mesh size is coarse. 

\section*{Appendix A}
\renewcommand{\thelemma}{A.\arabic{lemma}} 
\renewcommand{\thecorollary}{A.\arabic{corollary}} 
\newtheorem*{Remarkap}{Remark} 
In this appendix A, we give a more precise evaluation of the discrete $L^\infty$-norm of the solution of our proposed scheme than Theorem \ref{thm:3.1} by evaluating errors of the discrete quantities when the initial data is sufficiently smooth.
Note that we use the same notations in Sections 1--3. 

\begin{lemma} \label{lem:a}
If $U_{k}^{(0)} = u_{0}(k\Delta x) \ (k = 0,\ldots, K)$ for a function $u_{0} \in C^{3}([0,L])$, then there exist constants $C_{J}, C_{M} >0$ independent of $\Delta x$ and $\Delta t$ such that \vspace{-1mm}
\begin{equation*}
\left| J_{\rm d}(\bm{U}^{(0)}) - J(u_{0}) \right| \leq C_{J}, \quad \left| M_{\rm d}(\bm{U}^{(0)}) - M(u_{0}) \right| \leq C_{M}. \vspace{-1mm}
\end{equation*}
\end{lemma}

\begin{proof}
From the triangle inequality, we see that 
\begin{align}
\left| J_{\rm d}(\bm{U}^{(0)}) - J(u_{0}) \right| 
	& = \left| \frac{1}{2}\left(\sum_{k=0}^{K-1}G_{{\rm d}, k}^{+}(\bm{U}^{(0)} )\Delta x + \sum_{k=1}^{K}G_{{\rm d}, k}^{-}(\bm{U}^{(0)} )\Delta x \right) - \int_{0}^{L}G(u_{0},\partial_{x}u_{0})dx \right| \nonumber\\
	& \leq \left| \frac{1}{2}\!\left(\sum_{k=0}^{K-1}G_{{\rm d}, k}^{+}(\bm{U}^{(0)} )\Delta x + \sum_{k=1}^{K}G_{{\rm d}, k}^{-}(\bm{U}^{(0)} )\Delta x \!\right) \! - \! \sum_{k=0}^{K}{}^{\prime \prime}G(u_{0}(k\Delta x),\partial_{x}u_{0}(k\Delta x) )\Delta x \right| \nonumber\\
	& \quad + \left| \sum_{k=0}^{K}{}^{\prime \prime}G(u_{0}(k\Delta x),\partial_{x}u_{0}(k\Delta x) )\Delta x - \int_{0}^{L}G(u_{0},\partial_{x}u_{0})dx \right| . \label{En1}
\end{align} 
Since $G(u_{0},\partial_{x}u_{0}) \in C^{2}([0,L])$ from the assumption $u_{0} \in C^{3}([0,L])$, by using the Euler--Maclaurin summation formula and $\Delta x \leq L$, we estimate the second term on the right-hand side of \eqref{En1} as follows: 
\begin{align}
\left| \sum_{k=0}^{K}{}^{\prime \prime}G(u_{0}(k\Delta x),\partial_{x}u_{0}(k\Delta x) )\Delta x - \int_{0}^{L}G(u_{0},\partial_{x}u_{0})dx \right| 
	& \leq \frac{ (\Delta x)^{2} }{8}\int_{0}^{L}\left| \partial_{x}^{2}G(u_{0}, \partial_{x}u_{0})\right| dx \nonumber\\ 
	& \leq \frac{ L^{2} }{8}\int_{0}^{L}\left| \partial_{x}^{2}G(u_{0}, \partial_{x}u_{0})\right| dx. \label{second}
\end{align}
Next, we estimate the first term on the right-hand side of \eqref{En1}. 
By using Lemma \ref{lem:2.1}, we have 
\begin{align*}
& \left| \frac{1}{2}\!\left(\sum_{k=0}^{K-1}G_{{\rm d}, k}^{+}(\bm{U}^{(0)} )\Delta x + \sum_{k=1}^{K}G_{{\rm d}, k}^{-}(\bm{U}^{(0)} )\Delta x \!\right) \! - \! \sum_{k=0}^{K}{}^{\prime \prime}G(u_{0}(k\Delta x),\partial_{x}u_{0}(k\Delta x) )\Delta x \right| \\
	& \leq \frac{1}{2}\!\left( \sum_{k=0}^{K-1}\!\left| G_{{\rm d}, k}^{+}(\bm{U}^{(0)} ) \! - \! G(u_{0}(k\Delta x),\partial_{x}u_{0}(k\Delta x) ) \right| \Delta x + \sum_{k=1}^{K}\!\left| G_{{\rm d}, k}^{-}(\bm{U}^{(0)} ) \! - \! G(u_{0}(k\Delta x),\partial_{x}u_{0}(k\Delta x) ) \right| \Delta x \!\right). 
\end{align*}
We obtain from the assumption $U_{k}^{(0)} = u_{0}(k\Delta x) \ (k = 0,\ldots, K)$ that 
\begin{gather*}
G_{{\rm d}, k}^{+}(\bm{U}^{(0)} ) - G(u_{0}(k\Delta x),\partial_{x}u_{0}(k\Delta x) ) = \frac{\gamma}{2}\left( \left| \delta_{k}^{+}u_{0}(k\Delta x)\right|^{2} - \left| \partial_{x}u_{0}(k\Delta x) \right|^{2} \right) \quad (k=0, \ldots, K-1), \\
G_{{\rm d}, k}^{-}(\bm{U}^{(0)} ) - G(u_{0}(k\Delta x),\partial_{x}u_{0}(k\Delta x) ) = \frac{\gamma}{2}\left( \left| \delta_{k}^{-}u_{0}(k\Delta x)\right|^{2} - \left| \partial_{x}u_{0}(k\Delta x) \right|^{2} \right) \quad (k=1, \ldots, K). 
\end{gather*}
For $k=0, \ldots, K-1$, applying the Taylor theorem to $u_{0}$, there exists $\zeta_{1} \in (0,1)$ such that 
\begin{equation*}
\frac{ u_{0}( (k+1)\Delta x) - u_{0}(k\Delta x) }{\Delta x} = \partial_{x}u_{0}(k\Delta x) + \frac{\Delta x}{2}\partial_{x}^{2}u_{0}( (k+\zeta_{1})\Delta x). 
\end{equation*}
Similarly, for $1, \ldots, K$, there exists $\zeta_{2} \in (0,1)$ such that 
\begin{equation*}
\frac{ u_{0}(k\Delta x) - u_{0}( (k-1)\Delta x) }{\Delta x} = \partial_{x}u_{0}(k\Delta x) - \frac{\Delta x}{2}\partial_{x}^{2}u_{0}( (k-\zeta_{2})\Delta x). 
\end{equation*}
Hence, we have 
\begin{align*}
& G_{{\rm d}, k}^{+}(\bm{U}^{(0)} ) - G(u_{0}(k\Delta x),\partial_{x}u_{0}(k\Delta x) ) \\
	& = \frac{\gamma}{2}\left\{ \Delta x \partial_{x}u_{0}(k\Delta x)\partial_{x}^{2}u_{0}( (k+\zeta_{1})\Delta x) + \frac{ (\Delta x)^{2} }{4}\left( \partial_{x}^{2}u_{0}( (k+\zeta_{1})\Delta x) \right)^{2} \right\} \quad (k=0, \ldots, K-1), \\
& G_{{\rm d}, k}^{-}(\bm{U}^{(0)} ) - G(u_{0}(k\Delta x),\partial_{x}u_{0}(k\Delta x) ) \\
	& = -\frac{\gamma}{2}\left\{ \Delta x \partial_{x}u_{0}(k\Delta x)\partial_{x}^{2}u_{0}( (k-\zeta_{2})\Delta x) - \frac{ (\Delta x)^{2} }{4}\left( \partial_{x}^{2}u_{0}( (k-\zeta_{2})\Delta x) \right)^{2} \right\} \quad (k=1, \ldots, K).  
\end{align*}
Therefore, from $\Delta x \leq L$ and $K\Delta x = L$, we obtain that 
\begin{align*}
& \sum_{k=0}^{K-1}\!\left| G_{{\rm d}, k}^{+}(\bm{U}^{(0)} ) - G(u_{0}(k\Delta x),\partial_{x}u_{0}(k\Delta x) ) \right| \Delta x \\
	& \leq \frac{\gamma}{2}\left\{ \Delta x\sum_{k=0}^{K-1}\left| \partial_{x}u_{0}(k\Delta x)\partial_{x}^{2}u_{0}( (k+\zeta_{1})\Delta x) \right|\Delta x + \frac{ (\Delta x)^{2} }{4}\sum_{k=0}^{K-1}\left| \partial_{x}^{2}u_{0}( (k+\zeta_{1})\Delta x)\right|^{2}\Delta x \right\} \\
	& \leq \frac{\gamma L^{2} }{2}\left( A_{1}A_{2} + \frac{L}{4}A_{2}^{2} \right), 
\end{align*}
where $A_{i} := \max_{x \in [0,L]}| \partial_{x}^{i}u_{0}(x) | \ (i=1,2)$. 
Similarly, we have 
\begin{equation*}
\sum_{k=1}^{K}\!\left| G_{{\rm d}, k}^{-}(\bm{U}^{(0)} ) - G(u_{0}(k\Delta x),\partial_{x}u_{0}(k\Delta x) ) \right| \Delta x \leq \frac{\gamma L^{2} }{2}\left( A_{1}A_{2} + \frac{L}{4}A_{2}^{2} \right). 
\end{equation*}
Thus, we see that 
\begin{equation}
\left| \frac{1}{2}\!\left(\sum_{k=0}^{K-1}G_{{\rm d}, k}^{+}(\bm{U}^{(0)} )\Delta x + \sum_{k=1}^{K}G_{{\rm d}, k}^{-}(\bm{U}^{(0)} )\Delta x \!\right) \! - \! \sum_{k=0}^{K}{}^{\prime \prime}G(u_{0}(k\Delta x),\partial_{x}u_{0}(k\Delta x) )\Delta x \right|  
	\leq \frac{\gamma L^{2} }{2}\!\!\left(\! A_{1}A_{2} + \frac{L}{4}A_{2}^{2} \right). \label{first}
\end{equation}
From \eqref{En1}, \eqref{second}, and \eqref{first}, we conclude that 
\begin{equation}
\left| J_{\rm d}(\bm{U}^{(0)}) - J(u_{0}) \right| \leq \frac{ L^{2} }{8}\int_{0}^{L}\left| \partial_{x}^{2}G(u_{0}, \partial_{x}u_{0})\right| dx + \frac{\gamma L^{2} }{2}\!\left(\! A_{1}A_{2} + \frac{L}{4}A_{2}^{2} \right). \label{J_est}
\end{equation}
Also, from the Euler--Maclaurin summation formula, we obtain
\begin{equation}
\left| M_{\rm d}(\bm{U}^{(0)}) - M(u_{0}) \right| = \left| \sum_{k=0}^{K}{}^{\prime \prime}u_{0}(k\Delta x)\Delta x - \int_{0}^{L}u_{0}dx \right| \leq \frac{L^{2} }{8}\int_{0}^{L}\left| \partial_{x}^{2}u_{0}\right| dx. \label{M_est}
\end{equation}
The right-hand sides of \eqref{J_est} and \eqref{M_est} are the desired constants $C_{J}$ and $C_{M}$, respectively. \hfill $\Box$\\
\end{proof}
\noindent From Theorem \ref{thm:3.1} and Lemma \ref{lem:a}, we have the following corollary. 
\begin{corollary} 
If $U_{k}^{(0)} = u_{0}(k\Delta x) \ (k = 0,\ldots, K)$ for a function $u_{0} \in C^{3}([0,L])$, then the solution of the scheme \eqref{dCH1}--\eqref{dNBC} satisfies the following inequality: 
\begin{equation}
\left\| \bm{U}^{(n)}\right\|_{L_{\rm d}^{\infty}} \leq \frac{1}{L}( | M(u_{0}) | + C_{M}) + \left\{\! \frac{2L}{\gamma}\!\left(\! J(u_{0}) + C_{J} + L\left| \min\!\left\{\inf_{\xi \in \mathbb{R}}F(\xi), 0\right\} \!\right| \right) \!\right\}^{\! \frac{1}{2}} \quad (n=0,1,\ldots), \label{GB2}
\end{equation}
where 
\begin{gather*}
C_{J} := L^{2}\left\{ \frac{1}{8}\int_{0}^{L}\left| \partial_{x}^{2}G(u_{0}, \partial_{x}u_{0})\right| dx + \frac{\gamma}{2}\left( A_{1}A_{2} + \frac{L}{4}A_{2}^{2} \right) \right\}, \quad A_{i} := \max_{x \in [0,L]}\left| \partial_{x}^{i}u_{0}(x) \right| \quad (i=1,2), \\
C_{M} := \frac{L^{2} }{8}\int_{0}^{L}\left| \partial_{x}^{2}u_{0}\right| dx. 
\end{gather*}
\end{corollary}
\section*{Appendix B}
\renewcommand{\thelemma}{B.\arabic{lemma}} 
In this appendix B, we prove that the $(K+1) \times (K+1)$ matrix $A$ defined in Section 4 is nonsingular. 
Since proofs of the following lemmas can be found in \cite{Tre,Zha}, we omit them. 
\setcounter{lemma}{0}
\begin{lemma}[{\cite[Theorem 2.8]{Zha}}] \label{Ap:A} 
Let $A$ and $B$ be $n \times n$ complex matrices. 
Then, $AB$ and $BA$ have the same eigenvalues, counting multiplicity.  
\end{lemma} 
\noindent 
In the following lemmas, we denote the Hermitian conjugate or adjoint of a $n \times n$ matrix $A$ by $A^{\ast}$. 
\begin{lemma}[Sylvester's law of inertia {\cite[Theorem 8.3]{Zha}}] \label{Ap:B}
Let $A$ and $B$ be $n \times n$ Hermitian matrices. 
Then, there exists a nonsingular $n \times n$ matrix $S$ such that $B = S^{\ast}AS$ if and only if $A$ and $B$ have the same inertia, i.e., 
\begin{equation*}
{\rm In}(A) = {\rm In}(B), 
\end{equation*} 
where the inertia ${\rm In}(A)$ of $A$ is defined to be the ordered triple $(i_{+}(A), i_{-}(A), i_{0}(A))$, that is, 
\begin{equation*}
{\rm In}(A) := (i_{+}(A), i_{-}(A), i_{0}(A)), 
\end{equation*}
$i_{+}(A)$, $i_{-}(A)$, and $i_{0}(A)$ are the numbers of positive, negative, and zero eigenvalues of $A$, respectively (including multiplicities). 
\end{lemma}

\begin{lemma}[Cholesky factorization {\cite[Theorem 23.1]{Tre}}] \label{Ap:C} 
For any $n \times n$ Hermitian positive definite matrix $A$, there exists a unique $n \times n$ upper-triangular matrix $R$ whose diagonal components are all positive such that 
\begin{equation*}
A = R^{\ast}R. 
\end{equation*}
\end{lemma}
\noindent 
Using the above lemmas, we obtain the following lemma: 

\begin{lemma} \label{Ap:D} 
Let $A$ be an arbitrary $n \times n$ Hermitian positive semi-definite matrix and let $B$ be an arbitrary $n \times n$ Hermitian positive definite matrix. 
Then, the eigenvalues of $AB$ are all real and nonnegative. 
\end{lemma}

\begin{proof}
Applying Lemma \ref{Ap:C} to the Hermitian positive definite matrix $B$, there exists a unique $n \times n$ upper-triangular matrix $R$ such that 
\begin{equation*}
B = R^{\ast}R, \quad r_{ii} > 0 \quad (i = 1, \ldots, n), 
\end{equation*}
where $r_{ii} \ (i = 1, \ldots, n)$ are diagonal components of $R$. 
Hence, we have 
\begin{equation}
AB = A(R^{\ast}R) = (AR^{\ast})R. \label{AB}
\end{equation}
Incidentally, it holds from $r_{ii} > 0 \ (i = 1, \ldots, n)$ that $\det R^{\ast} = r_{11} \cdots r_{nn} > 0$. 
That is, $R^{\ast}$ is nonsingular. 
Therefore, using Lemma \ref{Ap:B}, we obtain 
\begin{equation}
{\rm In}(A) = {\rm In}( (R^{\ast})^{\ast}AR^{\ast}) = {\rm In}( RAR^{\ast}). \label{InA}
\end{equation}
Since $(AR^{\ast})R$ and $R(AR^{\ast})$ have the same eigenvalues from Lemma \ref{Ap:A}, by using \eqref{AB} and \eqref{InA}, we obtain 
\begin{equation*}
{\rm In}(A) = {\rm In}( RAR^{\ast}) = {\rm In}((AR^{\ast})R) = {\rm In}(AB).
\end{equation*}
Since $A$ is positive semi-definite, the eigenvalues of $A$ are all real and nonnegative. 
Namely, the eigenvalues of $AB$ are all real and nonnegative, too. \hfill $\Box$\\
\end{proof}
\noindent 
From Lemma \ref{Ap:D}, we have the following lemma: 

\begin{lemma} \label{Ap:E}
The $(K+1) \times (K+1)$ matrix $A$ defined by \eqref{matrixA} is nonsingular. 
\end{lemma}

\begin{proof}
We show that the determinant of $A$ is positive. 
Firstly, let us define the $(K+1) \times (K+1)$ matrices $D_{2}$ and $\tilde{D}_{2}$ by 
\begin{equation*}
D_{2} := \begin{pmatrix}
-2 & 2 & & &  \\
1 & -2 & 1 & &  \\
& \ddots & \ddots & \ddots & \\
& & 1 & -2 & 1 \\
& & & 2 & -2 \\
\end{pmatrix}, \quad
\tilde{D}_{2} := \begin{pmatrix}
\displaystyle -2-\frac{1}{\alpha} & 2 & & &  \\
1 & -2 & 1 & &  \\
& \ddots & \ddots & \ddots & \\
& & 1 & -2 & 1 \\
& & & 2 & \displaystyle -2 -\frac{1}{\alpha} \\
\end{pmatrix}. 
\end{equation*}
Then, we obtain 
\begin{equation*}
D_{2}\tilde{D}_{2} 
= \begin{pmatrix}
\displaystyle 6 + \frac{2}{\alpha} & -8 & 2 & & & & & & \\[5pt]
\displaystyle -4-\frac{1}{\alpha} & 7 & -4 & 1 & & & & & \\
1 & -4 & 6 & -4 & 1 & & & \\
& 1 & -4 & 6 & -4 & 1 & & & \\
& & \ddots & \ddots & \ddots & \ddots & \ddots & & \\
& & & 1 & -4 & 6 & -4 & 1 & \\
& & & & 1 & -4 & 6 & -4 & 1 \\
& & & & & 1 & -4 & 7 & \displaystyle -4-\frac{1}{\alpha} \\[5pt]
& & & & & & 2 & -8 & \displaystyle 6 + \frac{2}{\alpha} \\
\end{pmatrix}. 
%& = \begin{pmatrix}
%-2 & 2 & & &  \\
%1 & -2 & 1 & &  \\
%& \ddots & \ddots & \ddots & \\
%& & 1 & -2 & 1 \\
%& & & 2 & -2 \\
%\end{pmatrix} 
%\begin{pmatrix}
%\displaystyle -2-\frac{1}{\alpha} & 2 & & &  \\
%1 & -2 & 1 & &  \\
%& \ddots & \ddots & \ddots & \\
%& & 1 & -2 & 1 \\
%& & & 2 & \displaystyle -2-\frac{1}{\alpha} \\
%\end{pmatrix} \\ 
\end{equation*}
Namely, $A = I +\beta D_{2}\tilde{D}_{2}$. 
We remark that $D_{2}$ and $\tilde{D}_{2}$ are not symmetric. 
Thus, following the procedure of the proof for Lemma 0.1.1 in \cite{Fal}, we show that $D_{2}$ and $\tilde{D}_{2}$ are similar to some symmetric tridiagonal matrices, respectively. 
Let us define $b_{1} := 2$, $b_{i} := 1 \ (i=2,\ldots,K)$, $c_{i} := 1 \ (i=1,\ldots,K-1)$, and $c_{K} := 2$. 
Then, $D_{2}$ and $\tilde{D}_{2}$ are expressed as follows: 
\begin{equation*}
D_{2} = \begin{pmatrix}
-2 & b_{1} & & &  \\
c_{1} & -2 & b_{2} & &  \\
& \ddots & \ddots & \ddots & \\
& & c_{K-1} & -2 & b_{K} \\
& & & c_{K} & -2 \\
\end{pmatrix}, \quad 
\tilde{D}_{2} = \begin{pmatrix}
\displaystyle -2-\frac{1}{\alpha} & b_{1} & & &  \\
c_{1} & -2 & b_{2} & &  \\
& \ddots & \ddots & \ddots & \\
& & c_{K-1} & -2 & b_{K} \\
& & & c_{K} & \displaystyle -2-\frac{1}{\alpha} \\
\end{pmatrix}. 
\end{equation*}
Moreover, let $v_{11} := 1$ and $v_{ii} := \sqrt{(b_{1}b_{2}\cdots b_{i-1})/(c_{1}c_{2}\cdots c_{i-1})} \ (i=2,\ldots,K+1)$. 
Then, we have 
\begin{equation*}
v_{ii} = \sqrt{\frac{2 \cdot 1 \cdots 1}{1 \cdot 1 \cdots 1}} = \sqrt{2} \quad (i=2,\ldots, K), \quad v_{K+1K+1} = \sqrt{\frac{2 \cdot 1 \cdots 1 \cdot 1}{1 \cdot 1 \cdots 1 \cdot 2}} = 1. 
\end{equation*}
Furthermore, let us define the $(K+1) \times (K+1)$ matrix $V$ by \vspace{-1mm}
\begin{equation*}
\displaystyle V := \diag_{1 \leq i \leq K+1}v_{ii} = 
\begin{pmatrix}
v_{11} & & & & \\
& v_{22} & & & \\
& & \ddots & &\\
& & & v_{KK} &  \\
& & & & v_{K+1K+1} \\
\end{pmatrix}
= \begin{pmatrix}
1 & & & & \\
& \sqrt{2} & & & \\
& & \ddots & & \\
& & & \sqrt{2} & \\
& & & & 1 \\
\end{pmatrix}. \vspace{-1mm}
\end{equation*}
Then, we have \vspace{-1mm}
\begin{equation*}
V D_{2} V^{-1} 
= \left( \begin{array}{@{\hskip0pt}c@{\hskip5pt}c@{\hskip5pt}c@{\hskip5pt}c@{\hskip5pt}c@{\hskip5pt}c@{\hskip5pt}c@{\hskip0pt}} 
-2 & \sqrt{2} & & & & & \\
\sqrt{2} & -2 & 1 & & & & \\
& 1 & -2 & 1 & & & \\
& & \ddots & \ddots & \ddots & & \\
& & & 1 & -2 & 1 & \\
& & & & 1 & -2 & \sqrt{2} \\
& & & & & \sqrt{2} & -2 \\
\end{array} \right) , \quad 
V \tilde{D}_{2} V^{-1} 
= \left( \begin{array}{@{\hskip0pt}c@{\hskip7pt}c@{\hskip5pt}c@{\hskip5pt}c@{\hskip5pt}c@{\hskip5pt}c@{\hskip7pt}c@{\hskip0pt}} 
\displaystyle -2-\frac{1}{\alpha} & \sqrt{2} & & & & & \\
\sqrt{2} & -2 & 1 & & & & \\
& 1 & -2 & 1 & & & \\
& & \ddots & \ddots & \ddots & & \\
& & & 1 & -2 & 1 & \\
& & & & 1 & -2 & \sqrt{2} \\
& & & & & \sqrt{2} & \displaystyle -2-\frac{1}{\alpha} \\
\end{array} \right) . \vspace{-1mm}
\end{equation*}
Here, let $X := V D_{2} V^{-1} $ and $Y := V \tilde{D}_{2} V^{-1}$. 
Then, $X$ and $Y$ are $(K+1) \times (K+1)$ symmetric matrices. 
Furthermore, it holds that \vspace{-0.5mm}
\begin{equation*}
XY = V D_{2} V^{-1} V \tilde{D}_{2} V^{-1} = V D_{2} \tilde{D}_{2} V^{-1}. \vspace{-0.5mm}
\end{equation*}
That is, $D_{2}\tilde{D}_{2}$ is similar to $XY$. 
Hence, $D_{2}\tilde{D}_{2}$ and $XY$ have the same eigenvalues. 
Moreover, $-X$ is positive semi-definite, and $-Y$ is positive definite. 
Actually, for any non-zero vector $\bm{x} = (x_{1},x_{2},\ldots,x_{K+1})^{\top} \! \in \mathbb{R}^{K+1}$, it holds that \vspace{-1mm}
\begin{equation*}
(-Y)\bm{x} = \left(
    \begin{array}{c}
      \displaystyle \left( 2 + \frac{1}{\alpha} \right) x_{1} - \sqrt{2}x_{2} \\
      -\sqrt{2}x_{1} + 2x_{2} - x_{3} \\
      -x_{2} + 2x_{3} - x_{4} \\
      \vdots \\
      -x_{K-2} + 2x_{K-1} - x_{K} \\
      -x_{K-1} + 2x_{K} - \sqrt{2}x_{K+1} \\
      \displaystyle -\sqrt{2}x_{K} + \left( 2 + \frac{1}{\alpha} \right) x_{K+1}
    \end{array}
    \right) . \vspace{-1mm}
\end{equation*}
Thus, we have \vspace{-1mm}
\begin{align*}
\bm{x}^{\! \top}\! (-Y)\bm{x} \! 
	& = \left( 2 + \frac{1}{\alpha} \right)x_{1}^{2} - \sqrt{2}x_{1}x_{2} - \sqrt{2}x_{1}x_{2} + 2x_{2}^{2} - x_{2}x_{3} - x_{2}x_{3} + 2x_{3}^{2} - x_{3}x_{4} - \cdots \\
	& \quad \cdots - \! x_{K-2}x_{K \! - 1} \! + \! 2x_{K \! -1}^{2} \! - \! x_{K \! -1}x_{K} \! - \! x_{K \! -1}x_{K} \! + \! 2x_{K}^{2} \! - \! \sqrt{2}x_{K}x_{K+1} \! - \! \sqrt{2}x_{K}x_{K+1} \! + \!\! \left(\! 2 \! + \! \frac{1}{\alpha} \!\right)\! x_{K+1}^{2} \\
%	& = \frac{1}{\alpha}x_{1}^{2} + 2x_{1}^{2} - 2\sqrt{2}x_{1}x_{2} + x_{2}^{2} + x_{2}^{2} - 2x_{2}x_{3} + x_{3}^{2} + \cdots \\
%	& \quad \cdots + x_{K-1}^{2} - 2x_{K-1}x_{K} + x_{K}^{2} + x_{K}^{2} - 2\sqrt{2}x_{K}x_{K+1} + 2x_{K+1}^{2} + \frac{1}{\alpha}x_{K+1}^{2} \\
	%& = \frac{1}{\alpha}x_{1}^{2} \! + \! 2(x_{1}^{2} \! - \! \sqrt{2}x_{1}x_{2}) \! + \! x_{2}^{2} \! + \! (x_{2} \! - \! x_{3})^{2} \! + \! \cdots \! + \! (x_{\! K \! - \! 1} \! - \! x_{\! K})^{2} \! + \! x_{\! K}^{2} \! + \! 2(x_{\! K \! + \! 1}^{2} \! - \! \sqrt{2}x_{K}x_{\! K \! + \! 1}) \! + \! \frac{1}{\alpha}x_{\! K \! + \! 1}^{2} \\
	%& = \frac{1}{\alpha}x_{1}^{2} \! + \! 2\!\left\{\!\! \left(\!\! x_{1} \! - \! \frac{\sqrt{2}}{2}x_{2} \!\!\right)^{\!\!\! 2} \!\! - \! \frac{1}{2}x_{2}^{2} \!\!\right\} \! + \! x_{2}^{2} \! + \! (x_{2} \! - \! x_{3})^{2} \! + \cdots \\ 
	%& \quad \cdots + \! (x_{K-1} \! - \! x_{K})^{2} \! + \! x_{K}^{2} \! + \! 2\!\left\{\!\! \left(\!\! x_{K+1} \! - \! \frac{\sqrt{2}}{2}x_{K} \!\!\right)^{\!\!\! 2} \!\! - \! \frac{1}{2}x_{K}^{2} \!\!\right\} \! + \! \frac{1}{\alpha}x_{K+1}^{2} \\
	& = \frac{1}{\alpha}x_{1}^{2} \! + \! 2\!\left(\! x_{1} \! - \! \frac{\sqrt{2}}{2}x_{2} \!\right)^{\! 2} \! + \! (x_{2} \! - \! x_{3})^{2} \! + \! \cdots \! + \! (x_{K-1} \! - \! x_{K})^{2} \! + \! 2\!\left(\! x_{K+1} \! - \! \frac{\sqrt{2}}{2}x_{K} \!\right)^{\! 2} \! + \! \frac{1}{\alpha}x_{K+1}^{2} \geq 0. 
\end{align*}
Suppose that $\bm{x}^{\top}(-Y)\bm{x} = 0$. 
Then, we get $x_{1} = \cdots = x_{K+1} = 0$. 
This is contradictory to $\bm{x} \neq \bm{0}$. 
Namely, $-Y$ is positive definite. 
Similarly, by direct calculation, we can see that $-X$ is positive semi-definite. 
Therefore, from Lemma \ref{Ap:D}, eigenvalues of $XY = (-X)(-Y)$ are all real and nonnegative. 
Hence, eigenvalues of $A$ are all positive from two facts that eigenvalues of $D_{2} \tilde{D}_{2}$ are all real and nonnegative and that $\beta$ is positive. 
From the above, $\det A > 0$, i.e., $A$ is nonsingular. \hfill $\Box$\\
\end{proof}

\section*{Appendix C}
\renewcommand{\thelemma}{C.\arabic{lemma}} 
\setcounter{lemma}{0}
In this appendix C, we prove Lemma \ref{lem:5.1} and Lemma \ref{lem:5.2} we have used in Section 5. 
Note that we use the same notations, assumptions, lemmas, and the corollary in Sections 1--5. 
\begin{lemma} \label{Ap:F} 
Assume that $u \in C^{4}([0,L] \times [0,T])$. 
Then, we obtain the following equations on the errors: \vspace{-2mm}
\begin{gather}
\frac{e_{u,k}^{(n+1)} - e_{u,k}^{(n)}}{\Delta t} = \delta_{k}^{\langle 2 \rangle}e_{p,k}^{(n)} + \xi_{1,k}^{\left( n + \frac{1}{2}\right)} \quad (k=0,\ldots, K,\ n=0,1,\ldots, N-1), \label{eu1}\\
e_{p,k}^{(\! n\! )} \! = \! - \gamma \delta_{k}^{\langle 2 \rangle}\!\!\left(\!\! \frac{e_{u,k}^{(\! n+1\! )} \!\! + \! e_{u,k}^{(\! n\! )}}{2}\!\!\right) \! + \! \left(\! \frac{dF}{d(U_{k}^{(\! n+1\! )}\! ,U_{k}^{(\! n\! )})} \! - \! \frac{dF}{d(u_{k}^{(\! n+1\! )}\! ,u_{k}^{(\! n\! )})} \!\right) \! + \! \xi_{2,k}^{\left(\! n + \frac{1}{2}\!\right)} \quad (k \! = \! 0,\ldots, K,\ n \! = \! 0,1,\ldots, N \! - \! 1), \label{ep1}
\end{gather}
\begin{gather}
\frac{e_{u,0}^{(n+1)} - e_{u,0}^{(n)}}{\Delta t} = \delta_{k}^{\langle 1 \rangle}\left(\frac{e_{u,0}^{(n+1)} + e_{u,0}^{(n)}}{2}\right) + \xi_{3,0}^{\left( n + \frac{1}{2}\right)} \quad (n=0,1,\ldots, N-1), \label{eu0}\\
\frac{e_{u,K}^{(n+1)} - e_{u,K}^{(n)}}{\Delta t} = -\delta_{k}^{\langle 1 \rangle}\left(\frac{e_{u,K}^{(n+1)} + e_{u,K}^{(n)}}{2}\right) + \xi_{3,K}^{\left( n + \frac{1}{2}\right)} \quad (n=0,1,\ldots, N-1), \label{euK}\\
\delta_{k}^{\langle 1 \rangle}e_{p,k}^{(n)} = 0 \quad (k=0,K,\ n=0,1,\ldots, N-1). \label{ep_bc}
\end{gather}
\end{lemma}

\begin{proof}
For any fixed $n=0,1,\ldots, N-1$, from the definition of $\bm{e}_{u}$, \eqref{Eq_i}, and \eqref{dCH1}, we have 
\begin{align}
\frac{e_{u,k}^{(n+1)} - e_{u,k}^{(n)}}{\Delta t} 
	& = \frac{U_{k}^{(n+1)} - U_{k}^{(n)}}{\Delta t} - \frac{u_{k}^{(n+1)} - u_{k}^{(n)}}{\Delta t} \nonumber\\
	& = \frac{U_{k}^{(n+1)} - U_{k}^{(n)}}{\Delta t} - \partial_{t}u_{k}^{\left( n + \frac{1}{2}\right)} + \partial_{t}u_{k}^{\left( n + \frac{1}{2}\right)} - \frac{u_{k}^{(n+1)} - u_{k}^{(n)}}{\Delta t} \nonumber\\
	& = \delta_{k}^{\langle 2 \rangle} P_{k}^{(n)} - \partial_{x}^{2}p_{k}^{\left( n + \frac{1}{2}\right)} + \partial_{t}u_{k}^{\left( n + \frac{1}{2}\right)} - \frac{u_{k}^{(n+1)} - u_{k}^{(n)}}{\Delta t} \nonumber\\
	& = \delta_{k}^{\langle 2 \rangle} P_{k}^{(n)} - \delta_{k}^{\langle 2 \rangle}p_{k}^{\left( n + \frac{1}{2}\right)} + \delta_{k}^{\langle 2 \rangle}p_{k}^{\left( n + \frac{1}{2}\right)} - \partial_{x}^{2}p_{k}^{\left( n + \frac{1}{2}\right)} + \partial_{t}u_{k}^{\left( n + \frac{1}{2}\right)} - \frac{u_{k}^{(n+1)} - u_{k}^{(n)}}{\Delta t} \nonumber\\
	& = \delta_{k}^{\langle 2 \rangle}e_{p,k}^{(n)} + \xi_{1,k}^{\left( n + \frac{1}{2}\right)} \quad (k=1,\ldots, K-1). \label{eu}
\end{align}
Similarly, from \eqref{Eq_ii}, \eqref{dCH2}, and the definitions of $\bm{e}_{u}$ and $\bm{e}_{p}$, we obtain 
\begin{align}
e_{p,k}^{(n)} 
	& = -\gamma \delta_{k}^{\langle 2 \rangle}\left(\frac{U_{k}^{(n+1)} + U_{k}^{(n)}}{2}\right) + \frac{dF}{d(U_{k}^{(n+1)},U_{k}^{(n)})} + \gamma \partial_{x}^{2}u_{k}^{\left( n + \frac{1}{2}\right)} - F'( u_{k}^{\left( n + \frac{1}{2}\right)}) \nonumber\\
	& = - \gamma \delta_{k}^{\langle 2 \rangle}\left(\frac{e_{u,k}^{(n+1)} + u_{k}^{(n+1)} + e_{u,k}^{(n)} + u_{k}^{(n)}}{2}\right) + \gamma \partial_{x}^{2}u_{k}^{\left( n + \frac{1}{2}\right)} \nonumber\\
	& \quad + \frac{dF}{d(U_{k}^{(n+1)},U_{k}^{(n)})} - \frac{dF}{d(u_{k}^{(n+1)},u_{k}^{(n)})} + \frac{dF}{d(u_{k}^{(n+1)},u_{k}^{(n)})} - F'( u_{k}^{\left( n + \frac{1}{2}\right)}) \nonumber\\
	& = - \gamma \delta_{k}^{\langle 2 \rangle}\left(\frac{e_{u,k}^{(n+1)} + e_{u,k}^{(n)}}{2}\right) + \gamma\left\{\partial_{x}^{2}u_{k}^{\left( n + \frac{1}{2}\right)} - \delta_{k}^{\langle 2 \rangle}\left(\frac{u_{k}^{(n+1)} + u_{k}^{(n)}}{2}\right)\right\} \nonumber\\
	& \quad + \left(\frac{dF}{d(U_{k}^{(n+1)},U_{k}^{(n)})} - \frac{dF}{d(u_{k}^{(n+1)},u_{k}^{(n)})}\right) + \left(\frac{dF}{d(u_{k}^{(n+1)},u_{k}^{(n)})} - F'( u_{k}^{\left( n + \frac{1}{2}\right)})\right) \nonumber\\
	& = - \gamma \delta_{k}^{\langle 2 \rangle}\!\!\left(\! \frac{e_{u,k}^{(n+1)} \! + \! e_{u,k}^{(n)}}{2} \!\right) \! + \! \left(\! \frac{dF}{d(U_{k}^{(n+1)},U_{k}^{(n)})} - \frac{dF}{d(u_{k}^{(n+1)},u_{k}^{(n)})} \!\right) \! + \! \xi_{2,k}^{\left(\! n + \frac{1}{2} \!\right)} \quad (k=1,\ldots, K-1). \label{ep}
\end{align}
We show that the equalities \eqref{eu} and \eqref{ep} hold at $k=0,K$. 
We remark that the equations \eqref{Eq_i}--\eqref{Eq_ii} hold in the interior of the domain $(0,L)$ only. 
Hence, we cannot apply the equations \eqref{Eq_i}--\eqref{Eq_ii} directly in the calculation of \eqref{eu} and \eqref{ep} on the boundary. 
Therefore, we consider points slightly inside from the boundary of the domain, and we take the limit of them to show that \eqref{eu} and \eqref{ep} hold at $k = 0,K$. 
For any $\varepsilon \! \in \! (0,1)$, let 
\begin{gather*}
e_{u,0,\varepsilon}^{(n)} := U_{0}^{(n)} - u(\varepsilon \Delta x, n\Delta t), \quad e_{u,K,-\varepsilon}^{(n)} := U_{K}^{(n)} - u((K - \varepsilon) \Delta x, n\Delta t) \quad (n=0,1,\ldots, N), \\
e_{p,0,\varepsilon}^{(n)} \! := \! P_{0}^{(n)} \! - p\!\left(\! \varepsilon\Delta x, \left(\! n \! + \! \frac{1}{2}\!\right)\!\Delta t\!\right), \quad e_{p,K,-\varepsilon}^{(n)} \! := \! P_{K}^{(n)} \! - p\!\left(\! (K \! - \! \varepsilon)\Delta x, \left(\! n \! + \! \frac{1}{2}\!\right)\!\Delta t\!\right) \quad (n=0,1,\ldots, N-1).
\end{gather*}
Furthermore, for $n=0,1,\ldots, N-1$, let 
\begin{gather*}
\xi_{1,\varepsilon }^{\left( n + \frac{1}{2}\right)} := \partial_{t}u_{\varepsilon}^{\left( n + \frac{1}{2}\right)}  - \frac{u_{\varepsilon}^{(n+1)} - u_{\varepsilon}^{(n)}}{\Delta t} + \delta_{k}^{\langle 2 \rangle}\tilde{p}_{\varepsilon}^{\left( n + \frac{1}{2}\right)}  - \partial_{x}^{2}p_{\varepsilon}^{\left( n + \frac{1}{2}\right)} , \\
\xi_{1,K-\varepsilon}^{\left( n + \frac{1}{2}\right)} := \partial_{t}u_{K-\varepsilon}^{\left( n + \frac{1}{2}\right)}  - \frac{u_{K-\varepsilon}^{(n+1)} - u_{K-\varepsilon}^{(n)}}{\Delta t} + \delta_{k}^{\langle 2 \rangle}\tilde{p}_{K-\varepsilon}^{\left( n + \frac{1}{2}\right)}  - \partial_{x}^{2}p_{K-\varepsilon}^{\left( n + \frac{1}{2}\right)} , \\
\xi_{2,\varepsilon}^{\left( n + \frac{1}{2}\right)} := \gamma\left\{\partial_{x}^{2}u_{\varepsilon}^{\left( n + \frac{1}{2}\right)} - \delta_{k}^{\langle 2 \rangle}\left(\frac{\tilde{u}_{\varepsilon}^{(n+1)} + \tilde{u}_{\varepsilon}^{(n)}}{2}\right) \right\} + \frac{dF}{d(u_{\varepsilon}^{(n+1)},u_{\varepsilon}^{(n)})} - F'( u_{\varepsilon}^{\left( n + \frac{1}{2}\right)}), \\
\xi_{2,K-\varepsilon}^{\left( n + \frac{1}{2}\right)} := \gamma\left\{\partial_{x}^{2}u_{K-\varepsilon}^{\left( n + \frac{1}{2}\right)} - \delta_{k}^{\langle 2 \rangle}\left(\frac{\tilde{u}_{K-\varepsilon}^{(n+1)} + \tilde{u}_{K-\varepsilon}^{(n)}}{2}\right) \right\} + \frac{dF}{d(u_{K-\varepsilon}^{(n+1)},u_{K-\varepsilon}^{(n)})} - F'( u_{K-\varepsilon}^{\left( n + \frac{1}{2}\right)}). 
\end{gather*}
In a similar way as \eqref{eu}, we have 
\begin{equation}
\frac{e_{u,0,\varepsilon}^{(n+1)} - e_{u,0,\varepsilon}^{(n)}}{\Delta t} 
	= \delta_{k}^{\langle 2 \rangle}e_{p,0,\varepsilon}^{(n)} + \xi_{1,\varepsilon}^{\left( n + \frac{1}{2}\right)}, \quad 
\frac{e_{u,K,-\varepsilon}^{(n+1)} - e_{u,K,-\varepsilon}^{(n)}}{\Delta t} 
	= \delta_{k}^{\langle 2 \rangle}e_{p,K,-\varepsilon}^{(n)} + \xi_{1,K-\varepsilon}^{\left( n + \frac{1}{2}\right)}. \label{eu_ep0K}
\end{equation}
From the smoothness assumption of $u$, letting $\varepsilon$ tend to zero in \eqref{eu_ep0K}, we obtain 
\begin{equation*}
\frac{e_{u,0}^{(n+1)} - e_{u,0}^{(n)}}{\Delta t} = \delta_{k}^{\langle 2 \rangle}e_{p,0}^{(n)} + \xi_{1,0}^{\left( n + \frac{1}{2}\right)}, \quad \frac{e_{u,K}^{(n+1)} - e_{u,K}^{(n)}}{\Delta t} = \delta_{k}^{\langle 2 \rangle}e_{p,K}^{(n)} + \xi_{1,K}^{\left( n + \frac{1}{2}\right)} \quad (n=0,1,\ldots, N-1). 
\end{equation*}
In a similar way as \eqref{ep}, we get 
\begin{gather}
e_{p,0,\varepsilon}^{(n)} 
	= - \gamma \delta_{k}^{\langle 2 \rangle}\!\left(\frac{e_{u,0,\varepsilon}^{(n+1)} + e_{u,0,\varepsilon}^{(n)}}{2}\right) + \left( \frac{dF}{d(U_{0}^{(n+1)},U_{0}^{(n)})} - \frac{dF}{d(u_{\varepsilon}^{(n+1)},u_{\varepsilon}^{(n)})}\right) + \xi_{2,\varepsilon}^{\left( n + \frac{1}{2}\right)}, \label{ep_ep0}\\ 
e_{p,K,-\varepsilon}^{(n)} 
	= - \gamma \delta_{k}^{\langle 2 \rangle}\!\left(\frac{e_{u,K,-\varepsilon}^{(n+1)} + e_{u,K,-\varepsilon}^{(n)}}{2}\right) + \left( \frac{dF}{d(U_{K}^{(n+1)},U_{K}^{(n)})} - \frac{dF}{d(u_{K-\varepsilon}^{(n+1)},u_{K-\varepsilon}^{(n)})}\right) + \xi_{2,K-\varepsilon}^{\left( n + \frac{1}{2}\right)}. \label{ep_epK}
\end{gather}
From the smoothness assumptions of $u$ and $F$, letting $\varepsilon$ tend to zero in \eqref{ep_ep0} and \eqref{ep_epK}, we obtain 
\begin{gather*}
e_{p,0}^{(n)} 
	= - \gamma \delta_{k}^{\langle 2 \rangle}\!\left(\frac{e_{u,0}^{(n+1)} + e_{u,0}^{(n)}}{2}\right) + \left(\frac{dF}{d( U_{0}^{(n+1)},U_{0}^{(n)})} - \frac{dF}{d( u_{0}^{(n+1)},u_{0}^{(n)})}\right) + \xi_{2,0}^{\left( n + \frac{1}{2}\right)}, \\
e_{p,K}^{(n)} 
	= - \gamma \delta_{k}^{\langle 2 \rangle}\!\left(\frac{e_{u,K}^{(n+1)} + e_{u,K}^{(n)}}{2}\right) + \left(\frac{dF}{d( U_{K}^{(n+1)},U_{K}^{(n)})} - \frac{dF}{d( u_{K}^{(n+1)},u_{K}^{(n)})}\right) + \xi_{2,K}^{\left( n + \frac{1}{2}\right)} 
\end{gather*}
for $n=0,1,\ldots, N-1$. 
Next, from the definition of $\bm{e}_{u}$, \eqref{BCi}, and \eqref{dDBC}, we have 
\begin{align*}
\frac{e_{u,0}^{(\! n+1\! )} \! - \! e_{u,0}^{(\! n\! )}}{\Delta t}
%	& = \frac{U_{0}^{(n+1)} - U_{0}^{(n)}}{\Delta t} - \frac{u_{0}^{(n+1)} - u_{0}^{(n)}}{\Delta t} \\
	& = \frac{U_{0}^{(n+1)} - U_{0}^{(n)}}{\Delta t} - \partial_{t}u_{0}^{\left( n + \frac{1}{2}\right)} + \partial_{t}u_{0}^{\left( n + \frac{1}{2}\right)} - \frac{u_{0}^{(n+1)} - u_{0}^{(n)}}{\Delta t} \\
	& = \delta_{k}^{\langle 1 \rangle}\left(\frac{U_{0}^{(n+1)} + U_{0}^{(n)}}{2}\right) - \partial_{x}u_{0}^{\left( n + \frac{1}{2}\right)} + \partial_{t}u_{0}^{\left( n + \frac{1}{2}\right)} - \frac{u_{0}^{(n+1)} - u_{0}^{(n)}}{\Delta t} \\
	& = \delta_{k}^{\langle 1 \rangle}\!\!\left(\!\! \frac{U_{0}^{(\! n+1\! )} \!\! + \! U_{0}^{(\! n\! )}}{2} \!\!\right) \! - \! \delta_{k}^{\langle 1 \rangle}\!\!\left(\!\! \frac{\tilde{u}_{0}^{(\! n+1\! )} \!\! + \! \tilde{u}_{0}^{(\! n\! )}}{2} \!\!\right) \! + \! \delta_{k}^{\langle 1 \rangle}\!\!\left(\!\! \frac{\tilde{u}_{0}^{(\! n+1\! )} \!\! + \! \tilde{u}_{0}^{(\! n\! )}}{2} \!\!\right) \! - \! \partial_{x}u_{0}^{\! \left(\! n + \frac{1}{2} \!\right)} \! + \! \partial_{t}u_{0}^{\! \left(\! n + \frac{1}{2} \!\right)} \! - \! \frac{u_{0}^{(\! n+1\! )} \! - \! u_{0}^{(\! n\! )}}{\Delta t} \\
	& = \delta_{k}^{\langle 1 \rangle}\left(\frac{e_{u,0}^{(n+1)} + e_{u,0}^{(n)}}{2}\right) + \xi_{3,0}^{\left( n + \frac{1}{2}\right)} \quad (n=0,1,\ldots, N-1). 
\end{align*}
In the same manner, from the definition of $\bm{e}_{u}$, \eqref{BCi'}, and \eqref{dDBC'}, we get 
\begin{equation*}
\frac{e_{u,K}^{(n+1)} - e_{u,K}^{(n)}}{\Delta t}
	= -\delta_{k}^{\langle 1 \rangle}\left(\frac{e_{u,K}^{(n+1)} + e_{u,K}^{(n)}}{2}\right) + \xi_{3,K}^{\left( n + \frac{1}{2}\right)} \quad (n=0,1,\ldots, N-1).
\end{equation*}
Lastly, it holds from the definition of $\bm{e}_{p}$, \eqref{dNBC}, and \eqref{pb} that 
\begin{equation*}
0 = \delta_{k}^{\langle 1 \rangle}P_{k}^{(n)} = \delta_{k}^{\langle 1 \rangle}e_{p,k}^{(n)} + \delta_{k}^{\langle 1 \rangle}\tilde{p}_{k}^{\left( n + \frac{1}{2}\right)} = \delta_{k}^{\langle 1 \rangle}e_{p,k}^{(n)} \quad (k=0,K,\ n=0,1,\ldots, N-1). 
\end{equation*}
From the above, equations \eqref{eu1}--\eqref{ep_bc} on the errors $\bm{e}_{u}$ and $\bm{e}_{p}$ hold. \hfill $\Box$\\
\end{proof}

\begin{lemma} \label{Ap:G} 
Assume that $u \in C^{4}([0,L] \times [0,T])$.  
Furthermore, we suppose that the potential function $F$ is in $C^{3}$.  
Denote the bounds by \eqref{Bou}. 
Then, for any fixed $\varepsilon >0$, the following inequality holds: 
\begin{equation*}
\left\{ 1 - \Delta t\!\left( \frac{C_{3}^{2}}{\gamma} + \varepsilon \right) \right\} \left\| D\bm{e}_{u}^{(n+1)}\right\|^{2} 
	\leq \left\{ 1 + \Delta t\left( \frac{C_{3}^{2}}{\gamma} + \varepsilon \right) \right\}\left\| D\bm{e}_{u}^{(n)}\right\|^{2} + \Delta t R^{\left( n + \frac{1}{2}\right)} \quad (n=0,1, \ldots, N-1).
\end{equation*}
\end{lemma}

\begin{proof}
For any fixed $n=0,1,\ldots, N-1$, using Corollary \ref{col:2.1}, we have 
\begin{align}
& \frac{1}{2\Delta t}\left( \left\| D\bm{e}_{u}^{(n+1)}\right\|^{2} - \left\| D\bm{e}_{u}^{(n)}\right\|^{2} \right) 
	= \sum_{k=0}^{K-1}\left\{\delta_{k}^{+}\left( \frac{e_{u,k}^{(n+1)} - e_{u,k}^{(n)}}{\Delta t}\right)\right\}\left\{\delta_{k}^{+}\left( \frac{e_{u,k}^{(n+1)} + e_{u,k}^{(n)}}{2} \right)\right\}\Delta x \nonumber\\
	& = -\sum_{k=0}^{K}{}^{\prime\prime}\frac{e_{u,k}^{(n+1)} - e_{u,k}^{(n)}}{\Delta t}\delta_{k}^{\langle 2\rangle}\left( \frac{e_{u,k}^{(n+1)} + e_{u,k}^{(n)}}{2} \right)\Delta x + \left[\frac{e_{u,k}^{(n+1)} - e_{u,k}^{(n)}}{\Delta t}\delta_{k}^{\langle 1\rangle}\left( \frac{e_{u,k}^{(n+1)} + e_{u,k}^{(n)}}{2} \right)\right]_{0}^{K}. \label{Derr1}
\end{align}
Firstly, we consider the first term on the right-hand side of \eqref{Derr1}. 
From \eqref{eu1}, \eqref{ep1}, \eqref{ep_bc}, Corollary \ref{col:2.1}, and H\"older inequality, we obtain 
\begin{align*}
& -\sum_{k=0}^{K}{}^{\prime\prime}\frac{e_{u,k}^{(n+1)} - e_{u,k}^{(n)}}{\Delta t}\delta_{k}^{\langle 2\rangle}\left( \frac{e_{u,k}^{(n+1)} + e_{u,k}^{(n)}}{2} \right)\Delta x \\
	& = -\sum_{k=0}^{K}{}^{\prime\prime}\left(\delta_{k}^{\langle 2 \rangle}e_{p,k}^{(n)}\right)\delta_{k}^{\langle 2\rangle}\left( \frac{e_{u,k}^{(n+1)} + e_{u,k}^{(n)}}{2} \right)\Delta x -\sum_{k=0}^{K}{}^{\prime\prime}\xi_{1,k}^{\left( n + \frac{1}{2}\right)}\delta_{k}^{\langle 2\rangle}\left( \frac{e_{u,k}^{(n+1)} + e_{u,k}^{(n)}}{2} \right)\Delta x \\
	& = \frac{1}{\gamma}\sum_{k=0}^{K}{}^{\prime\prime}\!\left(\delta_{k}^{\langle 2 \rangle}e_{p,k}^{(n)}\right) e_{p,k}^{(n)} \Delta x - \frac{1}{\gamma}\sum_{k=0}^{K}{}^{\prime\prime}\!\left(\delta_{k}^{\langle 2 \rangle}e_{p,k}^{(n)}\right)\!\left\{\! \left(\frac{dF}{d( U_{k}^{(n+1)},U_{k}^{(n)})} - \frac{dF}{d( u_{k}^{(n+1)},u_{k}^{(n)})}\right) + \xi_{2,k}^{\left( n + \frac{1}{2}\right)}\!\right\}\Delta x \\
	& \quad - \sum_{k=0}^{K}{}^{\prime\prime}\xi_{1,k}^{\left( n + \frac{1}{2}\right)}\delta_{k}^{\langle 2\rangle}\left( \frac{e_{u,k}^{(n+1)} + e_{u,k}^{(n)}}{2} \right)\Delta x 
\end{align*}
\begin{align*}
	& = \! -\frac{1}{\gamma}\!\sum_{k=0}^{K-1}\!\!\left(\! \delta_{k}^{+}e_{p,k}^{(\! n\! )} \!\right)^{\!\! 2} \!\! \Delta x \! + \! \frac{1}{\gamma}\!\left[\! \left(\! \delta_{k}^{\langle 1 \rangle}e_{p,k}^{(\! n\! )} \!\right)\! e_{p,k}^{(\! n\! )} \!\right]_{0}^{K} \!\! + \! \frac{1}{\gamma}\!\sum_{k=0}^{K-1}\!\!\left(\! \delta_{k}^{+}e_{p,k}^{(\! n\! )} \!\right)\!\!\left\{\! \delta_{k}^{+}\!\!\left(\! \frac{dF}{d( U_{k}^{(\! n+1\! )}\! ,\! U_{k}^{(\! n\! )})} \! - \! \frac{dF}{d( u_{k}^{(\! n+1\! )}\! ,\! u_{k}^{(\! n\! )})} \! + \! \xi_{2,k}^{\left(\! n + \frac{1}{2} \!\!\right)} \!\right)\!\!\right\}\!\Delta x \\
	& \quad - \frac{1}{\gamma}\left[\left(\delta_{k}^{\langle 1 \rangle}e_{p,k}^{(n)}\right)\left(\frac{dF}{d( U_{k}^{(n+1)},U_{k}^{(n)})} - \frac{dF}{d( u_{k}^{(n+1)},u_{k}^{(n)})} + \xi_{2,k}^{\left( n + \frac{1}{2}\right)}\right)\right]_{0}^{K} \\
	& \quad + \sum_{k=0}^{K-1}\left(\delta_{k}^{+}\xi_{1,k}^{\left( n + \frac{1}{2}\right)}\right)\left\{\delta_{k}^{+}\left( \frac{e_{u,k}^{(n+1)} + e_{u,k}^{(n)}}{2} \right)\right\}\Delta x - \left[\xi_{1,k}^{\left( n + \frac{1}{2}\right)}\delta_{k}^{\langle 1 \rangle}\left( \frac{e_{u,k}^{(n+1)} + e_{u,k}^{(n)}}{2} \right)\right]_{0}^{K} \\
	& \leq -\frac{1}{\gamma}\left\| D\bm{e}_{p}^{(n)}\right\|^{2} + \frac{1}{\gamma}\left\| D\bm{e}_{p}^{(n)}\right\|\left\| D\left(\frac{dF}{d\left( \bm{U}^{(n+1)},\bm{U}^{(n)}\right)} - \frac{dF}{d\left( \bm{u}^{(n+1)},\bm{u}^{(n)}\right)} + \bm{\xi}_{2}^{\left( n + \frac{1}{2}\right)}\right)\right\| \\
	& \quad + \left\| D\bm{\xi}_{1}^{\left( n + \frac{1}{2}\right)}\right\|\left\| D\left( \frac{\bm{e}_{u}^{(n+1)} + \bm{e}_{u}^{(n)}}{2} \right)\right\| - \left[\xi_{1,k}^{\left( n + \frac{1}{2}\right)}\delta_{k}^{\langle 1 \rangle}\left( \frac{e_{u,k}^{(n+1)} + e_{u,k}^{(n)}}{2} \right)\right]_{0}^{K}. 
\end{align*}
Next, we consider the second term on the right-hand side of \eqref{Derr1}. 
It follows from \eqref{eu0} and \eqref{euK} that 
\begin{align*}
\left[\frac{e_{u,k}^{(n+1)} - e_{u,k}^{(n)}}{\Delta t}\delta_{k}^{\langle 1\rangle}\left( \frac{e_{u,k}^{(n+1)} + e_{u,k}^{(n)}}{2} \right)\right]_{0}^{K}
	& = - \left\{\! \delta_{k}^{\langle 1 \rangle}\!\!\left(\! \frac{e_{u,K}^{(n+1)} \! + \! e_{u,K}^{(n)}}{2} \!\right) \!\right\}^{\!\! 2} - \left\{\! \delta_{k}^{\langle 1 \rangle}\!\!\left(\! \frac{e_{u,0}^{(n+1)} \! + \! e_{u,0}^{(n)}}{2} \!\right) \!\right\}^{\!\! 2} \\
	& \quad + \xi_{3,K}^{\left( n + \frac{1}{2} \right)}\delta_{k}^{\langle 1\rangle}\!\!\left(\! \frac{e_{u,K}^{(n+1)} \! + \! e_{u,K}^{(n)}}{2} \!\right) - \xi_{3,0}^{\left( n + \frac{1}{2} \right)}\delta_{k}^{\langle 1\rangle}\!\!\left(\! \frac{e_{u,0}^{(n+1)} \! + \! e_{u,0}^{(n)}}{2} \!\right). 
\end{align*}
From the above, we obtain 
\begin{align*}
& \frac{1}{2\Delta t}\left( \left\| D\bm{e}_{u}^{(n+1)}\right\|^{2} - \left\| D\bm{e}_{u}^{(n)}\right\|^{2} \right) \\
	& \leq -\frac{1}{\gamma}\left\| D\bm{e}_{p}^{(n)}\right\|^{2} + \frac{1}{\gamma}\left\| D\bm{e}_{p}^{(n)}\right\|\left\| D\left(\frac{dF}{d\left( \bm{U}^{(n+1)},\bm{U}^{(n)}\right)} - \frac{dF}{d\left( \bm{u}^{(n+1)},\bm{u}^{(n)}\right)} + \bm{\xi}_{2}^{\left( n + \frac{1}{2}\right)}\right)\right\| \\
	& \quad + \left\| D\bm{\xi}_{1}^{\left( n + \frac{1}{2}\right)}\right\|\left\| D\left( \frac{\bm{e}_{u}^{(n+1)} + \bm{e}_{u}^{(n)}}{2} \right)\right\| \! - \!\left\{\!\delta_{k}^{\langle 1 \rangle}\!\!\left(\!\frac{e_{u,K}^{( n+1 )} \! + \! e_{u,K}^{(n)}}{2}\!\right)\!\right\}^{\!\! 2} \! - \! \left\{\! \delta_{k}^{\langle 1 \rangle}\!\!\left(\! \frac{e_{u,0}^{( n+1 )} \! + \! e_{u,0}^{(n)}}{2}\!\right)\!\right\}^{\!\! 2} \\
	& \quad + \left( - \xi_{1,K}^{\left( n + \frac{1}{2}\right)} + \xi_{3,K}^{\left( n + \frac{1}{2}\right)} \right) \delta_{k}^{\langle 1 \rangle}\left( \frac{e_{u,K}^{(n+1)} + e_{u,K}^{(n)}}{2} \right) + \left(\xi_{1,0}^{\left( n + \frac{1}{2}\right)} - \xi_{3,0}^{\left( n + \frac{1}{2}\right)} \right)\delta_{k}^{\langle 1 \rangle}\left( \frac{e_{u,0}^{(n+1)} + e_{u,0}^{(n)}}{2} \right).  
\end{align*}
From the above inequality, the Young inequality, and the inequality: $(a+b)^{2} \leq 2(a^{2} + b^{2})$ for all $a, b \in \mathbb{R}$, it holds that 
\begin{align*}
& \frac{1}{2\Delta t}\left( \left\| D\bm{e}_{u}^{(n+1)}\right\|^{2} - \left\| D\bm{e}_{u}^{(n)}\right\|^{2} \right) + \frac{1}{\gamma}\left\| D\bm{e}_{p}^{(n)}\right\|^{2} + \left\{\delta_{k}^{\langle 1 \rangle}\left(\frac{e_{u,K}^{(n+1)} + e_{u,K}^{(n)}}{2}\right)\right\}^{2} + \left\{\delta_{k}^{\langle 1 \rangle}\left(\frac{e_{u,0}^{(n+1)} + e_{u,0}^{(n)}}{2}\right)\right\}^{2} \\
	& \leq \frac{1}{\gamma}\left\| D\bm{e}_{p}^{(n)}\right\|\left\| D\left(\frac{dF}{d\left( \bm{U}^{(n+1)},\bm{U}^{(n)}\right)} - \frac{dF}{d\left( \bm{u}^{(n+1)},\bm{u}^{(n)}\right)} + \bm{\xi}_{2}^{\left( n + \frac{1}{2}\right)}\right)\right\| + \left\| D\bm{\xi}_{1}^{\left( n + \frac{1}{2}\right)}\right\|\left\| D\left(\! \frac{\bm{e}_{u}^{(n+1)} \! + \! \bm{e}_{u}^{(n)}}{2} \!\right)\right\| \\
	& \quad + \left( - \xi_{1,K}^{\left( n + \frac{1}{2}\right)} + \xi_{3,K}^{\left( n + \frac{1}{2}\right)} \right) \delta_{k}^{\langle 1 \rangle}\left( \frac{e_{u,K}^{(n+1)} + e_{u,K}^{(n)}}{2} \right) + \left(\xi_{1,0}^{\left( n + \frac{1}{2}\right)} - \xi_{3,0}^{\left( n + \frac{1}{2}\right)} \right)\delta_{k}^{\langle 1 \rangle}\left( \frac{e_{u,0}^{(n+1)} + e_{u,0}^{(n)}}{2} \right) 
\end{align*}
\begin{align*}
	& \leq \frac{1}{\gamma}\!\left\| D\bm{e}_{p}^{(n)}\right\|^{2} \!\! + \! \frac{1}{4\gamma}\!\left\| D\!\!\left(\!\frac{dF}{d\!\left( \bm{U}^{(n+1)},\bm{U}^{(n)}\right)} - \frac{dF}{d\!\left( \bm{u}^{(n+1)},\bm{u}^{(n)}\right)} + \bm{\xi}_{2}^{\left( n + \frac{1}{2}\right)}\!\right)\!\right\|^{2} \!\! + \! \left\| D\bm{\xi}_{1}^{\left( n + \frac{1}{2}\right)}\right\|\!\left\| D\!\!\left(\! \frac{\bm{e}_{u}^{(n+1)} \! + \! \bm{e}_{u}^{(n)}}{2} \!\right)\!\right\| \\
	& \quad + \frac{1}{2}\!\left\{ \left| \xi_{1,K}^{\left(\! n + \frac{1}{2} \!\right)} \right|^{2} \! + \! \left| \xi_{3,K}^{\left(\! n + \frac{1}{2} \!\right)} \right|^{2} \right\} \! + \! \left\{\! \delta_{k}^{\langle 1 \rangle}\!\!\left(\! \frac{e_{u,K}^{(\! n+1\! )} \! + \! e_{u,K}^{(n)}}{2} \!\right) \!\!\right\}^{\!\! 2} \! + \! \frac{1}{2}\!\left\{ \left| \xi_{1,0}^{\left(\! n + \frac{1}{2} \!\right)}\right|^{2} \! + \! \left| \xi_{3,0}^{\left(\! n + \frac{1}{2} \!\right)} \right|^{2} \right\} \! + \! \left\{\! \delta_{k}^{\langle 1 \rangle}\!\!\left(\! \frac{e_{u,0}^{(\! n+1\! )} \! + \! e_{u,0}^{(n)}}{2} \!\right) \!\!\right\}^{\!\! 2}\! .
\end{align*}
Namely, 
\begin{align}
& \frac{1}{2\Delta t}\left( \left\| D\bm{e}_{u}^{(n+1)}\right\|^{2} - \left\| D\bm{e}_{u}^{(n)}\right\|^{2} \right) \nonumber\\
& \leq \frac{1}{4\gamma}\left\| D\left(\frac{dF}{d\left( \bm{U}^{(n+1)},\bm{U}^{(n)}\right)} - \frac{dF}{d\left( \bm{u}^{(n+1)},\bm{u}^{(n)}\right)} + \bm{\xi}_{2}^{\left( n + \frac{1}{2}\right)}\right)\right\|^{2} + \left\| D\bm{\xi}_{1}^{\left( n + \frac{1}{2}\right)}\right\|\left\| D\left( \frac{\bm{e}_{u}^{(n+1)} + \bm{e}_{u}^{(n)}}{2} \right)\right\| \nonumber\\
& \quad + \frac{1}{2}\left\{\left|\xi_{1,0}^{\left( n + \frac{1}{2}\right)}\right|^{2} + \left|\xi_{1,K}^{\left( n + \frac{1}{2}\right)}\right|^{2} + \left|\xi_{3,0}^{\left( n + \frac{1}{2}\right)}\right|^{2} + \left|\xi_{3,K}^{\left( n + \frac{1}{2}\right)}\right|^{2}\right\}. \label{Derr2}
\end{align}
We consider the difference quotient of $F$. 
Now, using Lemma \ref{lem:4.2}, we have 
\begin{equation*}
\frac{dF}{d( U_{k}^{(\! n+1\! )}, U_{k}^{(\! n\! )})} - \frac{dF}{d( u_{k}^{(\! n+1\! )}, u_{k}^{(\! n\! )})} \! = \! \frac{1}{2}\bar{F}'' \!\! \left( U_{k}^{(\! n+1\! )}, u_{k}^{(\! n+1\! )}; U_{k}^{(\! n\! )}, u_{k}^{(\! n\! )}\right) \! e_{u,k}^{(\! n+1\! )} + \frac{1}{2}\bar{F}'' \!\! \left( U_{k}^{(\! n\! )}, u_{k}^{(\! n\! )}; U_{k}^{(\! n+1\! )}, u_{k}^{(\! n+1\! )}\right) \! e_{u,k}^{(\! n\! )} 
\end{equation*}
for $k=0,\ldots, K$. 
Hence, it follows from Lemma \ref{lem:4.3} that 
\begin{align*}
& \left\| D\left(\frac{dF}{d\left( \bm{U}^{(n+1)},\bm{U}^{(n)}\right)} - \frac{dF}{d\left( \bm{u}^{(n+1)},\bm{u}^{(n)}\right)}\right)\right\| \\
%& = \frac{1}{2}\left\| D\left( \bar{F}'' \left( \bm{U}^{(n+1)}, \bm{u}^{(n+1)}; \bm{U}^{(n)}, \bm{u}^{(n)}\right) \bm{e}_{u}^{(n+1)} + \bar{F}'' \left( \bm{U}^{(n)}, \bm{u}^{(n)}; \bm{U}^{(n+1)}, \bm{u}^{(n+1)}\right) \bm{e}_{u}^{(n)} \right)\right\| \\
& \leq \frac{1}{2}\left\| D\left( \bar{F}'' \left( \bm{U}^{(n+1)}, \bm{u}^{(n+1)}; \bm{U}^{(n)}, \bm{u}^{(n)}\right) \bm{e}_{u}^{(n+1)}\right)\right\| + \frac{1}{2}\left\| D\left( \bar{F}'' \left( \bm{U}^{(n)}, \bm{u}^{(n)}; \bm{U}^{(n+1)}, \bm{u}^{(n+1)}\right) \bm{e}_{u}^{(n)} \right)\right\| \\
& \leq \frac{1}{2}\left\| D \bar{F}''\!\left( \bm{U}^{(n+1)}, \bm{u}^{(n+1)}; \bm{U}^{(n)}, \bm{u}^{(n)}\!\right)\!\right\|\!\left\|\bm{e}_{u}^{(n+1)}\right\|_{L_{\rm d}^{\infty}} \! + \frac{1}{2}\left\| \bar{F}'' \!\left( \bm{U}^{(n+1)}, \bm{u}^{(n+1)}; \bm{U}^{(n)}, \bm{u}^{(n)}\!\right)\!\right\|_{L_{\rm d}^{\infty}}\!\left\| D\bm{e}_{u}^{(n+1)}\right\|\\
& \quad + \frac{1}{2}\left\| D \bar{F}''\!\left( \bm{U}^{(n)}, \bm{u}^{(n)}; \bm{U}^{(n+1)}, \bm{u}^{(n+1)}\right) \right\|\!\left\|\bm{e}_{u}^{(n)}\right\|_{L_{\rm d}^{\infty}} + \frac{1}{2}\left\| \bar{F}''\!\left( \bm{U}^{(n)}, \bm{u}^{(n)}; \bm{U}^{(n+1)}, \bm{u}^{(n+1)}\right) \right\|_{L_{\rm d}^{\infty}}\!\left\| D\bm{e}_{u}^{(n)}\right\|. 
\end{align*}
From \eqref{Bou} and Lemma \ref{lem:4.1}, we have 
\begin{equation*}
\left\| \bar{F}'' \!\!\left(\! \bm{U}^{(n+1)}, \bm{u}^{(n+1)}; \bm{U}^{(n)}, \bm{u}^{(n)}\!\right)\!\right\|_{L_{\rm d}^{\infty}} \leq C_{F,2}, \quad \left\| \bar{F}''\!\!\left(\! \bm{U}^{(n)}, \bm{u}^{(n)}; \bm{U}^{(n+1)}, \bm{u}^{(n+1)}\!\right) \!\right\|_{L_{\rm d}^{\infty}} \leq C_{F,2}.
\end{equation*}
Moreover, from \eqref{Bou} and Lemma \ref{lem:4.4}, we obtain 
\begin{gather*}
\left\| D \bar{F}''\left( \bm{U}^{(n+1)}, \bm{u}^{(n+1)}; \bm{U}^{(n)}, \bm{u}^{(n)}\right)\right\| \leq \frac{C_{F,3}}{6}\left( 2C_{1} + 2C_{1} + C_{1}+ C_{1}\right) = C_{1}C_{F,3}, \\
\left\| D \bar{F}''\left( \bm{U}^{(n)}, \bm{u}^{(n)}; \bm{U}^{(n+1)}, \bm{u}^{(n+1)}\right) \right\| \leq \frac{C_{F,3}}{6}\left( 2C_{1} + 2C_{1} + C_{1}+ C_{1}\right) = C_{1}C_{F,3}. 
\end{gather*} 
From the above, it holds that 
\begin{equation}
\left\| D\!\left(\! \frac{dF}{d\!\left( \bm{U}^{(\! n+1\! )},\bm{U}^{(\! n\! )}\right)} \! - \! \frac{dF}{d\!\left( \bm{u}^{(\! n+1\! )},\bm{u}^{(\! n\! )}\right)} \!\right) \!\right\| 
	\! \leq \! \frac{ C_{1}C_{F,3} }{2}\!\left(\! \left\|\bm{e}_{u}^{(\! n+1\! )}\right\|_{L_{\rm d}^{\infty}} \!\! + \! \left\|\bm{e}_{u}^{(\! n\! )}\right\|_{L_{\rm d}^{\infty}} \!\right) + \frac{ C_{F,2} }{2}\!\left( \left\| D\bm{e}_{u}^{(\! n+1\! )}\right\| \! + \left\| D\bm{e}_{u}^{(\! n\! )}\right\| \right). \label{Ddq}
\end{equation}
Next, we consider $\|\bm{e}_{u}^{(n+1)}\|_{L_{\rm d}^{\infty}}$ and $\|\bm{e}_{u}^{(n)}\|_{L_{\rm d}^{\infty}}$. 
From the same argument as \eqref{err_Linf1} in Theorem \ref{thm:5.1}, we have 
\begin{equation}
\left\|\bm{e}_{u}^{(n)}\right\|_{L_{\rm d}^{\infty}} 
%\leq \frac{1}{L}\sum_{k=0}^{K}{}^{\prime\prime}e_{u,k}^{(n)}\Delta x + L^{\frac{1}{2}}\left\| D\bm{e}_{u}^{(n)}\right\| 
\leq \Delta t\sum_{j=0}^{n-1}\left\|\bm{\xi}_{1}^{\left( j + \frac{1}{2}\right)}\right\|_{L_{\rm d}^{\infty}} + L^{\frac{1}{2}}\left\| D\bm{e}_{u}^{(n)}\right\| \quad (n=1,\ldots, N). \label{err_Linf}
\end{equation}
Applying \eqref{err_Linf} to \eqref{Ddq}, we obtain 
\begin{equation*}
\left\| D\left(\frac{dF}{d\left( \bm{U}^{(n+1)},\bm{U}^{(n)}\right)} - \frac{dF}{d\left( \bm{u}^{(n+1)},\bm{u}^{(n)}\right)}\right)\right\| 
	\leq C_{3}\left( \left\| D\bm{e}_{u}^{(n+1)}\right\| + \left\| D\bm{e}_{u}^{(n)}\right\| \right) + \Delta t C_{1}C_{F,3} \sum_{j=0}^{n}\left\|\bm{\xi}_{1}^{\left( j + \frac{1}{2}\right)}\right\|_{L_{\rm d}^{\infty}} 
\end{equation*}
for $n=0,1,\ldots, N-1$.
Therefore, we have 
\begin{align}
& \frac{1}{4\gamma}\left\| D\left(\frac{dF}{d\left( \bm{U}^{(n+1)},\bm{U}^{(n)}\right)} - \frac{dF}{d\left( \bm{u}^{(n+1)},\bm{u}^{(n)}\right)} + \bm{\xi}_{2}^{\left( n + \frac{1}{2}\right)} \right)\right\|^{2} \nonumber\\
	& \leq \frac{1}{4\gamma}\left( \left\| D\left(\frac{dF}{d\left( \bm{U}^{(n+1)},\bm{U}^{(n)}\right)} - \frac{dF}{d\left( \bm{u}^{(n+1)},\bm{u}^{(n)}\right)} \right)\right\| + \left\| D\bm{\xi}_{2}^{\left( n + \frac{1}{2}\right)}\right\| \right)^{2} \nonumber\\
	& \leq \frac{1}{4\gamma}\left\{ C_{3}\left( \left\| D\bm{e}_{u}^{(n+1)}\right\| + \left\| D\bm{e}_{u}^{(n)}\right\| \right) + \left(\Delta t C_{1}C_{F,3} \sum_{j=0}^{n}\left\|\bm{\xi}_{1}^{\left( j + \frac{1}{2}\right)}\right\|_{L_{\rm d}^{\infty}} + \left\| D\bm{\xi}_{2}^{\left( n + \frac{1}{2}\right)}\right\| \right) \right\}^{2} . \label{dFxi}
\end{align}
for $n=0,1,\ldots, N-1$. 
For simplicity, let 
\begin{equation*}
R_{1}^{(n)} := \Delta t C_{1}C_{F,3} \sum_{j=0}^{n}\left\|\bm{\xi}_{1}^{\left( j + \frac{1}{2}\right)}\right\|_{L_{\rm d}^{\infty}} + \left\| D\bm{\xi}_{2}^{\left( n + \frac{1}{2}\right)}\right\| \quad (n=0,1,\ldots, N-1). 
\end{equation*}
Let $\varepsilon >0$ be an arbitrarily fixed number. 
From \eqref{dFxi} and the inequality: $(a + b + c)^{2} \leq 2(a^{2} + b^{2}) + (\tilde{\varepsilon}/2)(a^{2} + b^{2}) + ( 1 + (4/\tilde{\varepsilon}) )c^{2}$ for all $a,b,c \in \mathbb{R}$, and $\tilde{\varepsilon} >0$, we obtain 
\begin{align}
& \frac{1}{4\gamma}\left\| D\left(\frac{dF}{d\left( \bm{U}^{(n+1)},\bm{U}^{(n)}\right)} - \frac{dF}{d\left( \bm{u}^{(n+1)},\bm{u}^{(n)}\right)} + \bm{\xi}_{2}^{\left( n + \frac{1}{2}\right)} \right)\right\|^{2} \nonumber\\
	& \leq \frac{1}{4\gamma}\left( C_{3}\left\| D\bm{e}_{u}^{(n+1)}\right\| + C_{3}\left\| D\bm{e}_{u}^{(n)}\right\| + R_{1}^{(n)} \right)^{2} \nonumber\\
	& \leq \frac{1}{4\gamma}\left\{ 2\left( C_{3}^{2}\left\| D\bm{e}_{u}^{(n+1)}\right\|^{2} \! + \! C_{3}^{2}\left\| D\bm{e}_{u}^{(n)}\right\|^{2} \right) \! + \! \frac{\varepsilon\gamma}{C_{3}^{2}}\!\left( C_{3}^{2}\left\| D\bm{e}_{u}^{(n+1)}\right\|^{2} \! + \! C_{3}^{2}\left\| D\bm{e}_{u}^{(n)}\right\|^{2} \right) \! + \! \left(\! 1 \! + \! \frac{2C_{3}^{2}}{\varepsilon \gamma} \!\right) \!\! \left( R_{1}^{(n)} \right)^{2}\right\} \nonumber\\
	& = \frac{C_{3}^{2}}{2\gamma}\left( \left\| D\bm{e}_{u}^{(n+1)}\right\|^{2} + \left\| D\bm{e}_{u}^{(n)}\right\|^{2} \right) + \frac{\varepsilon}{4}\left( \left\| D\bm{e}_{u}^{(n+1)}\right\|^{2} + \left\| D\bm{e}_{u}^{(n)}\right\|^{2} \right) + \frac{1}{4\gamma}\left( 1 + \frac{2C_{3}^{2}}{\varepsilon \gamma} \right) \!\! \left( R_{1}^{(n)} \right)^{2}. \label{Derr2_1st}
\end{align}
In addition, it follows from the Young inequality and the inequality: $(a+b)^{2} \! \leq \! 2(a^{2} + b^{2})$ for all $a, b \in \mathbb{R}$ that 
\begin{align}
\left\| D\bm{\xi}_{1}^{\left( n + \frac{1}{2}\right)}\right\|\left\| D\left( \frac{\bm{e}_{u}^{(n+1)} + \bm{e}_{u}^{(n)}}{2} \right)\right\| 
	& \leq \frac{1}{2}\left\| D\bm{\xi}_{1}^{\left( n + \frac{1}{2}\right)}\right\| \left( \left\| D\bm{e}_{u}^{(n+1)} \right\| + \left\| D\bm{e}_{u}^{(n)}\right\| \right) \nonumber\\
	& \leq \frac{1}{2}\left\{\frac{\varepsilon}{4}\left( \left\| D\bm{e}_{u}^{(n+1)} \right\| + \left\| D\bm{e}_{u}^{(n)}\right\| \right)^{2} + \frac{1}{\varepsilon}\left\| D\bm{\xi}_{1}^{\left( n + \frac{1}{2}\right)}\right\|^{2}\right\} \nonumber\\
	& \leq \frac{\varepsilon}{4}\left( \left\| D\bm{e}_{u}^{(n+1)} \right\|^{2} + \left\| D\bm{e}_{u}^{(n)}\right\|^{2} \right) + \frac{1}{2\varepsilon}\left\| D\bm{\xi}_{1}^{\left( n + \frac{1}{2}\right)}\right\|^{2}. \label{Derr2_2nd}
\end{align}
Consequently, using \eqref{Derr2}, \eqref{Derr2_1st}, and \eqref{Derr2_2nd}, we obtain 
\begin{align*}
& \frac{1}{2\Delta t}\left( \left\| D\bm{e}_{u}^{(n+1)}\right\|^{2} - \left\| D\bm{e}_{u}^{(n)}\right\|^{2} \right) \\
	& \leq \frac{C_{3}^{2}}{2\gamma}\left( \left\| D\bm{e}_{u}^{(n+1)}\right\|^{2} + \left\| D\bm{e}_{u}^{(n)}\right\|^{2} \right) + \frac{\varepsilon}{4}\left( \left\| D\bm{e}_{u}^{(n+1)}\right\|^{2} + \left\| D\bm{e}_{u}^{(n)}\right\|^{2} \right) + \frac{1}{4\gamma}\left( 1 + \frac{2C_{3}^{2}}{\varepsilon \gamma} \right) \!\! \left( R_{1}^{(n)} \right)^{2} \\
	& \quad + \frac{\varepsilon}{4}\left( \left\| D\bm{e}_{u}^{(n+1)} \right\|^{2} \! + \! \left\| D\bm{e}_{u}^{(n)}\right\|^{2} \right) \! + \! \frac{1}{2\varepsilon}\left\| D\bm{\xi}_{1}^{\left(\! n + \frac{1}{2} \!\right)}\right\|^{2} \! + \! \frac{1}{2}\left\{\! \left|\xi_{1,0}^{\left(\! n + \frac{1}{2}\!\right)}\right|^{2} \! + \! \left|\xi_{1,K}^{\left(\! n + \frac{1}{2} \!\right)}\right|^{2} \! + \! \left|\xi_{3,0}^{\left(\! n + \frac{1}{2} \!\right)}\right|^{2} \! + \! \left|\xi_{3,K}^{\left(\! n + \frac{1}{2} \!\right)}\right|^{2} \!\right\} \\
	& = \frac{C_{3}^{2}}{2\gamma}\left( \left\| D\bm{e}_{u}^{(n+1)}\right\|^{2} + \left\| D\bm{e}_{u}^{(n)}\right\|^{2} \right) + \frac{\varepsilon}{2}\left( \left\| D\bm{e}_{u}^{(n+1)}\right\|^{2} + \left\| D\bm{e}_{u}^{(n)}\right\|^{2} \right) + \frac{1}{2}R^{(n)}. 
\end{align*}
Multiplying both sides of the above inequality by $2\Delta t$, we conclude that 
\begin{equation*}
\left\{ 1 - \left( \frac{C_{3}^{2}}{\gamma} + \varepsilon \right)\Delta t \right\} \left\| D\bm{e}_{u}^{(n+1)}\right\|^{2} 
	\leq \left\{ 1 + \left( \frac{C_{3}^{2}}{\gamma} + \varepsilon \right)\Delta t \right\}\left\| D\bm{e}_{u}^{(n)}\right\|^{2} + \Delta t R^{(n)}
\end{equation*}
for $n = 0,1,\ldots, N-1$. \hfill $\Box$\\
\end{proof}

\begin{lemma} \label{Ap:H} 
We impose the same assumption as in Theorem \ref{thm:5.1}. 
Then, the following estimate holds: 
\begin{equation*}
\left| \delta_{k}^{+}\left\{ \partial_{x}^{2}u_{k}^{\left( n + \frac{1}{2}\right)} - \delta_{k}^{\langle 2 \rangle}\!\!\left( \frac{ \tilde{u}_{k}^{ (n + 1) } + \tilde{u}_{k}^{ (n) } }{2}\right) \right\} \right| \leq \frac{ (\Delta x)^{2} }{12} M_{5,0}(u) + \frac{ (\Delta t)^{2} }{4} M_{3,2}(u) \quad (k=0,\ldots, K-1). 
\end{equation*}
\end{lemma}

\begin{proof}
For any $x \in [0,L]$, applying the Taylor theorem to $\tilde{u}$, there exists $\theta_{2} \in (0,1)$ such that  
\begin{align}
\frac{\tilde{u}(x, (n+1)\Delta t) + \tilde{u}(x, n\Delta t)}{2} 
	& = \tilde{u}\!\left( x, \left( n+\frac{1}{2}\right)\!\Delta t\right) \nonumber\\
	& \quad + \frac{(\Delta t)^{2}}{16}\!\left\{\! \partial_{t}^{2}\tilde{u}\!\left( x, \left( n \! + \! \frac{1 \! + \! \theta_{2}}{2}\right)\!\Delta t\right) + \partial_{t}^{2}\tilde{u}\!\left( x, \left( n \! + \! \frac{1 \! - \! \theta_{2}}{2}\right)\!\Delta t\right) \!\right\}. \label{aveu}
\end{align}
Substituting $k\Delta x \ (k=0,\ldots,K)$ into $x$ in \eqref{aveu}, we obtain
\begin{align}
\delta_{k}^{\langle 2 \rangle}\!\!\left(\! \frac{\tilde{u}_{k}^{(\! n+1\! )} \! + \! \tilde{u}_{k}^{(\! n\! )}}{2} \!\right)  \!\! - \! \partial_{x}^{2}u_{k}^{\left(\! n + \frac{1}{2}\!\right)} 
	= \delta_{k}^{\langle 2 \rangle}\tilde{u}_{k}^{ \left(\! n + \frac{1}{2}\!\right) } \! - \! \partial_{x}^{2}u_{k}^{\left(\! n + \frac{1}{2}\!\right)} + \frac{(\Delta t)^{2}}{16}\!\left\{\! \delta_{k}^{\langle 2 \rangle}\!\!\left(\! \partial_{t}^{2}\tilde{u}_{k}^{ \left(\! n + \frac{ 1 + \theta_{2} }{2}\!\right) } \!\right) \! + \! \delta_{k}^{\langle 2 \rangle}\!\!\left(\! \partial_{t}^{2}\tilde{u}_{k}^{ \left(\! n + \frac{ 1 - \theta_{2} }{2}\!\right) } \!\right) \!\!\right\} \label{lap_u_eq}
\end{align}
for $k=0,\ldots, K$. 
Also, for $k=0,\ldots, K$, applying the Taylor theorem to $\tilde{u}$ and using \eqref{util}, there exists $\theta_{3} \in (0,1)$ such that  
\begin{equation}
\delta_{k}^{\langle 2 \rangle}\tilde{u}_{k}^{ \left( n + \frac{1}{2}\right) } - \partial_{x}^{2}u_{k}^{\left( n + \frac{1}{2}\right)} 
	= \begin{cases}
		\displaystyle \frac{ (\Delta x)^{2} }{12}\partial_{x}^{4}u_{\theta_{3}}^{ \left( n + \frac{1}{2}\right) } & (k=0), \\[5pt]
		\displaystyle \frac{ (\Delta x)^{4} }{24}\left(\! \partial_{x}^{4}u_{k+\theta_{3}}^{ \left( n + \frac{1}{2}\right) } + \partial_{x}^{4}u_{k-\theta_{3}}^{ \left( n + \frac{1}{2}\right) } \!\right) & (k=1,\ldots, K-1), \\[5pt]
		\displaystyle \frac{ (\Delta x)^{2} }{12}\partial_{x}^{4}u_{K-\theta_{3}}^{ \left( n + \frac{1}{2}\right) } & (k=K). \label{cd_u_2'}
	\end{cases}
\end{equation}
For details, see \cite{Oku}.  
Since $u \in C^{5}([0,L] \times [0,T])$ from the regularity assumption of $u$, applying the mean value theorem to $\partial_{x}^{4}u$ and using \eqref{cd_u_2'}, we obtain 
\begin{equation}
\left| \delta_{k}^{+}\left( \delta_{k}^{\langle 2 \rangle}\tilde{u}_{k}^{ \left( n + \frac{1}{2}\right) } - \partial_{x}^{2}u_{k}^{\left( n + \frac{1}{2}\right)} \right) \right| \leq \frac{ (\Delta x)^{2} }{12} M_{5,0}(u) \quad (k=0,\ldots , K-1). \label{est1}
\end{equation}
In the same manner, for $k=0,\ldots, K$, applying the Taylor theorem to $\partial_{t}^{2}\tilde{u}$ and using \eqref{util}, there exists $\theta_{4} \in (0,1)$ such that 
\begin{equation}
\delta_{k}^{\langle 2 \rangle}\!\!\left(\! \partial_{t}^{2}\tilde{u}_{k}^{ \left( n + \frac{1 \pm \theta_{2}}{2}\right) } \!\right) 
	= \begin{cases}
		\displaystyle \partial_{x}^{2}\partial_{t}^{2}u_{\theta_{4}}^{ \left( n + \frac{1 \pm \theta_{2}}{2}\right) } \! - \theta_{4}\Delta x\partial_{t}^{2}\partial_{x}^{3}u_{0}^{ \left( n + \frac{1 \pm \theta_{2}}{2}\right) } & (k \! = \! 0), \\[5pt]
		\displaystyle \frac{1}{2}\left(\! \partial_{x}^{2}\partial_{t}^{2}u_{k + \theta_{4}}^{ \left( n + \frac{1 \pm \theta_{2}}{2}\right) } \! + \! \partial_{x}^{2}\partial_{t}^{2}u_{k - \theta_{4}}^{ \left( n + \frac{1 \pm \theta_{2}}{2}\right) } \!\right) & (k \! = \! 1,\ldots, K \! - \! 1), \\[6pt]
		\displaystyle \partial_{x}^{2}\partial_{t}^{2}u_{K - \theta_{4}}^{ \left( n + \frac{1 \pm \theta_{2}}{2}\right) } \! + \theta_{4}\Delta x\partial_{t}^{2}\partial_{x}^{3}u_{K}^{ \left( n + \frac{1 \pm \theta_{2}}{2}\right) } & (k \! = \! K). \\
	\end{cases} \label{ut2_lap}
\end{equation}
Applying the mean value theorem to $\partial_{x}^{2}\partial_{t}^{2}u$ and using \eqref{ut2_lap}, we obtain 
\begin{gather}
\left| \delta_{k}^{+}\! \left\{ \delta_{k}^{\langle 2 \rangle}\!\left( \partial_{t}^{2}\tilde{u}_{k}^{ \left( n + \frac{ 1 \pm \theta_{2} }{2}\right) } \right) \right\} \right| \leq M_{3,2}(u) \quad (k=1,\ldots, K-2), \label{est2}\\
\left| \delta_{k}^{+}\! \left\{ \delta_{k}^{\langle 2 \rangle}\!\left( \partial_{t}^{2}\tilde{u}_{k}^{ \left( n + \frac{ 1 \pm \theta_{2} }{2}\right) } \right) \right\} \right| \leq M_{3,2}(u) + \theta_{4}M_{3,2}(u) \leq 2M_{3,2}(u) \quad (k=0,K-1). \label{est3}
\end{gather}
Hence, from \eqref{lap_u_eq}, \eqref{est1}, \eqref{est2}, and \eqref{est3}, we conclude that 
\begin{align*}
& \left| \delta_{k}^{+}\left\{ \partial_{x}^{2}u_{k}^{\left( n + \frac{1}{2}\right)} - \delta_{k}^{\langle 2 \rangle}\left(\frac{\tilde{u}_{k}^{(n+1)} + \tilde{u}_{k}^{(n)}}{2}\right) \right\} \right| \\
	& \leq \left| \delta_{k}^{+}\!\left(\! \partial_{x}^{2}u_{k}^{\left( n + \frac{1}{2}\right)} - \delta_{k}^{\langle 2 \rangle}\tilde{u}_{k}^{ \left( n + \frac{1}{2}\right) } \!\right) \right| + \frac{(\Delta t)^{2}}{16}\!\left( \left| \delta_{k}^{+}\!\left\{ \delta_{k}^{\langle 2 \rangle}\!\!\left( \partial_{t}^{2}\tilde{u}_{k}^{ \left( n + \frac{ 1 + \theta_{2} }{2}\right) } \right) \right\} \right| + \left| \delta_{k}^{+}\!\left\{ \delta_{k}^{\langle 2 \rangle}\!\!\left( \partial_{t}^{2}\tilde{u}_{k}^{ \left( n + \frac{ 1 - \theta_{2} }{2}\right) } \!\right) \right\} \right| \right) \\
	& \leq \frac{(\Delta x)^{2}}{12}M_{5,0}(u) + \frac{(\Delta t)^{2}}{4}M_{3,2}(u) 
\end{align*}
for $k=0,\ldots, K-1$. \hfill $\Box$\\
\end{proof}

\section*{Appendix D}
In this appendix D, we present the computation example under the Neumann boundary condition in order to compare the long-time behavior of solutions. 
Note that we use the same notations as in Section 1 and Section 2. 
\subsubsection*{Numerical results for the Neumann boundary condition} 
As stated in Section 6, in order to verify that the difference in the long-time behavior of the solutions occurs, we present the computation example for \eqref{Eq_i}--\eqref{Eq_ii} with the following homogeneous Neumann boundary conditions: \vspace{-1mm} 
\begin{align}
\begin{cases}
\left. \partial_{x}u(x,t)\right|_{x=0} = \left. \partial_{x}u(x,t)\right|_{x=L} = 0 & {\rm in} \ (0,\infty), \\
\left. \partial_{x}p(x,t) \right|_{x=0} = \left. \partial_{x}p(x,t) \right|_{x=L} = 0 & {\rm in} \ (0,\infty), \label{NBC}
\end{cases}
\end{align}
in the same setting as Computation example 3 in Section 6. 
We remark that the solution of \eqref{Eq_i}--\eqref{Eq_ii} with \eqref{NBC} also satisfies the conservative property \eqref{MC} and the dissipative property. 
However, in this case, the dissipative property is slightly different from \eqref{D}. 
More precisely, the solution of \eqref{Eq_i}--\eqref{Eq_ii} with \eqref{NBC} satisfies the following dissipative property: \vspace{-1.5mm} 
\begin{equation*}
\frac{d}{dt}J(u(t)) = - \int_{0}^{L}|\partial_{x}p(x,t)|^{2}dx \leq 0. \vspace{-1.5mm} 
\end{equation*}
Since there are no computation examples in the same setting as Computation example 3 in previous studies, we carry out the computation example by the following structure-preserving scheme.  
By using DVDM (see \cite{B}), the scheme is derived as follows: \vspace{-1mm}
\begin{gather}
\frac{U_ {k}^{(n+1)} - U_ {k}^{(n)}}{\Delta t} = \delta _ {k}^{\langle 2 \rangle}P_ {k}^{(n)} \quad (k=0,\ldots, K,\ n=0,1, \ldots), \nonumber\\[-1pt]
P_ {k}^{(n)} = -\gamma \delta_{k}^{\langle 2 \rangle}\!\left( \frac{U_{k}^{(n+1)} + U_{k}^{(n)}}{2}\right) + \frac{dF}{d(U_{k}^{(n+1)},U_{k}^{(n)})} \quad (k=0,\ldots, K,\ n=0,1, \ldots), \nonumber\\[-1pt]
\delta _ {k}^{\langle 1 \rangle}U_ {k}^{(n)} = 0 \quad (k=0,K,\ n=0,1, \ldots), \label{UNBC}\\[-1pt] 
\delta _ {k}^{\langle 1 \rangle}P_ {k}^{(n)} = 0 \quad (k=0,K,\ n=0,1, \ldots). \nonumber
\end{gather}~\vspace{-5.5mm}\\
Figure \ref{fig:NBC} shows the time development of the solution obtained by the above scheme. \vspace{-3.5mm}
\begin{figure}[H]
  \begin{center}
   \includegraphics[width=78mm]{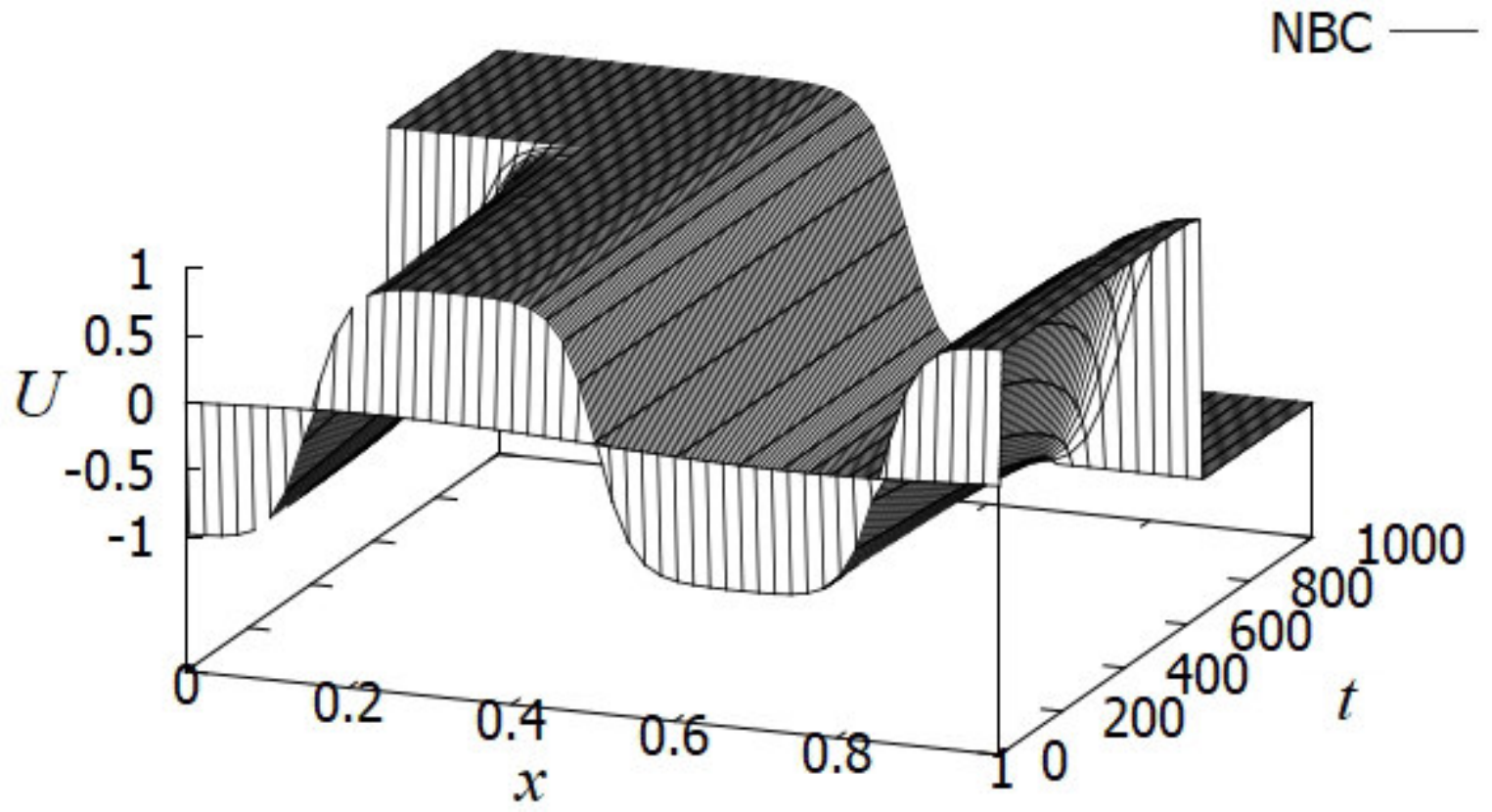} \vspace{-6mm} 
  \end{center}
  \caption{Numerical solution to \eqref{Eq_i}--\eqref{Eq_ii} with \eqref{NBC} obtained by the discrete variational derivative scheme.}
  \label{fig:NBC}
\end{figure}
\noindent
As can be seen from Figure \ref{fig:Our5} and Figure \ref{fig:NBC}, the solution to \eqref{Eq_i}--\eqref{Eq_ii} with \eqref{NBC} arrives at the different state from that to \eqref{Eq_i}--\eqref{Eq_ii} with \eqref{BCii} and \eqref{DBC'}. 
Thus, the results assure that the difference in the long-time behavior of the solutions occurs. 

Next, Figure \ref{fig:NBC_mass} shows the time development of $M_{\rm d}(\bm{U}^{(n)})$ obtained by the above scheme. 
Figure \ref{fig:NBC_energy} shows the time development of $A_{\rm d}^{(n)} - \bar{J}_{\rm d}(\bm{U}^{(0)})$ obtained by the above scheme, where \vspace{-2mm}
\begin{gather*}
A_{\rm d}^{(n)} := \bar{J}_{\rm d}( \bm{U}^{(n)} )  +  \sum_{l=0}^{n-1}\sum _ {k=0}^{K}{}^{\prime \prime}\frac{ \left| \delta_{k}^{+} P_{k}^{( l )} \right|^{2} + \left| \delta_{k}^{-} P_{k}^{( l )} \right|^{2} }{2}\Delta x\Delta t \quad (n=1,2,\ldots), \\[-2pt]
\bar{J}_{\rm d}( \bm{U}^{(n)} ) := \sum _ {k=0}^{K}{}^{\prime \prime}\left\{ \frac{\gamma}{2}\frac{ \left| \delta_{k}^{+} U_{k}^{( n )} \right|^{2} + \left| \delta_{k}^{-} U_{k}^{( n )} \right|^{2} }{2} + F(U_{k}^{(n)}) \right\}\Delta x \quad (n=0,1,\ldots). 
\end{gather*}~\vspace{-6mm}\\
\begin{Remarkap} 
For any $\{f_{k}\}_{k=-1}^{K+1} \in \mathbb{R}^{K+3}$ satisfying the discrete homogeneous Neumann boundary condition $\delta_{k}^{\langle 1 \rangle}f_{k} = 0 \ (k=0,K)$, 
the following equality holds: \vspace{-1.5mm}
\begin{equation*}
\sum _ {k=0}^{K}{}^{\prime \prime}\frac{ |\delta_{k}^{+} f_{k}|^{2} + | \delta_{k}^{-} f_{k} |^{2} }{2}\Delta x = \sum _ {k=0}^{K-1} |\delta_{k}^{+} f_{k}|^{2} \Delta x. \vspace{-1.5mm} 
\end{equation*}
From this equality and \eqref{UNBC}, we obtain $\bar{J}_{\rm d}( \bm{U}^{(n)} ) = J_{\rm d}( \bm{U}^{(n)} ) \ (n=0,1, \ldots)$. \vspace{-3.5mm} 
\end{Remarkap}
\begin{figure}[H]
 \begin{minipage}{0.495\hsize}
  \begin{center}
   \includegraphics[width=78mm]{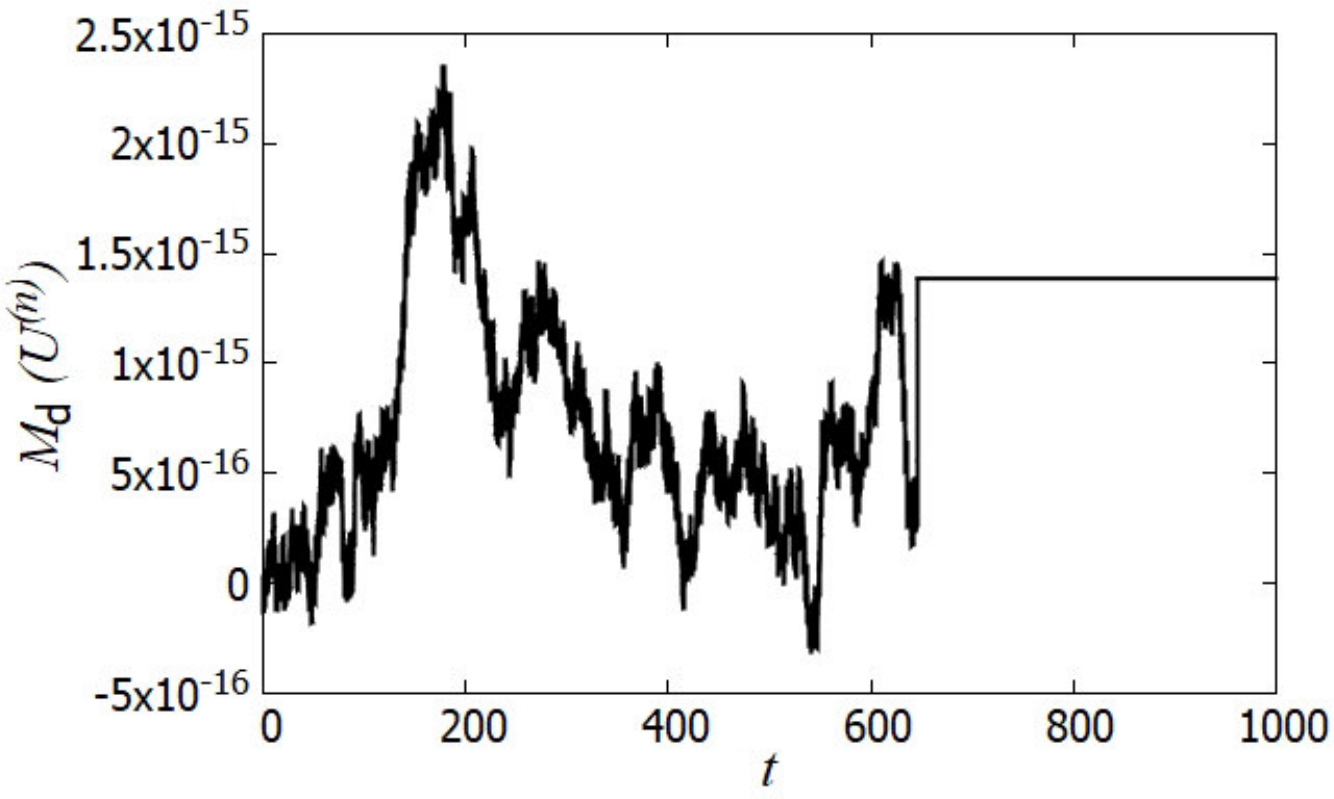} \vspace{-6mm} 
  \end{center}
  \caption{Time development of $M_{\rm d}(\bm{U}^{(n)})$ obtained by the discrete variational derivative scheme. $M_{\rm d}(\bm{U}^{(n)})$ does not change by about $14$ orders of magnitude.}
  \label{fig:NBC_mass}
 \end{minipage}
 \begin{minipage}{0.495\hsize}
  \begin{center}
   \includegraphics[width=78mm]{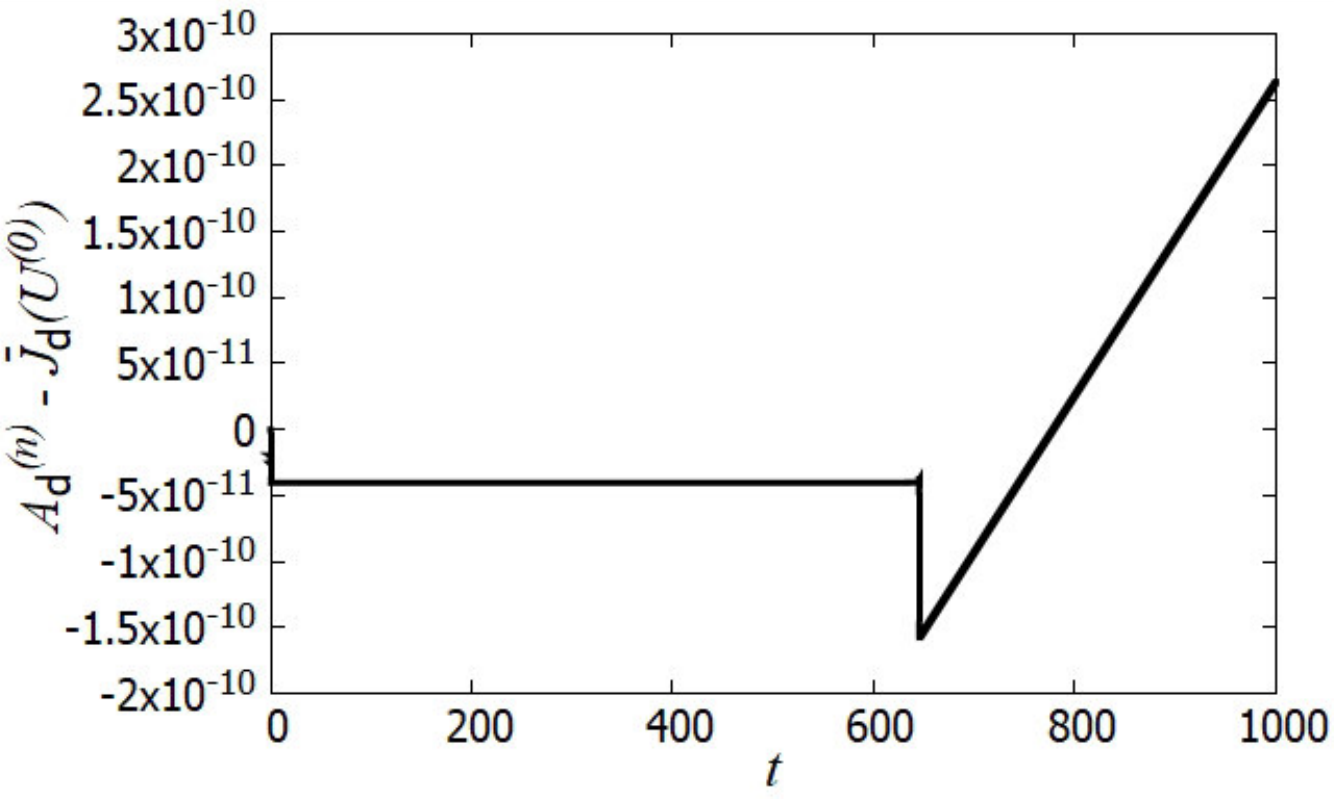} \vspace{-6mm} 
  \end{center}
  \caption{Time development of $A_{\rm d}^{(n)} - \bar{J}_{\rm d}(\bm{U}^{(0)})$ obtained by the discrete variational derivative scheme. $A_{\rm d}^{(n)}$ does not change by about $9$ orders of magnitude.}
  \label{fig:NBC_energy}
 \end{minipage}
\end{figure}
\noindent 
These graphs show that the quantities $M_{\rm d}(\bm{U}^{(n)})$ and $A_{\rm d}^{(n)}$ are conserved numerically. 
More precisely, $M_{\rm d}(\bm{U}^{(n)})$ does not change by about $14$ orders of magnitude, and $A_{\rm d}^{(n)}$ does not change by about $9$ orders of magnitude.

\end{document}